\documentclass[reqno]{amsart}
\usepackage{amsmath, amssymb}
 \usepackage{mathabx}
\usepackage{enumerate, xspace}
\usepackage[dvips]{graphicx}
\numberwithin{equation}{section}
\newtheorem{thm}{Theorem}[section]
\newtheorem{lem}[thm]{Lemma}
\newtheorem{cor}[thm]{Corollary}

\newtheorem{prop}[thm]{Proposition}

\newtheorem{rem}[thm]{Remark}

\newcommand\N{{\mathbb N}}

\newcommand\R{{\mathbb R}}
\newcommand\Z{{\mathbb Z}}
 \newcommand\T {{\mathbb T}}
\newcommand\1{{1\kern-.25em\hbox{\rm I}}}
\newcommand\eu{{1\kern-.25em\hbox{\sm I}}}
\newcommand{\dha}{{d}}

\newcommand\HH{{\mathcal H}}
\newcommand\A{{\mathcal A}}
\newcommand\B{{\mathcal B}}
\newcommand\CC{{\mathcal C}}
\newcommand\D{{\mathcal D}}

\newcommand\FF{{\mathcal F}}
\newcommand\GG{{\mathcal G}}

\newcommand\LL{{\mathcal L}}

\newcommand\PP{{\mathcal P}}

\newcommand\TT{{\mathcal T}}
\newcommand\NN{{\mathcal N}}

\newcommand\XX{{\mathcal X}}

\newcommand\RR{{\mathcal R}}

   \newcommand \llangle{ {\langle\hskip-2pt\langle}}
 \newcommand \rrangle{ {\rangle\hskip-2pt\rangle}}
 
  \newcommand\z{\zeta}
 
 \newcommand\e{\epsilon}

\newcommand\g{\gamma}

\newcommand\s{\sigma}
\let\a=\alpha
\let\b=\beta
\let\d=\delta
\let \l = \lambda
\let \t = \tau

\newcommand\G{\Gamma}
\newcommand\Om{\Omega}

\newcommand{\nada}[1]{}

\begin{document}
 \title[Spectral properties]
 {Spectral properties of  integral operators   in problems of  interface dynamics}
 \author{Enza Orlandi}\thanks{E.O. supported by   Universit\'a di ROMA TRE}
\address{Enza Orlandi, 
Dipartimento di Matematica e Fisica\\
Universit\`a  di Roma Tre\\
 L.go S.Murialdo 1, 00156 Roma, Italy. }
\email{{\tt orlandi@mat.uniroma3.it}}
\date{\today}
\begin{abstract}
 We  consider    a family of integral operators  which appears when analyzing  layered equilibria and front dynamics of a  phase kinetics equation with a conservation law.      We   study the spectra  of these operators in $L^2$  and derive
 a lower bound for the    associated  quadratic forms   in terms of the  $H^{-1}$ norm.   
 \end{abstract}

\maketitle

\section{Introduction} 
The purpose of this paper is to   derive spectral estimates for   a family of integral operators
which appear when analyzing  layered equilibria and front dynamics of a  phase kinetics equation with a conservation law.   We start
by recalling some background.
 
Consider  in  the torus   $\T^d$  the nonlocal and nonlinear  evolution equation   
\begin {equation} \label  {1.0} \frac {\partial}{\partial t}  m(x,t)    =  \nabla\cdot\bigl(
\nabla  m(x,t)  - 
  {\beta (1-m(x,t)^2) (J\star
\nabla m )(x,t)} \bigr)  \end {equation}
where $ \beta>1 $,  $ \star$ denotes convolution and $J$ is a smooth,
spherically symmetric probability density 
with compact support.
 This equation first appeared
in the literature in a paper  \cite{LOP}  on the dynamics of Ising systems
with a long--range interaction and so--called ``Kawasaki'' or ``exchange''
dynamics and  later it was rigorously
derived  in \cite{GL1}.   In this physical context,
$m(x,t) \in [-1,1]$ is the magnetization density at $x$ at time $t$, viewed on the length scale of the
interaction, and $\beta$ is the inverse temperature. 
This introduction is not the place to fully explain the physical 
origins of the equation \eqref {1.0}, and
familiarity with them is not needed to understand our results or their
proofs. We refer to the previous quoted paper   for more physical insight. 
      The equation \eqref {1.0}   can be written  in a
gradient flow form. To do  this, we introduce the free energy functional
$ \FF(m)$:
\begin {equation} \label {freeen} \FF (m) = \int_{\T^d}  [V(m(x)) - V(m_\beta)]\dha     x+ {1\over 4} 
\int_{\T^d}\int_{\T^d}  J(x-y) [m(x)-m(y)]^2\dha     x\dha     y,   \end {equation}   
where  $ V(m)$ is
\begin {equation} \label {littlef} V(m) = -  \frac 12  m^2 + \frac 1 \beta \left [ \left  (\frac {1+   m} 2 \right )\ln \left (\frac {1+   m} 2 \right)+   \left (\frac {1 -  m} 2 \right )\ln \left (\frac {1-   m} 2 \right) \right ].
   \end {equation} 
 For $\beta >1$, this  potential function $V$  is a symmetric double well potential on
$[-1,1]$.  We denote the positive minimizer of $V$ on $[-1,1]$ by
$m_\beta$.
It is easy to see that $m_\beta$ is the positive solution of the equation
\begin {equation} \label {(fixedpt) } m_\beta = \tanh (\beta m_\beta).
 \end {equation}
Then  equation \eqref {1.0} can be written as
\begin {equation} \label {gradflow}{\partial\over \partial t}m    =  \nabla\cdot\biggl( 
\sigma(m)\nabla\biggl({\delta \FF \over \delta  m } \biggr)\biggr) 
  \end {equation}
where the {\it mobility} $\sigma(m)$ is given by
\begin {equation} \label {mobility} \sigma(m) = \beta (1 -m^2).   \end {equation}
Formally one derives
\begin {equation} \label {dissipate}{{\rm d}\over\dha     t}\FF (m(t))   =  -\int\biggl|
  \nabla \left ({\delta \FF \over \delta  m } \right) \biggr|^2
\sigma(m(t)){\rm d}x  \end {equation}
thus $ \FF $ is a Lyapunov function for \eqref {1.0}. 
This suggests that the free energy should want to tend locally to one of the
two minimizing values, $\pm m_\beta$, and that the interface between a
region of $+m_\beta$ magnetization and a region of $-m_\beta$ magnetization
should have a ``profile'' -- in the direction orthogonal to the interface --
that makes the transition from one local equilibrium to the other in a
way that minimizes the free energy. This is indeed the case, as it has been shown in    \cite{CCO1},
\cite{CCO2} in dimension $d=1$ and in  \cite{CO}  in dimension $d \le 3$. Also clearly, the minimizers
of the free energy $\pm m_\beta$ represent the 
``pure phases'' of the system.
However, unless the initial data
$m_0$ happens to satisfy $\int_\T  m_0(x){\rm d}x = \pm m_\b |\T|$, these 
``pure phases''  
cannot be reached  because of the
conservation law. Instead, what will eventually be produced is a region in which 
$m(x) \approx +m_\b$, with 
$m(x) \approx -m_\b$ in its complement, and with a smooth transition 
across 
its boundary. This is referred to a {\it phase segregation}, 
and the boundary  is the {\it interface} between the two phases. 
If we ``stand far enough back'' from $\T$, all we see is the 
interface, and we do not see any structure
across the interface -- the structure now being on an invisibly small 
scale. The evolution of
$m$ under the \eqref {1.0}, or another   evolution 
equation of this type,  
drives an evolution of the interface.
To see any evolution of the interface, one must wait a long time. 
More specifically, let $\lambda$
be a small parameter, and introduce new variables $\tau$ and $\xi$ 
through
$$\tau = \lambda^3t \qquad{\rm and}\qquad \xi = \lambda x\ .$$
Then of course 
$${\partial \over \partial t} = \lambda^3 {\partial \over \partial 
\tau}
\qquad{\rm and}\qquad {\partial \over \partial x} = \lambda  
{\partial \over \partial \xi}\ .$$
Hence if $m(x,t)$ is a solution of \eqref {1.0}, and 
we define
$m^\lambda(\xi,\tau)$ by $m^\lambda(\xi,\tau) = m(x(\xi),t(\tau))$,
we obtain
\begin {equation} \label {1.1}\begin {split}  & \frac {\partial}{\partial \t}  m^\l (\xi,\t)    = 
   \frac 1 \l  \nabla \cdot \left ( \s (m^\l)  \nabla
\left[  \frac 1 \b \hbox {arctanh} \, 
m^\l 
   - (J^\l \star   m^\l )  \right]    \right ) (\xi, \tau)  
         \end {split} \end {equation}
 where we denoted  $ J^\l (\xi)= \l^{-d} J(\l^{-1}
\xi)$.   
One should just bear in mind that now we are looking at the evolution 
over a {\it very} 
long time scale when $\lambda$ is small.
 
One might hope that for small values of
$\lambda$, {\it all information about the evolution on $m^\l $ is 
contained in the evolution of the interface
$\G_{\tau}$}.  
This is indeed the case.  
The sharp interface limit of  equations tipyfied by \eqref {1.1}    has 
been  investigated by
Giacomin and  Lebowitz, see  \cite{GL2}, where it  is  heuristically proven that the limit motion is driven by  the Mullins--Sekerka  flow.  They applied asymptotic analysis    in the same spirit of the 
heuristic work of  Pego \cite{P}  who derived the sharp interfaces limit for the Cahn-Hilliard
equation.

  The heuristic analysis of Pego   has been  rigorously proven  for
the Cahn-Hilliard equation by  \cite{ABC}  in the 1994 and   later on,      applying  a different method,  in  \cite{CCO3}. 
In both the papers the proof  is based on two steps.
The first step is to construct   approximate solutions to   the Cahn-Hilliard
equation, which  are,  in the limit, close to the  the Mullins--Sekerka  flow.
 The second step is to show that the family of approximate solutions is indeed 
suitably close  
to the solution of the Cahn-Hilliard equation.  To show this  second step     spectral estimates  are needed. These were proven  in  
 \cite {AF}, in dimensions $d=2$,  and \cite {Chen} in any dimensions. 

 In \cite{ABC}   the family of the approximate solutions to    Cahn-Hilliard
equations  are constructed  by  matched asymptotic  expansions. 
In \cite{CCO3},    the approximate solutions
are constructed by an alternative method.  The method based on the Hilbert expansion used in kinetic theory besides its relative
simplicity, it leads to calculable higher order corrections to the interface motion.  More important in this context is that  the above approach allows   to construct approximate solutions to the non local evolution equation \eqref {1.1}.   
 
 Hence,   to prove rigorously the heuristic analysis done by Giacomin and Lebowitz in  \cite{GL2},
 one  might  first construct  approximate solutions and this  can be done applying the same method as in   \cite{CCO3}.  Then   one needs to derive spectral estimates   to show that  
 the constructed  approximate solutions  are    close in  some convenient norm  to the solution of  \eqref {1.1}.    
 In this paper we prove     such  spectral estimates.   
We  set the problem in a bounded domain $\Om \subset \R^2$.
The restriction at  dimension $d=2$ is purely technical. In $d=2$ we still can  use global set of coordinates to represent the operator when close to the interface, $\G$.  Namely in such a case any simple, smooth, closed,  one  dimensional curve can be mapped  into a one dimensional circle.  This allows to use Fourier series by going to the universal cover.
In dimension $d \ge 3$ this would   not be possible and one needs to deal with    local coordinates and, possibly, Fourier transforms. 
 Indeed, we believe that the result
holds in any dimension  and we leave this problem to further investigation. 
  
We consider a family $m_A^\l$ of smooth functions
which as $\l \to 0$ approach a step function with values $\pm m_\b$ which is discontinuous 
along a smooth curve $\G \subset \Om$.   This family  consists  of approximate solutions  to \eqref {1.1}, which can be constructed  as in \cite {CCO3}.  
The functions $m_A^\l$   have a very specific behaviour  in the direction orthogonal to $\G$, as $\l \to 0$.  The specific form of $m_A^\l$ is given in \eqref   {2.2}.  To  simplify  notations from now on  we drop the index $\l$.
   The linear equation  obtaining  linearising    \eqref {1.1}  at  $m_A$,  is~ \footnote { We  are  now considering    \eqref {1.1} in  $\Om \Subset  \R^2$ with  non flux boundary conditions and with   the convolution operator ``restricted" in $\Om$.} 
        \begin {equation} \label {E.10}  
\frac {\partial}{\partial \t} v (\xi,\t)    =  \frac 1 \l  \nabla \cdot \left ( \s (m_A)  \nabla \cdot 
\left [  \frac v  {\s (m_A)}  
 - (J^\l \star_\Om  v )  \right ]    \right ) (\xi, \tau), \quad \xi \in \Om,  
          \end {equation}
where
$ (J^{\l} \star_\Om v)
(\xi) = \int_\Om  J^\l (\xi-\xi') v(\xi') d\xi'$,   $v = m^\l-m_A$ and  we assume that $ v(\xi, 0)=0$, $\xi \in \Om$.  By the conservation law,    $ \int_\Om  d\xi v (\xi, \tau) = 0$ for all $\tau\ge 0$. To simplify the explanation let us first pretend that   $\s(m_A) =1$  in  $\nabla \cdot   \s (m_A)  \nabla$.
Let us then denote  \begin {equation} \label  {8.10a}  \frac 1 \l ( A^\l_{m_A} v)  
(\xi)=  \frac 1 \l \left \{ \frac {v (\xi)} { \s(m_A(\xi))}- (J^{\l} \star_\Om v)
(\xi)  \right \},
 \quad \xi \in \Om.    \end {equation}
Accordingly, to  lower bound  the spectrum of the linear operator on the right hand side of \eqref {E.10} in $H^{-1} (\Om)$
it suffices to  show that
   \begin {equation} \label  {VL4}  \frac 1 \l  \langle\hskip-2pt\langle  v,   A^\l_{m_A} v  \rrangle \ge - C  \|v\|^2_{H^{-1}(\Om)},  \end {equation} 
   where   $ \llangle v, g \rrangle = \int_\Om v(\xi) g(\xi) d\xi$ and $C>0$ does not depend on $\l$.   In the general   case, when  $\s(m_A) \neq 1$,  one can argue in the same manner by using a weighted $H^{-1}$ norm, the weigh   being $\s(m_A)$. Because   $0<a\le \s(m_A) \le \b$ the  weighted $H^{-1} $ norm is equivalent to the  $H^{-1}$.  So we  will  prove   \eqref  {VL4}, with $v \in  H^{-1}(\Om)$.
   
As in  Alikakos and Fusco, \cite {AF}, or in  X. Chen,  \cite {Chen}, we     first  prove a lower bound of the spectrum  of the operator in \eqref  {8.10a}  in the $L^2$ norm, see Theorem \ref {82}, i.e
\begin {equation} \label  {VL2} \llangle v,   A^\l_{m_A} v \rrangle \ge - C  \l^2  \|v\|^2_{L^2(\Om)}. 
 \end {equation}
Then,    when
\begin {equation} \label  {VL3} \llangle v,   A^\l_{m_A} v \rrangle \le 0 \end {equation}we show that
\begin {equation} \label  {VL3a}   \l  \|v\|^2_{L^2(\Om)} \le  \|v\|^2_{H^{-1}(\Om)}  \end {equation}
 see Theorem \ref {86}.   In this way we   prove  the    lower bound  \eqref {VL4}.

\subsection { Sketch   of the proof.}  

 The proof of \eqref {VL4}   presents some similarities with   the one given  in    \cite {AF}  and  in  \cite {Chen}  for the Cahn-Hilliard  (C-H) case.
Nevertheless,  the implementation of  each single step  requires different technique due to the non locality of the operator.   In particular, we cannot follow the  method  in   \cite {AF}  and  in  \cite {Chen}, as we cannot split the integral  kernel  in \eqref {8.10a} into a tangential and in normal part, whereas expressions like $\|\nabla  v \|^2 $  split in such a way naturally.

  Now we   sommarize    how we proceed in proving \eqref {VL4}.
 We consider  a neighborhood  of $\G$,   $\NN(\G)$.
We map  the curve $\G$ in a circle $T$ having perimeter equal to the length of the curve $L$
 and each point  $\xi \in \NN (\G)$ is mapped  to $(s,r) \in   \TT= T \times [-d_0, d_0]$ in a diffeomorphic way.  We denote   by   $  \a (s,r) $  the    Jacobian of the map.
 We take a subset  $\NN_1 (\G) \subset \NN (\G)$  and  we assume that 
 \begin {equation} \label {sc1a}   \inf_{ \xi \in \Om \setminus   \NN_1(\G)}  \frac 1 {\sigma (m_A(\xi))} >C^* >1.  \end {equation}
To take advantage of \eqref  {sc1a} we    split the quadratic form  \eqref {VL2} in two integrals: one over 
 $\NN_1(\G)$ the other in $ \Om \setminus \NN_1(\G)$.
   To this end we introduce  the  indicator function of the set  $\NN_1 (\G)$, $ \eta_1 (\xi)=1$ when $\xi \in  \NN_1(\G)$,  $ \eta_1 (\xi)=0$  in  $ \Om \setminus \NN_1 (\G)$ and $\eta_2(\xi)= 1- \eta_1(\xi)$.
 Because of  the non locality of the operator, taking into account that  $\eta_2 (\xi)\eta_1 (\xi)=0$ for $\xi \in \Om$ and the symmetry of $J^{\l} (\cdot)$ we obtain that 
\begin {equation} \label {gch2b}  \begin {split}     \int_{ \Om}  ( A^\l_{m_A} v)  (\xi) v(\xi) \dha    \xi &   = \int_{ \Om}  ( A^\l_{m_A} \eta_1 v)  (\xi)  \eta_1 (\xi) v(\xi) \dha    \xi   \cr & +
 \int_{ \Om}  ( A^\l_{m_A} \eta_2 v)  (\xi)  \eta_2 (\xi) v(\xi) \dha    \xi   \cr &  -
2 \int_{ \Om} \dha    \xi  \eta_1 (\xi) v  (\xi) (J^{\l} \star  \eta_2 v)(\xi).
\end {split}
\end {equation} 
Condition \eqref  {sc1a} implies  $$ \int_{ \Om}  ( A^\l_{m_A} \eta_2 v)  (\xi)  \eta_2 (\xi) v(\xi) \dha    \xi   \ge (C^*-1) \| \eta_2   v\|^2_{L^2 (\Om)} >0.$$ 
Even  if we    show that 
  \begin {equation} \label {D1a}   \int_{\Om }  (A^\l_{m_A}  \eta_1v ) (\xi)  \eta_1 v(\xi)\dha    \xi  \ge - C \l^2 \int_\Om v^2 (\xi) d\xi,  \end {equation}
  the last term of \eqref  {gch2b}  might create problems for getting estimate \eqref {VL2}. 
  The first task in proving estimate \eqref {VL2} is therefore to  show that there exists $\NN_1 (\G)$ so that 
 \begin {equation} \label {rm.10}  \left |  \int_{ \Om} \dha    \xi  \eta_1 (\xi) v  (\xi) (J^{\l} \star  \eta_2 v)(\xi) \right |  \le  C\l^2   \|   v\|^2_{L^2 (\Om)}. \end {equation}
 This is proven at the beginning of the proof of Theorem \ref {82}. 
      The following step is    to  show \eqref {D1a}. 
  We    write the   quadratic form    on the left hand side of  \eqref {D1a}   in local  variables.  Notice that the integrals are on the set  $\NN (\G)$.
   In these local variables the  convolution  operator $ (J^{\l} \star  \eta_1v)(\xi)$
 becomes equal, up to  order $\l^2$,   to  an operator  which is not    a convolution 
 anymore     but    it  is  still self-adjoint with respect to a weighed   Lebesgue measure on $\TT$, see Lemma \ref {M0}. 
  In this way we show, see Lemma \ref {F2},  that
  setting  $$\hat v  (s,r) = \sqrt { \a (s,r) } v (s,r)$$
\begin {equation} \label {D10a}   \int_{\Om }  (A^\l_{m_A}  \eta_1v ) (\xi)  \eta_1 v(\xi)\dha    \xi  \ge 
\int_{\TT} ds dr (L^\l \hat v) (s,r)\hat v(s,r) - C \l^2 \|v\|^2_{L^2(\Om)}.   \end {equation}
It turns out that the operator $L^\l $ is conjugate to the operator $ \1 - \PP^\l$, where $  \PP^\l$ is an integral operator, positivity improving.   Therefore $L^\l $  and $ \1 - \PP^\l$ have the same spectrum. By the  Perron -Frobenius Theorem we know that  the  principal eigenfunction  is  point wise  positive.
To  estimate the  principal eigenvalue we still need other ingredients. 
We show that  
   when the operator $L^\l$  acts on  functions depending only on the signed distance $r$ from $\G$, becomes equal, up to order $\l^2$,       to  a one dimensional operator 
$L_1^{\l,s}$,       $s \in T$.   The knowledge of the spectrum of  $L_1^{\l,s}$ (obtained using the results of \cite{O1}) together with the properties of  the  eigenfunctions associated to the principal eigenvalues of  $L_1^{\l,s}$ and    $\PP^\l$  allows to upper and lower bound the principal eigenvalue $\mu_0$ of 
$L^\l$.   In this way we   show that  $ -C \l^2 \le \mu_0  \le   C \l^2$. Hence  \eqref {D1a}  follows.
      To prove the  $H^{-1} (\Om)$ lower  bound, see \eqref {VL4},  we cannot proceed as  in  \cite {AF} and in  \cite {Chen}. 
 For the non local operator  \eqref {8.10a} a  bound on its quadratic form does not
imply any bound  on the  gradient of the  function.       
Nevertheless,  we are still able to  show that  when 
 \begin {equation} \label {D11}   \int_{\Om }  (A^\l_{m_A}  \eta_1v ) (\xi)  \eta_1 v(\xi)\dha    \xi \le \l^2 \|v\|^2_{L^2(\Om)}  \end {equation}
then $v$ in local variables  enjoys the following
decomposition, see   Theorem  \ref {g10},
\begin {equation} \label {D12}   v (s, r)  =  Z(s)\frac 1 { \sqrt \l} \psi^0_0 (\frac r \l) + v^R (s,r)
 \end {equation}
 where  $ \|Z \|_{L^2(T)} \simeq 1$,  $ \|\nabla Z\|_{L^2(T)}  \le C$,  $\| v^R\|^2_{L^2(\NN(\G))} \le C\l^2$   and $ \psi^0_0  (\cdot)$ is a strictly positive, even, smooth  function, exponential decreasing.
In such a way we obtain a decomposition similar to the one in \cite {AF} and    \cite {Chen} and we can   proceed as in  \cite {Chen} to prove \eqref {VL3a}.   To prove the decomposition \eqref {D12}  a more complete  analysis of the spectrum of  the operator $L^\l$ is needed.   
Let  $ \{ \mu_k\}_{k\in \N}$  be the eigenvalues of   $L^\l$ in $L^2(\TT)$. Essentially we prove that   
 $ -C\l^2 + c_0k^2 \l^2  \le \mu_k \le C \l^2 + c_1k^2 \l^2$,  for $k \le \frac {h_0} \l$, and,   for  $k > \frac {h_0} \l$, $\mu_k \ge \nu>0$,    where $C$, $c_0$,  $c_1$, $h_0$ and $\nu$    are positive,  real numbers independent on $\l$.
In addition, when  $k \le \frac {h_0} \l$, we also need, and establish,  a  precise  knowledge of the  shape of the associated  eigenfunctions.
          To obtain such informations on the spectrum of the operator $ L^\l$ we make a judicious use of Fourier analysis.

      \bigskip
\noindent {\bf Acknowledgements} I  am indebted to Sao Carvalho and Eric Carlen for   discussions on this  and related problems, discussions which    started many years ago  when visiting Georgia Tech  and  never stopped.   I    also benefited  from  suggestions   by    Carlangelo Liverani.
     \section { Notations and Results} 

Let  $\Omega$  be  a bounded domain in $\R^2$ with sufficiently smooth boundary  and let   
  $\G \subset \Om$  be a simple,  smooth,   closed curve,   boundary of an open set   $ \Omega^-$.   We denote    by $ \Omega^+ = \Omega \setminus (\Omega^-\cup \G).$  We  denote by $C$ a constant which might change from one occurrence to the other,  independent on $\l$.    

\subsection { The interaction $J^\l$. }
Let  $J$   be 
a smooth spherically symmetric, translation invariant, 
probability density on $\R^2$  with compact support, $ \{\xi \in \R^2: |\xi| \le 1\}$.
We assume  $J \in C^1(\R^2)$.
We say  that $\xi \in \R^2$ and $\xi' \in \R^2 $ interact  each other if    $J(\xi-\xi')>0$. 
  For $  \xi = (\xi_1, \xi_2) $    we denote    \begin {equation}  \label{1d} \bar J (\xi_1) = \int_{\R}   J(\xi) d\xi_2, \end {equation}
 and for     $ \l \in (0,1]$  
     $$ J^\l (\xi) =  \frac 1 {\l^2}  J( \frac { \xi}\l),  \qquad  \bar J^\l (\xi_1) = \frac 1 \l  \bar J (\frac {\xi_1} \l) .$$
    The scaling is such that  for all $ \l \in  (0,1]$ 
    $$ \int_{\R^2}       J^\l (\xi) d\xi =1,    \quad\int_\R   \bar J^\l (\xi_1) d\xi_1=1 .$$  

\subsection { Interaction in bounded domain}
In a bounded domain  $Q$   we   require   the interaction  to    act  only when $\xi$ and $\xi'$ are in $Q$. 
Let $v$ be a function having support $Q$.
 We denote for $\xi \in Q$ 
  \begin{equation} \label {ok1} (J \star v) (\xi)=(J \star_Q v) (\xi)= \int_Q J(\xi- \xi') v(\xi') d \xi'.  \end {equation} 
  We do not add the suffix $Q$, unless confusion arises.  Same notation when $J$ is replaced by $J^\l$.
  
\subsection { Local variables  in a neighborhood of  $\G$ }
We parametrize $\G$  by arc length. Let $L$  be the length of $\Gamma$ and  let $ T$  be a    circle  of  length  $L$.
  Let     $ \g :  T \to   \G $ be  the map    parametrizing $ \G$   
 so that  for $s \in T $,  $|\g'(s)|=1$.   We assume    $\g \in C^5(T)$.  We denote by $ \nu(s)$ a smoothly  varying unit vector normal    to $ \G$  at  the point $\gamma (s)$.      We denote    by $ k(s) $ the signed curvature  defined by  $ \nu'(s) =  -k(s) \g'(s)$ .        We  therefore have   
      \begin{equation} \label {ok2}  |\nu(s)|=1, \quad   |\g'(s)|=1,    
      \quad \nu'(s) =  -k(s) \g'(s), \qquad \g''(s) = k(s) \nu (s), \qquad   \g'(s) \cdot  \nu(s)  =0,   \end {equation} 
      where for two vectors    $w_1$ and $w_2$ in $\R^2$, $w_1 \cdot w_2$ is the scalar product.
      Let  $d(\xi, \G)$ be  the euclidean distance of the point $\xi$ from  $\G$ and $r= r(\xi, \G)$ be the signed distance of a point $\xi$ from $\G$ with the convention that  $r <0$  when $\xi \in \Omega^-$   and  $r>0$  when $\xi \in \Omega^+$.
Let  $d_0>0$  be so that  
 $$\NN(d_0) = \{ \xi \in \R^2: d (\xi, \G) \le d_0 \} \subset \Omega.$$     
We  require $d_0$ to be small enough so  that         the   map   
     $ \rho:  T   \times [-d_0,d_0]  \to \NN(d_0) $,
     \begin{equation} \label {ch1a}   \xi = \g(s) + \nu (s) r  = \rho (s,r),   \end {equation}
           is      a  diffeomorphism.    We denote    $I= [-d_0,d_0]$,   $ \TT=  T   \times I$ and 
      $  \a (s,r) $ for $(s,r) \in \TT$    the Jacobian of the local change of
variables  
 \begin{equation} \label {sm1}  \a(s,r)= \hbox {det} \left (
\frac {\partial \rho (s,r)} {\partial (s,r)} \right)  =  1- r k(s).  \end {equation} 
 We further require $d_0$ to be small enough so that 
 \begin {equation} \label {SG1a}  \sup_{s \in T}  |k(s)| d_0    \le \frac 12. \end {equation}
 This  implies       $ \frac 32  \ge  \a(s,r) \ge \frac 12$  for all $(s,r) \in \TT$.  

A  function $u(\xi)$ for $ \xi \in \NN(d_0)$ becomes by the  change of  coordinates  
  $v(s,r)= u ( \rho (s,r))$, $(s,r) \in \TT$.   In the sequel we identify functions of variable $\xi$ and functions of variable $(s,r)$ in the domain $\NN(d_0)$.
  We often lift the function $v$ on the universal cover of $\TT$ without mentioning.
  We    write
  \begin{equation} \label {se1} \int_{\NN(d_0)} u(\xi) d\xi =  \int_{\TT}  u(s,r)  \a (s,r)  ds dr   \end {equation} 
 and  for $v$ and $w$ in $L^2(\TT)$ 
\begin {equation} \label {ass1}  \langle v, w \rangle = \int_{\TT} v(s,r)  w(s,r) ds dr. \end {equation}   
 \subsection  { The  approximate  solution    $m_A(\cdot)$}
Let  $\beta>1$,  $ m_\beta $ be  the strictly positive solution  of
 $ m= \tanh  \beta  m $  and   $\bar m $ be    the unique  antisymmetric  solution of
 \begin {equation} \label  {2.2}    m(z) = \tanh (\beta \bar J \star m) (z)  \quad \hbox {in} \quad \R, \qquad   m(0)=0, \quad  \lim_{z \to \pm
\infty} 
 m(z) =
\pm m_\b.  \end {equation}
Notice that $m_\b<1$ and  it  can be proven that   $ \bar m \in C^\infty (\R)$, it  is strictly 
increasing,  and  there exist $a>0$, $\alpha> \alpha_0>0$ and $c>0$ so that 
\begin{equation} \label{decay} \begin {split}
&0  <  m_\beta^2 - \bar  m ^2(z)  \le ce^{-\alpha |z|}\ ,  \cr  
  &  \quad | {\bar m}'  (z)- a \a  e^{-\alpha |z|}| \le    ce^{-\alpha_0 |z|}. 
\end {split}\end{equation} 
A proof of these estimates      and     related  results  can be found in \cite[Chapter 8, Section 8.2]{Pr}.
 We assume  that  $m_A(\cdot)$ in  $\NN( d_0)$ has the expansion 
 \begin{equation} \label {8.16}    m_A (\xi)  =       \bar m (\frac {r(\xi, \G)} \l )  + \l [  h_1 (\frac {r(\xi, \G)}
\l)g(s (\xi)) + \phi(\xi) ]+ \l^2 q^\l (\xi )    \qquad  \xi \in \NN( d_0).     \end {equation}  
  We assume that    $h_1  \in C^1 (\R) \cap L^\infty (\R)$, $ \lim_{|z| \to \infty} h_1(z) =0$ exponentially fast and 
 that  $h_1$ is  an   even function.  Hence  if $v$ is an odd function 
  \begin {equation} \label  {2.3a} \int_R   h_1(z) v(z) dz =0.  \end {equation}
In particular, since $ \bar m$ is an odd function, 
    \begin {equation} \label  {2.3}  \int_\R  \frac {\bar m(z)} {\s^2(\bar m(z))}  h_1(z) (\bar m' (z))^2\dha     z =0  \end {equation}
 where   $\s(m) =  \b (1-m^2)$,       see    Remark  \ref {ccos1}.
  We assume that
$g $ and $\phi $ are smooth functions, at leact $C^3$,
 \begin {equation} \label  {2.5} \sup_{\l \in (0,1]} \sup_{\xi \in \NN(d_0)} \left (  g(s(\xi)) + \phi(\xi) +  q^\l (\xi)  \right ) \le C,\end {equation}
  \begin{equation} \label {MG2} | \nabla^\G m_A (\xi)| \le C \l, \qquad \xi \in \NN(d_0), 
    \end {equation}
    where  $\nabla^\G$  is the tangential derivative to $\G$. 
  The function     $\phi(\cdot)$ has   a 
Lipschitz norm bounded  uniformly with respect to $\l$: 
 \begin {equation} \label  {2.4} \|\phi \|_{\hbox {Lip}   (\NN(d_0))} \le C.   \end {equation}
We assume that away of the interface,
 $$m_A (\xi) \simeq  \pm  m_\b    \qquad  \quad  \xi  \in \Omega^{\pm} \cap  \left [ \Omega \setminus  \NN( d_0) \right ] $$
where $ \simeq $ means up to correction of order $\l$. 
We require  that
  there exist $C^*>1$,   $a>0$  and $\l_0>0$  so that for $ \l\le \l_0$
\begin {equation} \label {sc1}   \inf_{ \xi \in \Om \setminus   \NN (\frac {d_0} 2)}  \frac 1 {\sigma (m_A(\xi))} >C^*,  \end {equation}
  and 
  \begin {equation} \label  {d2b}  a \le   \s(m_A(\xi))  \le  \beta,  \qquad \xi \in  \Omega. \end {equation}

 \begin {rem} \label {ccos}   The  assumptions  regarding $m_A$ are  suggested  by the preliminary results obtained by   applying the method  of \cite {CCO3}  to the  construction of  approximate solutions of  \eqref {1.1}.  
 The condition     \eqref {MG2}   is stronger than the corresponding requirement  used in proving the spectral estimates in the Cahn-Hilliard case. In   \cite {Chen},   it is required $| \nabla^\G m_A (\xi)| \le C$, for $\xi \in \NN(d_0)$, although the  function 
 approximating the solution of the Cahn-Hilliard case constructed in \cite {ABC} satisfies  condition \eqref {MG2}, see   \cite[ formula (5.2)] {ABC}. \end {rem}
 \begin {rem}   \label {ccos1} 
 Condition \eqref {2.3} corresponds to the condition  required  by     Alikakos, Fusco  and Chen  in the Cahn-Hilliard  case, see for example \cite [formula (1.12)]{Chen}.   
 In the Chen notation one needs  
\begin {equation} \label  {Xc}    \int_\R  f''(\bar m(z)) h_1(z) (\bar m' (z))^2\dha     z =0, \end {equation}
 where $ f''(\cdot)$ is the third derivative of  a double well potential.
Let $V(m)$ be the double well potential defined in \eqref {littlef}.   We have 
 $f(m)= V'(m)=   - m +  \frac 1 {2 \b} \ln
\frac {1+m}{1-m}$,   $f'(m)= -1 + \frac 1 \b \frac 1 {(1-m^2)},$ 
     $f''(m)=    2  \frac {m} {\beta (1-m^2)^2} $. Inserting this into \eqref {Xc} gives    condition  \eqref {2.3}.
     \end {rem} 
     \subsection {Main Results}
Denote  by $ A^\l_{m_A}$ the  operator  acting on   functions $ v \in L^2(\Om)$:
\begin {equation} \label  {8.10} ( A^\l_{m_A} v)  
(\xi)=  \frac {v (\xi)} { \s(m_A(\xi))}- (J^{\l} \star_ \Om v)
(\xi)  
 \quad \xi \in \Om.    \end {equation}
 For $\l \le \l_0$,  $\s(m_A(\xi))$  is strictly positive    for $\xi \in \Omega$, see \eqref {d2b},  therefore the  operator \eqref {8.10} is well defined.
 Denote  
$$ \XX= \left \{  v \in H^1(\Omega):    \; \Delta  w  = v, \int v =0  \right \} .$$ 
 We have the following main result. 
 \begin {thm} \label {mains}    Set  $\beta>1$. There exists $ \l_1>0$ such that for $ \l \in (0,\l_1]$ 
\begin {equation} \label {msp}   \inf_{\{v \in \XX \} }  \frac 1 \l  \frac {    
\int_\Omega   \left ( A^\l_{m_A} v  (\xi)\right )  v(\xi) d\xi }  { \int_\Omega   
|\nabla w (\xi)|^2 d \xi  } \ge - C.  \end {equation}
\end {thm}
 The proof of Theorem  \ref {mains}  is  based on  two important  intermediate results.
 \begin {thm}   \label {82}    Set  $\beta>1$.  There exists $ \l_2$ such
that for $ \l \in (0,\l_2]$  
 $$ \int_\Om   A^\l_{m_A} v  (\xi) v(\xi) d\xi \ge - C \l^2 \int_\Om v^2 (\xi) d\xi. $$
  \end {thm}

  \vskip0.5cm
 \begin {thm}   \label {86}   Set  $\beta>1$.  There exists 
$\l_4$ such that  if   for $ \l \le \l_4$, $v \in  \XX $  
\begin {equation} \label {v1} \int_\Omega  A^\l_{m_A} v  (\xi) v(\xi) d\xi    \le  C \l^2 
\int_\Omega v^2 (\xi) d\xi   
 \end {equation}
then 
\begin {equation} \label {S.200}  \|\nabla w \|^2_{L^2(\Omega)}  \ge  C \l \|v \|^2_{L^2(\Omega)}. 
 \end {equation}
 where $ w$ solves $\Delta w= v$  in $\Omega$ with Neumann boundary conditions on $\partial \Omega$.
 \end {thm}
Theorem \ref {82}  and Theorem \ref {86}   imply the thesis of Theorem \ref {mains}. 
 Preliminary  to our analysis  is the study  of   the spectra  of  one    dimensional linear operators. This is done  in Section 3.
 In Section 4 we prove Theorem \ref {82}.  In Section 5 we study the spectrum of a two dimensional  convolution  operator.
  The knowledge of it allows 
   to  prove in  Section 6   the  representation theorem for function  $v$  so that \eqref   {D12}  holds. This is the main ingredient to show Theorem 
 \ref {86} which is proven in   Section 7. In the Appendix, Section 8, we   collect  some  estimates  needed  to prove the results. 
 
  \section {One dimensional  convolution operators in enlarged intervals.}
  Denote by   
$z = \frac r \l$  the stretched variable and  by $I_\l  = [ -\frac {d_0} \l,  \frac {d_0} \l] $ the stretched  interval.    Also we denote  by $v$   the generic function of $(s,r)$  and by $V$ the  generic function of $(s,z)$.
   Define   for  $V \in L^2(I_\l) $ the following  operator
 \begin{equation} \label{op1} (\LL^0 V) (z) =  \frac {V(z) } 
 {\s (\bar m(z))} -    (\bar J \star_{I_\l} V) (z), \qquad z \in I_\l.    \end {equation}
  Preliminary to the analysis of  the spectrum of   $\LL^0$ is the knowledge of the spectrum  of the following  operator 
  $ \LL$ defined  on   the space $L^2 (\R)$.  
  \begin{equation} \label {G3} (\LL V)(z) =  \frac {V(z)} {\s(\bar m(z))}  -   (\bar J \star V) (z), \qquad z \in \R.  \end {equation} 
Spectral properties  of $\LL$   are given  in 
  \cite {DOPTE}.  The spectrum of $\LL$ is positive, the lower bound of the spectrum is 0
  which is an eigenvalue of multiplicity one and the corresponding eigenvalue is $ \bar m' (\cdot)$,
  i.e 
 \begin{equation} \label {G4} \LL \bar m'  =0.   \end {equation} 
    The remaining part of the spectrum is strictly bigger  then some positive number. 
    The operator $\LL^0$ is  the ``restriction" of the operator $\LL$ in  the  bounded interval $I_\l$. The spectrum of  $ \LL^0$ is studied in \cite {O1}.   We  collect  in Theorem \ref {81}   stated below  the main results.  
   Denote 
    $$  (V, W)  = \int_{I_\l} V(z)  W(z) dz, \qquad  \|V\|^2 = \int_{I_\l} V(z)^2 dz. $$
\vskip0.5cm
\begin {thm} [\cite {O1}] \label {81}    For any $\beta>1$ there exists $\l_0(\beta)$ so that for $\l \le\l_0(\beta)$ the following holds.

(0)The operator $\LL^0$ is a bounded,   self adjoint  operator on $L^2(I_\l) $.
 
 (1)     There exist  $ \mu_0^0  \in \R $ and  $ \psi^0_0 \in L^2(I_\l) $, $ \psi^0_0 $  strictly positive in $ I_\l$  
so that 
\begin{equation}
 \label{op2}  \LL^0  \psi^0_0=   \mu_0^0  \psi^0_0.  \end{equation} 
The eigenvalue $ \mu_0^0 $  has multiplicity one.  
 \begin{equation} \label{8.2}  0 \le  \mu_0^0 \le      C e^{- \frac  {2 \a } \l },   \end {equation}
   where $\a>0$ is given in \eqref {decay}.
Further  $\psi^0_0 \in C^\infty (I_\l)$, $\psi^0_0(z) =  \psi^0_0(-z)$ for $z \in I_\l$. 
The spectrum of $\LL^0$ is discrete  and any other eigenvalue  is strictly
bigger than  $\mu_0^0 $.

(2) Let $ \mu_2^0 $ be the second  eigenvalue   of $\LL^0$ and     $D>0$ independent on $\l$.
We have that 
  \begin{equation} \label{8.3} \mu_2^0= \inf_{  (V,\psi^0_0) =0; \| V\|=1}   ( V, \LL^0 V)\ge D.
 \end {equation}

(3)  Let  $\psi^0_0$ be the
  normalized eigenfunction  corresponding to $\mu_0^0$ we have
  \begin{equation} \label{8.4} \| \psi^0_0- \frac {\bar m'}{\|\bar
m'\|}\|  \le  C  e^{-
 \frac {2\a } \l}. 
 \end {equation}
   \end {thm}
   The point (0)   it is  easy to  prove.
To show point (1) one first  notes  that the     Perron-Frobenius Theorem holds   for the operator
$ (\PP^0  g)(z)= p(z) (\bar  J \star_{I_\l} g)(z) $,    $ g \in L^2(I_\l)$,  since $   \bar J$  is a positivity improving integral kernel. 
The operator $ \LL^0$ is conjugate to the operator $ \1-   \PP^0$.   This  implies immediately the result  stated in point (1).
  Estimate \eqref {8.3} is obtained  by  applying   the operator to a convenient trial function and using \eqref {2.3}.  The most difficult part is to show point (2). This has been obtained   by  applying a generalization  of   Cheeger's inequality. 
  For more details see  \cite {O1}. 
 Next we introduce a family of one dimensional operators.
For any $s \in T $ and for $m_A$ given in \eqref {8.16},
 let
 \begin{equation} \label{8.100} (\LL^s V) (s,z) = \frac {V(s,z)} { \s(  m_A (s, \l z))}-   (\bar J \star_{I_\l,z}  
V) (s,z), \qquad z \in I_\l   
 \end {equation}
be the operator acting on   $ L^2 ( I_\l) $ where  
 \begin{equation} \label{TU1}   (\bar J  \star_{I_\l,z} 
V) (s,z)  = \int_{I_\l}  \bar J (z-z') V(s,z') dz'.  \end {equation}
We  stress that $\LL^s$ acts for any fixed $s$  only on the $z$ variable of $V$.
     We denote 
$$  \langle V, W \rangle_s = \int_{I_\l} V(s,z)  W(s,z) dz, \qquad  \|V\|^2_s = \int_{I_\l} V(s,z)^2 dz. $$
By the definition of  $m_A$ given in \eqref {8.16} we have 
 \begin{equation} \label{mob1} \frac 1 {\b (1-m^2_A (s, \l z))}=  \frac 1 {\b (1-\bar m^2(z))}\left [ 1  +  \l  \frac{2 \bar m (z)} {(1-\bar
m^2(z)) }[ h_1(z) g(s) +   
\phi (s, \l z)]   \right ]  +   q^\l(s,\l z) \l^2. \end {equation}
  By the  point wise bound  \eqref {2.5}  we have that 
  $$ \left | \frac 1 {\b (1-m^2_A (s, \l z))}-  \frac 1 {\b (1-\bar m^2(z))} \right | \le C \l. $$
Therefore the  operator $\LL^s$ is for each $ s \in T$, a $\l-$perturbation of the operator $\LL^0$, i.e.
$$ \sup_{ \{\|V\|_s=1\}}  |\langle (\LL^s- \LL^0)V, V \rangle_s| \le C \l. $$
Nevertheless, by  \eqref {2.3},  it is possible to show that the perturbation on the principal eigenvalue of $\LL^0$ is of order $\l^2$.
   We have the following result.
   \vskip0.5cm
\begin {thm} \label {83}       For any $\beta>1$ there exists $\l_0 = \l_0(\beta)$ so that for $\l \le\l_0$ the following holds.

 \begin {enumerate} 
 \item  For all $s \in T$, the operator $\LL^s$ is a bounded,   selfadjoint  operator on $L^2(I_\l) $.
  There exist  $ \mu_1 (s)  \in \R $ and   $\Psi_1(s, \cdot)  \in L^2(I_\l) $, $ \Psi_1(s, \cdot) $  strictly positive in $ I_\l$  
so that 
  \begin{equation}   \label{gc3a}    \LL^s  \Psi_1(s, \cdot)=    \mu_1 (s)  \Psi_1(s, \cdot). \end {equation}
The eigenvalue $ \mu_1 (s)  $  has multiplicity one and any other point of the spectrum is strictly
bigger than  $\ \mu_1 (s)  $.

\item We have that  for all $s \in T$
  \begin{equation}   \label{8.310}  C \l^2 \ge  \mu_1 (s) =  \inf_{
\|V  
\|_s=1 }  \langle \LL^s V,V \rangle_s   \ge - C \l^2, \end {equation}
   \begin{equation}   \label{8.510}  \Psi_1(s, \cdot) = \frac 1 {\|m'\|} \bar m' (\cdot)+  \Psi_1^R(s,\cdot)     
\end {equation}
where
\begin{equation}   \label{8.500}\sup_{s \in T }  \| \Psi_1^R \|_s  \le C \l .\end {equation}
 Moreover, there exist $z_1>0$ and $\zeta_1>0$  independent on $\l$ so that
 \begin{equation}   \label{8.k1}   \Psi_1(s, z) \ge \zeta_1, \qquad |z| \le z_1, \qquad s \in T. \end {equation}
\item There exists  $\gamma>0$ such for every $ \l \in (0,\l_0]$ and
$s\in T $ 
\begin{equation}   \label{second} \mu_2(s) =   \inf_{ \langle \Psi,\Psi_1\rangle_s=0; \| \Psi\|_s=1}\langle \Psi, \LL^s \Psi\rangle _s \ge  \g.  \end {equation}

\item   
\begin{equation}   \label{E.9}  \sup_{s \in T} \|\nabla_s \Psi_1\|_s \le C  \|\nabla_s m_A\|_{L^\infty (\NN(d_0)}.    \end {equation}
\end {enumerate}
 \end {thm}
\begin {proof}
By the  symmetry  of    $\bar J $  and   \eqref {d2b},  immediately one gets  that    $\LL^s$ is a bounded,   self-adjoint  operator on $L^2(I_\l) $, for   $s \in T$.    Point (1)  of the theorem  follows  by  the positivity improving property of the integral kernel $\bar J$,
similarly as  done 
  in   proving    point (1) of Theorem \ref {81}.
  As a consequence we have 
that     the  principal    eigenvalue of the spectrum of $ \LL^s$,   $ \mu_1 (s)  $,  has   multiplicity one and any other point of the spectrum of  $ \LL^s$  is strictly
bigger  than  $ \mu_1 (s)  $.   Further the  associated  eigenfunction     $ \Psi_1(s, \cdot)$ does not change sign and we assume that it is 
 positive.

 \vskip0.5cm \noindent (2)  We show \eqref {8.310}.    Taking into account \eqref {mob1},  we  split  the operator $\LL^s$ as following   
  \begin{equation}   \label{8.120a}  
 \langle \LL^s V,V \rangle_s  =  \langle \LL^0 V,V \rangle_s  +  \langle [\LL^s- \LL^0] V,V \rangle_s,
  \end {equation}
  where $ \LL^0$ is the operator defined in \eqref {op1}.  We  write
  \begin{equation}   \label{8.120}        \langle [\LL^s-\LL^0] V,V \rangle_s     =  I_{2,s} (V) +  I_{3,s}  (V) + I_{4,s} (V)  \end {equation}
where
 \begin{equation}      \label {tre.2}    I_{2,s} ( V) = \l  2 g (s) \b \int_{  I_\l }  dz     \frac {1} {\s^2(\bar m (z))} 
\bar m(z) h_1(z) V(s,z)^2,    \end {equation}
 \begin{equation}      \label {tre.3}  I_{3,s}  (V)= \l    2  \b   \int_{ I_\l }  dz     \frac {1} {\s^2(\bar m (z))} 
\bar m(z) \phi (s,\l z)V(s,z)^2,   \end {equation}
 \begin{equation}      \label {tre.4} I_{4,s}  (V)= \l^2     \b   \int_{ I_\l }  dz     \frac {1} {\s^2(\bar m (z))} 
\bar m(z)  q^\l (s, \l z)V(s,z)^2 .  \end {equation}
Take as trial function 
    \begin{equation}   \label{dom.1} \bar V (s,z) =   \frac { \bar m'(z)} { \|\bar m'\|_{L^2(I_\l)}}   \qquad s \in T, \quad z \in I_\l . \end {equation}
By the variational form for the eigenvalues
\begin{equation}   \label{s.1} \mu_1 (s) \le   \langle \bar V,  \LL^s \bar V \rangle_s =  \langle \LL^0 \bar V, \bar V \rangle_s  +  \langle [\LL^s- \LL^0] \bar V, \bar V \rangle_s. 
 \end {equation}
 Next we compute the right hand side of \eqref {s.1}. We have 
 \begin{equation}       \begin {split}  &   \langle \bar V,  \LL^0 \bar V \rangle_s  =  \frac {1} { \|\bar m'\|^2_{L^2(I_\l)}}    \int_{ I_\l }  dz    \left \{
\frac {\bar m'(z)  }
{\s(\bar m)}- (\bar J \star_{I_\l} \bar m')(z)     \right \}    \bar m'(z) \cr &  =
 \frac {1} { \|\bar m'\|^2_{L^2(I_\l)}}     \int_{ I_\l }   dz    \left \{
\frac {\bar m'(z)  }
{\s(\bar m)}- (\bar J \star \bar m')(z)     \right \} \bar m'(z)  -    \frac {1} { \|\bar m'\|^2_{L^2(I_\l)}}   \int_{ I_\l} dz    \bar m'(z)  \int_{I_\l^c}  \bar J (z-z')  \bar m' (z') dz' \cr & =
-   \frac {1} { \|\bar m'\|^2_{L^2(I_\l)}}   \int_{ I_\l} dz    \bar m'(z)  \int_{I_\l^c}  \bar J (z-z')  \bar m' (z') dz',
       \end {split}   
 \end {equation}
since, see \eqref {G4},  $ (\LL \bar m' ) (z)= \frac { \bar m'(z)  }
{\s(\bar m (z))}- (\bar J \star \bar m')(z)  =0$.
By \eqref {decay} 
we have that 
 \begin{equation}  \label {al1}  \left |   \langle \bar V,  \LL^0 \bar V \rangle_s \right | \le     \frac {1} { \|\bar m'\|^2_{L^2(I_\l)}}  \int_{ I_\l} dz    \bar m'(z)  \int_{I_\l^c}  \bar J (z-z')  \bar m' (z') dz' \le C e^{-2\a  \frac {d_0} \l}.  \end {equation}
We split, as in  \eqref {8.120},   the second term in \eqref {s.1}  
  \begin{equation}      \label {tre.1}     \langle [\LL^s- \LL^0] \bar V, \bar V \rangle_s =  
  I_{2,s} (\bar V) + I_{3,s} (\bar V) + I_{4,s} (\bar V).  \end {equation}
 By condition \eqref {2.3}
$$ I_{2,s} (\bar V) = 0.$$
 Applying again \eqref {2.3}
   \begin{equation}      \label {sv1}      I_{3,s} (\bar V)   =
\l 2\b  \frac {1} { \|\bar m'\|^2_{L^2(I_\l)}}      \int_{  I_\l }  dz     \frac {1} {\s^2(\bar m (z))} 
\bar m(z)[ \phi (s,\l z)- \phi (s,0)] (\bar m'(z))^2.  
 \end {equation}
 By  \eqref {2.4}  and the exponentially decreasing of  $\bar m'$, see    \eqref {decay}, we have 
  $$| I_{3,s}  (\bar V)| \le \l 2\b   \frac {1} { \|\bar m'\|^2_{L^2(I_\l)}}     \left | \int_{  I_\l }  dz     \frac {1} {\s^2(\bar m (z))} 
\bar m(z) | \phi (s,\l z)- \phi (s,0)| (\bar m'(z))^2      \right |   \le \l^2C.$$
 By \eqref  {2.5} and \eqref {decay},  we have that
 $$ I_{4,s}  (\bar V)| \le C \l^2.$$
   Therefore,  recalling  \eqref {al1} we get 
\begin{equation}    \label {8.25}\mu_1 (s) \le \langle \bar V,   \LL^s \bar V \rangle_s  \le   C \l^2.   \end {equation}
We need to show  $ \mu_1 (s) \ge  -C\l^2$. 
Let   $\Psi_1(s, \cdot)$ be the  normalized positive eigenfunction associated to $ \mu_1 (s) $.
Since  $  \mu_1 (s)  \le  C\l^2$,  by Lemma \ref {exp}  stated in the Appendix,   for   $\l$ small enough,      
$\Psi_1(s,
\cdot)$  as function of $z$ decays   to zero exponentially fast for $z$ large enough.   Set 
\begin{equation}    \label {v.21}  \Psi_1(s,
\cdot) = a (s) \psi^0_0 (\cdot)  + (\psi^0_0  )^{\perp} (s,
\cdot),   \end {equation}
 where $\psi^0_0 (\cdot) $ is the normalized eigenfunction associated to the  principal  eigenvalue 
of $\LL^0$, see \eqref {8.4},  
 $$  a (s) = \int_{ I_\l }  \Psi_1 (s,z)
\psi^0_0 (z)  dz,  \qquad \int_{  I_\l } \psi^0_0 (z) (\psi^0_0  )^{\perp} (s,z)
{\rm d} z =0 \qquad s \in T. $$ 
 By  \eqref {8.120a} we have that
\begin{equation}    \label {8.20}    \mu_1 (s)=    \langle   \Psi_1 , \LL^s   \Psi_1 \rangle_s     = \langle \Psi_1, \LL^0   \Psi_1 \rangle_s  +  \langle [\LL^s- \LL^0]  \Psi_1,  \Psi_1 \rangle_s. \end {equation}
By  \eqref  {v.21}  we obtain 
\begin{equation}    \label {8.21}  \begin {split} & \langle   \Psi_1, \LL^0   \Psi_1 \rangle_s   = a^2(s)    \langle \psi^0_0, \LL^0\psi^0_0 \rangle_s+   \langle(\psi^0_0)^{\perp},
\LL^0(\psi^0_0  )^{\perp}\rangle_s  
\cr & \ge 
 a^2 (s)\mu_0^0 +  \mu_2^0  \| (\psi^0_0  )^{\perp}\|^2_s, \end {split}\end {equation}
 where $\mu_0^0$ and $\mu_2^0$ are, respectively, the first and second eigenvalue of $\LL^0$.
 We split  the  second term  on   the right hand side of \eqref {8.20}  as  in \eqref {tre.1} obtaining   
 \begin{equation}   \langle [\LL^s- \LL^0]  \Psi_1,  \Psi_1 \rangle_s  
 =  I_{2,s}( \Psi_1)+  I_{3,s} ( \Psi_1)+ I_{4,s} ( \Psi_1).  \end {equation}
   By the $L^\infty$ bounds on $  \s$, $q^\l$ and $\bar m$   we  have
\begin{equation}    \label {last1}  \left |  I_{4,s} (\Psi_1) \right | \le  C \l^2. \end {equation}
Taking into account \eqref {v.21} we get 
\begin{equation}    \label {8.23}  \begin {split} &    I_{2,s} (\Psi_1)  =
  \l g (s) \b a(s) \int_{  I_\l }  dz     \frac {1} {\s^2(\bar m (z))} 
\bar m(z) h_1(z)   \psi^0_0 (z)^2   \cr & +    \l g (s) \b \int_{  I_\l }  dz     \frac {1} {\s^2(\bar m (z))} 
\bar m(z) h_1(z)   \left [ (\psi^0_0  )^{\perp}(s,z)^2 + 2 a(s) \psi^0_0 (z)(\psi^0_0  )^{\perp}(s,z)  \right  ].      \end {split}    \end {equation}
 By       \eqref  {8.4}   and  \eqref  {2.3}  we have 
 \begin{equation}   \label {sl2}     \left | \int_{  I_\l }  dz     \frac {1} {\s^2(\bar m (z))} 
\bar m(z) h_1(z)   (  \psi^0_0 (z))^2 \right |  \le   C  e^{- \frac
{\a }
\l}.  \end {equation}
By the $L^\infty$ bounds on $  \s$, $h_1$ and $\bar m$ we have 
 \begin{equation}   \label {sl3}  \left | \int_{  I_\l }  dz     \frac {1} {\s^2(\bar m (z))} 
\bar m(z) h_1(z) \left [    ( (\psi^0_0 )^{\perp}(s,z))^2 + 2 a(s)\psi^0_0 (z)(\psi^0_0  )^{\perp}(s,z) \right ]  \right | 
\le    C  \|  (\psi^0_0  )^{\perp} \|_s.  \end {equation}
Therefore, by \eqref   {sl2}  and \eqref  {sl3},  
\begin{equation}    \label {sl1}    \left |    I_{2,s} (\Psi_1) 
 \right |  \le    \l C  \|  (\psi^0_0 )^{\perp} \|_s.  \end {equation}
For the  third   term in  the right hand side of \eqref {8.20},  we have 
\begin{equation}    \label {D.21}  \begin {split}  &  I_{3,s}  (\Psi_1) = \l  \b \int_{   I_\l  }  dz     \frac {1} {\s^2(\bar
m (z))} 
\bar m(z)
\phi_1(0,s) ( \Psi_1(s,z))^2  \cr & +
\l   \int_{   I_\l  }  dz     \frac {1} {\s(\bar m (z))}  \bar m(z)
[\phi_1(\l z,s) - \phi_1(0,s)]  (\Psi_1 (s,z))^2 \cr &
 \le \l C  \|  (\psi^0_0  )^{\perp} \|_s  + C\l^2.  \end {split}  
 \end {equation}
We  bound  the first term on the right hand side of \eqref  {D.21}, as in \eqref  {8.23}, by  splitting $ \Psi_1$, see \eqref {v.21}, taking into account    that  \eqref  {2.3} hold.   The bound for the  second  term on the right hand side of \eqref  {D.21}  is a consequence of   the
Lipschitz bound \eqref {2.4} for 
$\phi$  and the exponentially decay 
to zero of $\Psi_1(s, \cdot)$.
Therefore, see \eqref {8.20},\eqref {8.21},  \eqref {8.23}  and \eqref {D.21}  and \eqref {last1} we obtain
\begin{equation}   \label {8.24}  \langle   \Psi_1, \LL^s   \Psi_1 \rangle_s  \ge  a^2 (s) \mu_0^0 +  \mu_2^0  \| (\psi^0_0  )^{\perp}\|^2_s   -\l  \| 
(\psi^0_0 )^{\perp} \|_s  -  C\l^2.  \end {equation}
By  Theorem \ref {81}  we have that $\mu_0^0 \ge 0 $, $\mu_2^0 \ge D$. This, together
with the upper bound \eqref {8.25},  implies   
 \begin{equation}   \label {D.60} \|  (\psi^0_0
)^{\perp} \|_s  \le C \l .\end {equation}
Further,  writing $ \Psi_1$  as  in  \eqref {v.21} we have 
$$ 1=  \|\Psi_1\|_s^2 = a^2 (s)  \|\psi^0_0 \|^2 +  \|(\psi^0_0  )^{\perp}  \|^2_s  = a^2(s)  + 
\|(\psi^0_0  )^{\perp}  \|^2_s .  $$
Then 
 \begin{equation}   \label {D.30} a^2 (s) = 1-  \|(\psi^0_0 
)^{\perp}  \|^2_s  \ge  1- C \l^2, \qquad s \in T. \end {equation}
From \eqref {8.24}  and \eqref {D.60} we get therefore 
\begin{equation}   \label {8.26}  \langle   \Psi_1, \LL^s   \Psi_1 \rangle_s   \ge - C \l^2.  \end {equation}
 Next we  show \eqref  {8.510}.        By    \eqref {8.4}  there exists a function $R(z)$, $z \in I_{\l}$,  $ \|R \| \le Ce^{-  \frac {\a} {\l}}$   so that  $\psi^0_0 (\cdot) =  \frac {\bar m' (\cdot)}{\|\bar m'\| }+ R (\cdot)$.  Therefore by   \eqref {v.21}  
  we have  
$$  \Psi_1(s,
\cdot) = a (s) \psi^0_0 (\cdot)  + (\psi^0_0)^{\perp} (s,
\cdot)    = a(s) \left [   \frac {\bar m' (\cdot)}{\|\bar m'\| }+ R (\cdot) \right ] + (\psi^0_0 )^{\perp} (s,
\cdot).$$
Denote by 
   $$ \Psi_1^R(s,\cdot) =   a(s) R (\cdot)   +     (\psi^0_0  )^{\perp} (s, \cdot).$$ 
By   \eqref {D.60}
 $$ \| \Psi_1^R\|_s  \le  \| (\psi^0_0  )^{\perp}\|_s + C 
e^{-
\frac {\a}
\l}   \le C \l   \qquad s \in  T $$  
 and \eqref {8.510} is proven.
 The \eqref {8.k1} is a consequence of exponential decay of $ \Psi_1 (s,\cdot)$, see Lemma \ref {exp}. 
  Further $$ \int_{I_\l}   \Psi_1(s,z)^2 dz   =1, \qquad \forall s \in T. $$ 
Therefore  we must have  that there  
exists $z_1>0$  independent on $\l$ so that 
\begin {equation} \label {ab10} \int_{- z_1}^{z_1}    \Psi_1(s,z)^2 dz \ge \frac 1 2. \end {equation}
This implies that  there exists $\zeta_1>0$, independent on $\l$ so that 
$$  \Psi_1(s,z) \ge \zeta_1, \qquad   |z| \le z_1,  \forall s \in T.$$
Namely, if this is false, there  exists $ \bar z  \in   [-z_1,z_1] $ so that 
$ \Psi_1(s, \bar z) =0$.  Since $\bar J$ is positivity improving and $ \Psi_1$ is an eigenfunction one can easily show that  $ \Psi_1(s,  z)=0$   for  $ z\in  [-z_1,z_1] $.
This is impossible since \eqref {ab10} holds.

\vskip0.5cm \noindent (3) Let $ \Psi_2 (s, \cdot)$ be one of the normalized eigenfunctions
corresponding to the second eigenvalue $ \mu_2 (s)$  of $ \LL^s$,   then
$$\int_{ I_\l } \Psi_1 (s,z) \Psi_2 (s,z)\dha     z =0, \qquad \forall  s \in T.$$ 
Split $ \Psi_1 (s, \cdot) $ as in \eqref {v.21}.
   We obtain
$$0=   \langle  \Psi_1,    \Psi_2 \rangle_s  = a(s) \langle \psi^0_0,    \Psi_2\rangle_s + \langle(\psi^0_0)^{\perp},   \Psi_2\rangle_s. $$
Therefore, taking into account \eqref {D.30}, 
$$  \langle \psi^0_0,   \Psi_2 \rangle_s = 
- \frac 1 {a(s)}  \langle (\psi^0_0)^{\perp},   \Psi_2\rangle_s.  
$$
Hence, by  \eqref {D.60}, 
\begin{equation}   \label {E.3} \left | \langle \psi^0_0,   \Psi_2 \rangle_s  \right |
      \le \frac 1  {|a(s)|}\|(\psi^0_0)^{\perp}\|_s \le  C \l. \end {equation}
 Further we   assume that $\mu_2(s)$ satisfies the hypothesis of Lemma \ref {exp}. In fact either it  is
small and therefore  satisfies the hypothesis of Lemma \ref {exp}, either it  is large, than there is
nothing to prove. Then we can argue as in Lemma \ref {exp} and show that $\Psi_2 (s,\cdot)$, as
function of $z$  decays exponentially fast to zero. 
So we can decompose
$ \Psi_2(s,\cdot)$   as
\begin{equation}   \label {E.1}   \Psi_2(s, \cdot)=  \langle  \Psi_2, \psi^0_0 \rangle_s  \psi^0_0(\cdot) + \left [  
\Psi_2(s, \cdot)-  \langle  \Psi_2, \psi^0_0 \rangle_s \psi^0_0(\cdot) \right ].  \end {equation}
We therefore obtain
$$  \mu_2(s) =  \langle \LL^s \Psi_2, \Psi_2\rangle_s =  \langle \LL^0   \Psi_2,
  \Psi_2 \rangle_s +  \langle [\LL^s- \LL^0]  \Psi_2,
  \Psi_2 \rangle_s .$$
By      \eqref {8.120} 
$$ \left | \langle [\LL^s- \LL^0]  \Psi_2,
  \Psi_2 \rangle_s\right | \le C \l .$$
Then inserting \eqref {E.1}  we have that
\begin{equation}   \label {E.2}  \begin {split}&  \mu_2(s)  \ge 
\mu_0^0   \langle  \Psi_2, \psi^0_0 \rangle_s^2 +  \mu_2^0 \left [  
1-  \langle  \Psi_2, \psi^0_0 \rangle_s^2 \right ]- C\l  \cr &  \ge  D  [1- C \l^2] - C\l  \end {split}
\end {equation}
where $ D>0$ is the lower bound in \eqref {8.3}  and we applied  estimate \eqref {E.3}.  
For $ \l$ small enough, there exists $\g>0$ so that \eqref {second} holds.

\vskip0.5cm \noindent (4) To prove  \eqref {E.9} we  differentiate  the  eigenvalue equation for $\Psi_1$ with respect
to $s$.  We obtain
\begin{equation}   \label {EE.1}   [\LL^s - \mu_1(s) ]\Psi_{1 s}= \mu_{1s}(s)\Psi_1 - \frac d {ds} \left (\frac {1} { \s(m_A)} \right )\Psi_1.
 \end {equation}
 Since 
\begin{equation}   \label {EE.10}  \int \Psi_{1 s} (z) \Psi_1(z) dz =0 
\end {equation}
\begin{equation}  \label {EE.2} \  \langle [\LL^s - \mu_1(s) ]\Psi_{1 s}
,\Psi_{1 s} \rangle_s   =   
  - \langle \frac d {ds} \frac {1} { \s(m_A)} \Psi_1, \Psi_{1 s} \rangle_s.  
  \end {equation}
 Therefore,
 $$ \left | \langle [\LL^s - \mu_1(s) ]\Psi_{1 s}
, \Psi_{1 s} \rangle_s \right |  \le    C \| \Psi_{1s}\|_s \  \|\nabla_s m_A\|_{L^\infty 
(\NN(d_0)}.$$ 
On the other hand, by \eqref {second}
$$ \langle [\LL^s - \mu_1(s) ]\Psi_{1 s}, \Psi_{1 s} \rangle_s   \ge (\g-\mu_1(s) )  \|\Psi_{1 s}\|_s^2. $$
Then  
\begin{equation}  \label {EE.5} \|\Psi_{1 s} \|_s \le  \frac { C} { (\g-\mu_1(s) )}   \|\nabla_s m_A\|_{L^\infty 
(\NN(d_0)} \le C\|\nabla_s m_A\|_{L^\infty (\NN(d_0)}  . \end {equation}
 This completes the proof of the theorem.
  \end {proof}
  In the following  we  deal  with functions defined in   $ I= [-d_0,d_0]$.  To this end, 
for       $u(s, \cdot) \in L^2 (I) $, $s \in  T $, 
  denote  
   \begin {equation} \label {D1}  (L^{\l,s}_1 u ) (s,r) = \frac {   u   (s,r)} {\s(m_A(s,r))} -  (\bar J ^{\l} \star_{I,r} u )(s,r), \end {equation}
where
  \begin {equation} \label {Dt1}  (\bar J ^{\l} \star_{I,r} u )(s,r)= \int_{I}  \bar J^{\l} (r-r') u(s,r') dr'.\end {equation}
 The subscript  $1$ in \eqref {D1} is to remind the reader that $L^{\l,s}_1$ acts only only on functions of the $r-$ variable.  
 We   immediately  have the following.
 \begin {prop}  \label {cec1}The spectrum of $L^{\l,s}_1$ on  $L^2 (I)$ is equal to the spectrum of $ \LL^s$ on  $L^2 (I_\l)$.
   In particular the principal eigenvalue of $L^{\l,s}_1$  is 
  \begin {equation} \label {SS2a} \psi_1(s,r) =  \frac 1 {\sqrt \l}  \Psi_1(s, \frac r \l),  \end {equation}
  where $\Psi_1(s, \cdot)$ is the principal eigenvalue of $\LL^s$.
 \end {prop} 
 \begin {proof}
 The operator $ \LL^s:  L^2 (I_\l)  \to   L^2 (I_\l)$ and the operator 
 $ L^s_1:  L^2 (I)  \to   L^2 (I)$   are conjugate.
 Namely, let   $T:   L^2 (I_\l)  \to  L^2 (I)$  be the map so that   $(T U) (s,r)  = \frac 1 {\sqrt \l} U(s, \frac r \l )$.
 The map $T$ is an isometry: 
   $$   \int_{I} ((T U) (s,r))^2   dr = \int_{I}  \frac 1 {  \l} U^2(s, \frac r \l )dr =  \int_{I_\l} U^2(s,z) dz $$
    and  $ L^{\l,s}_1=    T \LL^s  T^{-1}$.
    Therefore the spectrum of $L^s_1$ is equal to the spectrum of $\LL^s$ and   \eqref {SS2a} holds 
      \end {proof}

  \section {$L^2$ estimates}
 This section is devoted to the proof of   Theorem \ref {82}.    Let  $\eta(\cdot)$ be the indicator function of the set $\NN (d_0)$,  $\eta (\xi)=1$ when $\xi \in \NN (d_0)$,  $\eta (\xi)=0$ when $\xi \notin \NN (d_0)$.
In   Subsection 4.1    we  bound  from below    $  \int_{\NN (d_0)}  (A^\l_{m_A} \eta u)  (\xi)  \eta (\xi) u(\xi)\dha    \xi$
in term of the quadratic form of  the local operator $L^\l$ defined  in \eqref {lg1}.
Then,   in Subsection 4.2  we bound from below   the quadratic form of  the local operator $L^\l$.
Finally in Subsection 4.3 we show  Theorem \ref {82}.

\subsection 
  {Bound  from below  of  $  \int_{\NN (d_0)}  (A^\l_{m_A} \eta u)  (\xi)  \eta (\xi) u(\xi)\dha    \xi$.} We 
  start  writing in local coordinates   when $\xi \in  \NN(d_0)$ the  integral  $  \int_{\NN (d_0)} J^\l (\xi-\xi') u(\xi') d \xi'   $.   
       \begin {lem} \label {M0}  Let   $\xi \in  \NN(d_0)$,  $ \xi = \rho  (s,r)    $,   $(s,r) \in \TT $ be the change of variables defined in \eqref {ch1a}, $u(s,r)= u(\rho (s,r))$ and $\a(s,r)$ as in  \eqref {sm1}.  We have that 
       \begin {equation} \label  {VM0}  \begin {split}   
 \int_{\NN (d_0)} J^\l (\xi-\xi') u(\xi') d \xi'     &=  
 \int_{ \TT}    J^\l (s,s', r, r')    u(s',r') \a(s',r') ds' dr'     \cr &  +
 \int_{ \TT} R_1^\l (s,s', r, r')     u(s',r')  \a(s',r')  ds' dr'  
 \cr &  +   \int_{ \TT}  R_2^\l (s,s', r, r')    u(s',r')  \a(s',r')  ds' dr', 
  \end {split}\end {equation}
  where,
   for $s^*= \frac {s+s'} 2$, $r^*= \frac {r+r'} 2$ and  $\a(s^*,r^*)= 1-  k(s^*)r^*$
   \begin {equation} \label  {sme1} J^\l (s,s', r, r')=  J^\l \left ((s-s') \a(s^*,r^*),     (r-r')\right),  \end {equation}
   $R_1^\l$ is defined in \eqref  {sm4}  and  $R_2^\l$ in \eqref  {sm5}.
   Further we  have that 
\begin {equation} \label  {doc1} \left |  \int_{ \TT} R_1^\l (s,s', r, r')      \a(s',r')ds' dr' \right |   \le  C\l^2,  \end {equation}
\begin {equation} \label  {doc2}  \left |  \int_{ \TT}  R_2^\l (s,s', r, r')      \a(s',r') ds' dr' \right |   \le  C\l^4.  \end {equation}
   \end {lem}
  For  the  proof of the lemma   it is enough  that          $\G$  is  a  $C^3$ curve.  The proof  is simple although lengthy and it is reported in the appendix.  The  decomposition  obtained  in Lemma \ref {M0}  turns out to be  very useful.
    Set    for $u \in  L^2 \left ( \TT \right )$
    \begin {equation} \label {me7}(B^\l  u) (s,r)  =  \int_{\TT}     J^\l (s,s', r, r')   \a(s^*,r^*)    u  (s',r')      ds' dr',          \end {equation} 
    and
\begin {equation} \label {lg1} (L^{\l}  u ) (s,r) = \frac {   u   (s,r)} {\s(m_A(s,r))} - (B^\l  u) (s,r).     \end {equation}
The operator $ L^{\l}$ is selfadjoint in $L^2 \left ( \TT \right )$.
We  have the following result. 
\begin {lem} \label {F2}  
Let   $\eta(\cdot)$ be  the indicator function of the set $\NN (d_0)$.
    Set 
 \begin {equation} \label {F.1a} \hat u  (s,r) = \sqrt { \a (s,r) }  u (s,r). \end {equation}  We get that 
\begin {equation} \label {gi.1a}    \int_{\NN (d_0)}  (A^\l_{m_A} \eta  u ) (\xi) \eta (\xi) u(\xi)\dha    \xi  \ge      
\langle  \hat u,  L^{\l} \hat u \rangle  -  \l^2  C \|u\|^2_{\NN (d_0)}.
 \end {equation}
 \end {lem} 
\begin {proof}
By changing variables we have that
 $$\int_{\NN (d_0)}  (A^\l_{m_A} \eta  u ) (\xi) \eta (\xi) u(\xi) \dha    \xi = \int_{ \TT}  ds    dr  \a(s,r)   u (s,r)  \left \{ \frac {u (s,r) } { \s(m_A(s,r)) }- ( J^\l   \star \eta u) (s,r)     \right \}. $$
 Taking into account Lemma \ref {M0}    we have 
\begin {equation} \label {8.12}   \begin {split} &  \int_{ \TT}  ds    dr   \a(s,r)  u (s,r)   \left \{ \frac {u (s,r) } {\s(m_A(s,r)) }- ( J^\l \star \eta  u) (s,r)     \right \}   \cr &   \ge     
\int_{ \TT}  ds    dr   \a (s,r)    u (s,r) \left \{
\frac {u (s,r) }  { \s(m_A(s,r)) }-      \int_{\TT}    J^\l (s,s', r, r')    u(s',r') \a(s',r') ds' dr'    \right \}    -   \l^2 C \|u\|^2_{L^2(\NN(d_0))}.
\end {split}  \end {equation}
By \eqref {doc1}  and \eqref {doc2} one obtains 
 $$ \left | \int_{ \TT}  ds    dr  dr   \a (s,r)   u(s,r)    \int_{ \TT}  ds'     dr'     R_1^\l (s,s', r, r')  u(s',r')     \a (s',r')    \right |   \le  C\l^2 \|u\|^2_{L^2( \TT, \a(s,r) ds dr)} = C\l^2 \|u\|^2_{L^2(\NN(d_0))},$$
  $$ \left | \int_{ \TT}  ds    dr  dr   \a (s,r)   u(s,r)    \int_{ \TT}  ds'     dr'     R_2^\l (s,s', r, r')  u(s',r')     \a (s',r')    \right |   \le  C\l^4\|u\|^2_{L^2( \TT, \a(s,r) ds dr)} =  C\l^4 \|u\|^2_{L^2(\NN(d_0))}.$$
   Set $  \hat u$ as in \eqref {F.1a} and 
\begin {equation} \label {se5} (\hat  \D^\l \hat u) (s,r)   =   \sqrt {\a(s,r)}  \int_{\TT}       J^\l (s,s', r, r')   \sqrt {\a(s',r')}   \hat u  (s',r')      ds' dr'.   \end {equation}
  We have immediately  that 
  \begin {equation} \label {s.2a}   \begin {split}  &
\int_{ \TT}  ds    dr      \a (s,r)    u (s,r) \left \{
\frac {u (s,r) }
{\s(m_A(s,r))}-     \int_{\TT}    J^\l (s,s', r, r')    u(s',r') \a(s',r') ds' dr'      \right \}    \cr &=  
\int_{ \TT}  ds    dr     \hat u (s,r)   \left \{
\frac {\hat u (s,r) }
{\s(m_A(s,r))}- (\hat  \D^\l   \hat u) (s,r)     \right \}   \cr &=
\int_{ \TT}  ds    dr       \left \{
\frac {\hat u (s,r) }
{\s(m_A(s,r))}- ( B^\l   \hat u) (s,r)     \right \}    \cr &  +
\int_{ \TT}  ds    dr       \hat u (s,r)  (\CC^\l \hat u) (s,r)
\end {split}  \end {equation}
where  
\begin {equation} \label {ass2}  (\CC^\l \hat u) (s,r) 
=  \int_{\TT}         J^\l (s,s', r, r')   \left [ \sqrt { \a(s,r) \a(s',r')}    - \a(s^*,r^*) \right ]  \hat u  (s',r')      ds' dr'.     \end {equation}
 By Lemma \ref  {M1} given below  we have that
     \begin {equation} \label {f.1} |\langle  \hat u,  \CC^\l \hat u \rangle| \le C \l^2 \|\hat u \|^2_{L^2(\TT)} .  \end {equation}  
 Taking into account    definition \eqref {lg1}, from \eqref  {s.2a}  we get     \eqref {gi.1a}.  
 \end {proof}

     \begin {lem} \label {M1}  Take $|s-s'| \le \l$ and $ |r-r'| \le \l $. We have 
 \begin {equation} \label {me1}  \sqrt { \a(s,r) \a(s',r')}  =       \a(s^*,r^*)  \sqrt { 1   +  b (s,s',r,r') }   
  \end {equation}
  where $b (s,s',r,r')$ is defined in \eqref  {me5}  and 
   \begin {equation} \label {me2a}  |  b (s,s',r,r')| \le C \l^2.  \end {equation}
     \end {lem}
     \begin {proof}  Adding and subtracting $k(s^*) r^*$ we have 
      \begin {equation} \label {me2}   \begin {split}  &   \a(s,r) \a(s',r') =  \a(s^*,r^*)^2 +  \a(s^*,r^*)\left  [ 2 k(s^*)r^*  - k(s')  r'   - k(s) r \right ]  \cr & + [k(s^*) r^*- k(s') r'] [k(s^*) r^*- k(s) r] 
      \cr & =  \a(s^*,r^*)^2  \left [ 1 +  b (s,s',r,r') \right ] 
     \end {split} 
     \end {equation}
    where     \begin {equation} \label {me5}  b (s,s',r,r') =   \frac {  \left  [ 2 k(s^*)r^*  - k(s')  r'   - k(s) r \right ] } {\a(s^*,r^*)} + 
       \frac { [k(s^*) r^*- k(s') r'] [k(s^*) r^*- k(s) r] }   {\a(s^*,r^*)^2}.  \end {equation}
Next we show  \eqref {me2a}. 
             We have
              \begin {equation} \label {me3} 
             \left  [ 2 k(s^*)r^*  - k(s')  r'   - k(s) r \right ] =  [   k(s^*)- k(s)] r +  [ k(s^*)   - k(s')] r'.   
               \end {equation}
               Taylor expanding  $k(\cdot)$ we have 
        \begin {equation} \label {me4}   
         k(s') = k(s^*) + k'(s^*) [ s'-s^*] + \frac 12  k''(s^*) [ s'-s^*]^2  +   \frac 16 k'''(\tilde s) [ s'-s^*]^3,
          \end {equation}
          where $\tilde s \in (s,s')$.
  Similarly  we proceed for $  k(s) $.   By \eqref {me3},  taking into account \eqref {me4} and that  $  s'-s^*= \frac { [s'-s]} 2$ and  $ s-s^* = \frac { [s-s']} 2$ we obtain
    \begin {equation} \label {me5a}  \begin {split}  &
\left | [k(s^*) - k(s')] r'  +  [k(s^*) - k(s)] r \right |  \cr &  \le   \left |   k'(s^*)  \frac { [s'-s]} 2 [r-r']   +  \frac 1  4  k''(s^*)   [s'-s]^2  r^* 
+   \frac 1 { 48} k'''(\tilde s)  [s'-s]^3 [r-r']   \right |  \le   C \l^2.
 \end {split}  \end {equation}
 By similar computations
 $$ \left |  [k(s^*) r^*- k(s') r'] [k(s^*) r^*- k(s) r]  \right |  \le   C \l^2. $$
  \end {proof}

     \subsection { Properties of $L^\l$} 
      The operator $L^{\l}$  acts on  function of $(s,r)$. In  Corollary  \ref {M2aa}, stated below, 
we show that  when the operator  $ L^{\l}$ acts on functions depending only on  the variable $r$,  it   is, in the $L^2$ norm,  $\l^2$  close  to the one dimensional operator $L_1^{\l,s}$, defined in \eqref {D1}.   Preliminarily  we need to see how  the operator $B^\l $, defined in  \eqref  {me7}, acts on functions depending only on $r$.
This is the content of the next lemma.      The symmetry of $\bar J $   is essential  to obtain  the  estimate \eqref {lu8}.
      \begin {lem} \label {M2a}    Let $  B^\l $ be  the operator defined in \eqref  {me7} and $v$ a function depending only on  the $r$ variable.  For any $s\in T $ we have  
 \begin {equation} \label {mer3} (   B^\l v) (r)=  (\bar J^{\l} \star _Iv)(r)  +  (\Gamma^{\l,s} v)(r),  \quad r \in I \end {equation} 
where    $  \Gamma^{\l,s}$ is the operator defined in \eqref {lu2a},  
\begin {equation} \label {lu8}  | \Gamma^{\l,s} v (r)|  \le  C \l^2
        (\bar J^{\l}  \star_I  |v| ) (r), \qquad r \in I.       \end {equation} 
   When  $v \in L^2(I)$,       \begin {equation} \label {lu8a}  \| \Gamma^{\l,s} v \|_{L^2(I)} \le  C \l^2
       \|v \|_{L^2(I)}.       \end {equation}  \end {lem}
 \begin {proof}  Taking into account that the support of $J^{\l}$, see  the definition of  the operator $ B^\l$,  is the ball of radius $\l$ centred in $(s,r) \in \TT$ we make the following local change of variable:
   For each  $(s,r) \in \TT$ and $r' \in I$,     
   \begin {equation} \label {mer1} w = f(s') =  -(s-s') \a(s^*,r^*), \qquad |s-s'| \le \l. \end {equation}  
   Notice that   we are not 
  explicitly writing  the dependence on $(s,r,r')$.
   Denoting by $f' $ the derivative with respect to $s'$ of $f$, we have 
  $$ f'(s') =  \a(s^*,r^*) + \frac 12 (s-s')   k'(s^*)r^*. $$
  By \eqref {SG1a} and for $\l$ small enough, $f'(s') >0$ when  $ |s-s'| \le \l$.
    By the inverse function theorem we have  
  $$ ds' = \frac 1 { f' (s'(w))} dw, \qquad   f' (s' (w)) = \a(s^*(w),r^*) +\frac 12 \frac {w} {\a(s^*(w),r^*)} k'(s^*(w))r^*,$$
   where by an abuse of notation we set $s^*(w)= \frac 12 s+ \frac 12 f^{-1} (w)$.
   We have 
   \begin {equation} \label {lu2}   \begin {split}  &  ( B^\l v) (r) =    \int_{\TT}  J^\l (  (s-s')\a(s^*,r^*),   (r-r') )  
  \a(s^*,r^*)    v (r')      ds' dr' \cr & =   \int_ {I} dr'
  \int_{f^{-1}(|s-s'| \le \l) }        J^\l ( w, r-r' )  
  \a(s^* (w),r^*)  \frac 1 { f' (s'(w))}  v (r')      dw   \cr &=
  \int_ {I} dr'
  \int   J^\l ( w, r-r' ) 
     v (r')      dw    +    (\Gamma^{\l,s} v)(r)
   \end {split} \end {equation}
where 
  \begin {equation} \label {lu2a}    (\Gamma^{\l,s} v)(r)=    \int_ {\TT} dr'
 dw       J^\l (w, r-r')  
 \left [   \frac 1 {1+ \frac {w   k'(s^*(w))r^*} {  \a(s^*(w),r^*)^2 }} - 1 \right ]  v (r'). \end {equation}
For $r$ and $r'$ in  $I$,  the support of $ J^\l (w, r-r') $   is   the ball of radius $\l$ centered in  the point $(0,r) \in \TT$. Hence   we 
 have  for $ r \in I$ and $r' \in I$
   \begin {equation} \label {merc2}  \int    J^\l (w, r-r') dw  
     = \bar J^{\l}  (r-r'),  \end {equation}
        where $\bar J^{\l}  $  is defined in  Subsection 2.1.
 Therefore from   \eqref {lu2},  \eqref {lu2a}  and \eqref {merc2} we have  \eqref {mer3}. 
 Denote shortly by  $a=  \frac {w   k'(s^*(w))r^*} {  \a(s^*(w),r^*)^2 }$.
 Since $|w| \le \l$ we have that    $|a| \le C\l$.  Writing  $  \frac 1 {1+ a} = \sum_{k=0}^\infty   (-a)^k $ we have that
 $$  \left [   \frac 1 {1+ a} - 1 \right ] = -a + \sum_{k=2}^\infty   (-a)^k. $$ 
 Therefore from \eqref {lu2} we obtain 
  \begin {equation} \label {lu9}   \begin {split}  &    ( \Gamma^{\l,s} v) (r)  = -
  \int_ {I} dr'
  \int           J^\l (w, r-r')  
    \frac {w   k'(s^*(w))r^*} {  \a(s^*(w),r^*)^2 }    v (r')      dw   \cr &  +
  \int_ {I} dr'
  \int           J^\l (w, r-r')  
  \sum_{k=2}^\infty  (-a)^k     v (r')      dw. 
  \end {split} \end {equation}
  For the second term we have
   \begin {equation} \label {merc3}  \begin {split}  & 
   \left |  \int_ {I} dr'
  \int            J^\l (w, r-r')  
  \ \sum_{k=2}^\infty   (-a)^k    v (r')      dw \right | \cr & \le  C\l^2 
  \int_ {I} dr'
  \int            J^\l (w, r-r')  
   |v (r')|      dw  =  C\l^2 
  \int_ {I} dr'  \bar   J^\l (r-r')     |v (r')|.   \end {split} 
 \end {equation}
For the first term  of \eqref{lu9} develop in Taylor expansion around $w=0$
$$ \frac {    k'(s^*(w))r^*} {  \a(s^*(w),r^*)^2 } = \frac {    k'(s^*(0))r^*} {  \a(s^*(0),r^*)^2 }     + g(\tilde w) w$$
where we denote by $g$ the derivative of  $\frac {    k'(s^*(w))r^*} {  \a(s^*(w),r^*)^2 }$ with respect to $w$. We have that
  \begin {equation} \label {lu3}   \begin {split}  &    \int_ {[-  d_0,  d_0]} dr'
  \int         J^\l (w, r-r')  
    \frac {w   k'(s^*(w))r^*} {  \a(s^*(w),r^*)^2 }    v (r')      dw  \cr & =
    \int_ {[-  d_0,  d_0]} dr'   \frac {    k'(s^*(0))r^*} {  \a(s^*(0),r^*)^2 }    v (r')
  \int             J^\l (w, r-r')  
             w dw
  \cr & +
   \int_ {[-  d_0,  d_0]} dr'  \int         J^\l (w, r-r')  
   g(\tilde w) w^2     v (r')      dw.
    \end {split} \end {equation}
   By the symmetry of $J^\l$
    $$  \int             J^\l (w, r-r')  
             w dw
=0.$$
Since $ |g(\tilde w)|  w^2 \le C \l^2$,  from \eqref {lu3},  \eqref  {merc3} and \eqref  {lu9}
we obtain
       \begin {equation} \label {lu9a}     |( \Gamma^{\l,s} v) (r) | \le   C \l^2
          (\bar J^{\l}  \star  |v| ) (r).      \end {equation}
           \end {proof}

  \begin {cor} \label  {M2aa}  Let  $ v \in L^2 (I)$ we have that
  \begin {equation} \label {SD30} (L^\l v) (r)=  (L^{\l,s}_1 v)(r)  +  (\Gamma^{\l,s} v)(r),   \end {equation} 
where $  \Gamma^{\l,s}$ is the  operator  defined in \eqref  {lu2a}  and
\begin {equation} \label {lu8aa}  | \Gamma^{\l,s} v (r)|  \le  C \l^2
        (\bar J^{\l}  \star  |v| ) (r).       \end {equation}
 \end {cor}
 The proof is immediate by  recalling the definition of $L^\l$ given in \eqref {lg1}  and   $L^{\l,s}_1$ given in \eqref {D1}.
 \begin {cor} \label {lunsoir}
  \begin {equation} \label {lus1a}   \int_{\TT }  J^\l (s,s', r, r')   \a(s^*,r^*) ds' dr'  =  1 +   (\Gamma^{\l,s})(r) \end {equation}
where  we denote by
 $   (\Gamma^{\l,s})(r) $ the quantity defined in \eqref  {lu2a}  when applied to the function $v(r)=1$
 for all $r \in I$
 and
 $$ |(\Gamma^{\l,s})(r)| 
   \le  C\l^2.  $$
   Further
   \begin {equation} \label {lus1}  \int_{T }  J^\l (s,s', r, r')   \a(s^*,r^*) ds'  =  J^{\l}(r-r')   +   (\Gamma^{\l,s}_1)(r,r'),  \end {equation}
   where $\Gamma^{\l,s}_1$ is defined in \eqref {lu2ab}, 
   \begin {equation} \label {lus1b} |(\Gamma^{\l,s}_1)(r,r')| 
   \le  C\l^2 J^{\l}(r-r').  \end {equation}
\end {cor}
\begin {proof}  The proof of \eqref {lus1a} is straightforward.  Take $v(r)=1$ in Lemma \ref {M2a} and  the thesis follows. 
To show  \eqref {lus1},  we proceed similarly as in Lemma \ref {M2a}.  
For $r$ and $r'$ in $I$, we have as  in  \eqref {lu2}   
 \begin {equation} \label {lu2g}   \begin {split}  &     \int_{T}     J^\l (  (s-s')\a(s^*,r^*), r-r')  
  \a(s^*,r^*)          ds'   \cr & =    
  \int           J^\l (w, r-r')  
  \a(s^* (w),r^*)  \frac 1 { f' (s'(w))}        dw   \cr &=
  \int      J^\l (w, r-r')  
           dw    +    (\Gamma^{\l,s}_1)(r) = \bar J^{\l} (r-r') + (\Gamma^{\l,s}_1)(r,r')
   \end {split} \end {equation}
where 
  \begin {equation} \label {lu2ab}    (\Gamma^{\l,s}_1 )(r,r')=   
  \int           J^\l (w, r-r')  
 \left [   \frac 1 {1+ \frac {w   k'(s^*(w))r^*} {  \a(s^*(w),r^*)^2 }} - 1 \right ]        dw .  \end {equation}
Similarly as done in \eqref {lu9} and \eqref {merc3}, \eqref {lu3} and \eqref {lu9a}  we   have 
\begin {equation} \label {lu9g}     |( \Gamma^{\l,s}_1) (r,r') | \le   C \l^2
           \bar J^{\l} (r-r').      \end {equation}
 \end {proof}
 \subsection { Two dimensional  integral operators in    enlarged bounded domains.}
  Let $\TT_\l$  be   the  enlarged cylinder  
 \begin {equation}    \label {lu1x}  \TT_\l = T \times I_\l, \quad   I_\l= [-\frac {d_0} {\l}, \frac {d_0} {\l}]. \end {equation}
Notice that  the circle $T$ is kept unchanged. 
     Denote
$z^*= \frac {z+z'}2$,  $s^*= \frac {s+s'} 2$ and, see \eqref {sm1}, 
  \begin {equation} \label  {sme2a}  \a(s^*,z^*)= 1-  \l k(s^*)z^*, \end {equation}
\begin {equation} \label  {sme1a}  J^c ( s,s',z,z')= \frac 1 \l  J \left  (  \frac {  (s-s')} {\l} \a(s^*,z^*),  (z-z)\right ) \a(s^*,z^*), \end {equation}
    and
     \begin {equation} \label {me7d}(\B  V) (s,z)  =  \int_{\TT_\l}     J^c (s,s', z,z')       V  (s',z')      ds' dz'.         \end {equation} 
    The $J^c$ is obtained by rescaling only the $r-$variable of the integral kernel defining the operator $B^\l$, see \eqref {me7}. 
     Denote  for $V \in L^2(\TT_\l)$  
\begin{equation}   \label{op2t}    (\A V)(s,z) = \frac 1 {\s(m_A(s,z))} V(s,z)-  \int_{\TT_\l}  J^c ( s,s',z,z') V(s',z') ds' dz' , \end {equation}
where $J^c$ is defined in \eqref {sme1a}.   We avoid to write  explicitly the dependence on $\l$ on the previous notations, but the reader should bear in mind
that if not explicitly written,  there is  almost always an  hidden dependence on $\l$.   
In this section we  study the spectrum   of  the integral operator   $\A: L^2(\TT_\l) \to  L^2(\TT_\l)$, defined in   \eqref  {op2t}.
  The  operator  $\A$ is 
  conjugate   to     $L^\l:  L^2(\TT) \to  L^2(\TT)$, hence the spectrum of $\A$ and $L^\l$ are the same, see Theorem \ref {M2}
We have the following result.
\begin {thm} \label {GC2} 

 (0)The operator $ \A   $ is a bounded,   self adjoint  operator on $L^2(\TT_\l) $ with discrete spectrum.
 
 (1)  There exist  $ \mu_0  \in \R $ and  $ \Phi_0 \in L^2 (\TT_\l) $, $\Phi_0 $  strictly positive in $\TT_\l$  
so that 
\begin{equation}
 \label{op2}  \A  \Phi_0=   \mu_0  \Phi_0.  \end{equation} 
The eigenvalue $ \mu_0  $  has multiplicity one and any other point of the spectrum is strictly
bigger than  $\mu_0 $.  

2) There exists $\zeta_1>0$ and $z_1>0$ independent on $\l$ so that 
  \begin{equation} \label{8.4a} \Phi_0 (s,z) \ge  \zeta_1  \qquad |z| \ge z_1.
 \end {equation}

(3) There exists   $C>0$  independent on $\l$ so that 
   \begin{equation} \label{8.2h} -C \l^2   \le  \mu_0 \le      C \l^2.   \end {equation}

\end {thm}
The proof is given at the end of the subsection. 
This  result immediately implies the following.
\begin {thm} \label {M2}   Let    $L^{\l}$ be the operator defined in \eqref {lg1}. The spectrum of $L^\l$ is equal to spectrum of $\A$. In particular we have 
    $$\langle u,  L^{\l} u  \rangle \ge   \mu_0  \|u\|^2_{L^2(\TT)} \ge  - \l^2  C \|u\|^2_{L^2(\TT)}. $$
  \end {thm}
  The thesis  immediately  follows by  Theorem \ref  {GC2} and  noticing that     $L^{\l}$ and   $\A$  are conjugate. The proof is   similar to the one given   in Proposition  \ref  {cec1}.
     The following results are straightforward consequences of Lemma \ref {M2a} and Corollary \ref {lunsoir}.
    
 \begin {prop} \label {GC3}    Let $ V \in L^2(I_\l)$. For any $s\in T $ we have  
 \begin {equation} \label {GC1} (\B V) (z)=  (\bar J\star_{I_\l} V)(z)  +  (\Gamma^{s} V)(z),   \end {equation} 
 \begin {equation} \label {GC2a} \int_{\TT_\l} ds' dz'  J^c ( s,s',z,z')  =1 +  (\Gamma^{s})(z),   \end {equation} 
\begin {equation} \label {GC3b} \int_{\TT_\l} ds' J^c ( s,s',z,z')  = \bar J (z-z')  +  (\Gamma^{s}_1)(z,z'),   \end {equation} 
\begin {equation} \label {GC5} | (\Gamma^{s} V)(z)| \le C \l^2 (\bar J \star_{I_\l} |V|)(z),  \quad    |\Gamma^{s}(z)| \le C \l^2,   \quad  |\Gamma^{s}_1(z,z')| \le C \l^2 \bar J (z-z').  \end {equation} 
\end{prop}
\begin {proof}  Taking into account the definition of $J^c$, see \eqref {sme1a}, applying
similar argument as in   Lemma \ref {M2a} and Corollary \ref {lunsoir}  one gets the statements.
\end {proof}

Denote 
\begin {equation}   \label {luz3} p(s,z)=\s(m_A(s,z)),  \qquad (s,z) \in \TT_\l, \end {equation}
 $$ \HH = \left  \{ u: \int_{\TT_\l} u(s,z)^2 \frac 1 {p(s,z)} dz ds  < \infty \right \}, $$ 
  $$ \llangle V,  U \rrangle  = \int_{\TT_\l} V(s,z) U(s,z) \frac 1 {p(s,z)} dz ds,$$
 $$\|V\|^2_{\HH} =  \int_{\TT_\l} V(s,z)^2   \frac 1 {p(s,z)} dz ds, $$
  and   $\PP$ the linear integral operator acting on functions $V \in \HH$
\begin {equation}   \label {S1.2a}(\PP V) (s,z)=   p(s,z) \int_{\TT_\l}  J^c ( s,s',z,z') V(s',z') ds' dz'.  \end {equation}
By the property of $m_A$ stated in Subsection 2.4,  $  \beta \ge  p(s,z) >a>0$ and $ p \in C^1(\TT_\l)$. 
  \vskip0.5cm
\begin {thm} \label {P-Fb}     The operator  $\PP $ is  a compact, self-adjoint operator  on    $\HH $,
positivity improving.
Further, there exist  $\nu_0>0$ and  $V_0 \in \HH
$, $V_0 $  strictly positive  function,  so that
\begin {equation}   \label {S1.1} \PP V_0 = \nu_0 V_0. \end {equation}
The eigenvalue $ \nu_0$  has multiplicity one and any other point of the spectrum is strictly
inside the ball of radius $\nu_0$. The eigenfunction   $V_0 \in C^1 ( \TT_\l)$.  
Further there exists $z_1>0$ and $ \zeta>0$ independent on $\l$ so that 
\begin {equation}   \label {luz1}  V_0(s,z)  \ge \zeta, \qquad   |z| \le z_1,  \forall s \in T.  \end {equation}
 \end {thm}
 \begin {proof} 
 It is immediate to   see that
 $$ \llangle  \PP V, W \rrangle =  \llangle  V, \PP W \rrangle.$$
The compactness can be shown  by proving that  any bounded set of $\HH$  is mapped by $\PP$  in a relatively compact 
set.   
  To show the positivity improving we  take $V  \ge 0$     in   $\TT$,  $ V \neq 0$,  and show that
  for all $ (s,z) \in \TT_\l$,   
\begin {equation}   \label {St1} (\PP V) (s,z)>0. \end {equation}
 Namely, assume that there exists   $(\bar s, \bar z) \in \TT_\l$ so that  
 $(\PP V) (\bar s, \bar z)=0$.  Since     $J^c \ge 0$ we  have that $ V (s,z) =0 $ for  all $(s,z) \in  Q_\l (\bar s, \bar z)=  \{(s,z)  \in \TT_\l:  |s-\bar s| \le \l, |z-\bar z| \le 1 \}$.  
 Repeating the  same argument for  points $(s,z)$ in  $ Q_\l (\bar s, \bar z)$  we   end  up that  $ V(s,z)=0$  for  $(s,z)\in \TT_\l$, obtaining   a contradiction.  Therefore
  the positivity improving  property is  proven.
  From the positivity  hypothesis on $J^c$,    there exists  an 
integer $n_\l$ such that for $n \ge n_\l $, there is $\z>0$ so that for any $ \underline x=(s,z)$ and $ \underline {\bar x}= (\bar s, \bar z)$  in $ \TT_\l$
\begin {equation}   \label {S1.3}   \int \dha     \underline x_1\dha    \underline x_2....{\rm d} \underline x_n J^c ( \underline x,  \underline x_1) J^c ( \underline x_1, \underline x_2)... J^c ( \underline x_{n},\underline {{\bar x}})  > \z.
\end {equation}  
  Denote for $ \underline x \in\TT$ and $ \underline {\bar  x} \in\TT$
\begin {equation}   \label {S1.4}  K(\underline x ,\underline {\bar  x}) = p( \underline x )  \int \dha     \underline x _1\dha     \underline x _2....{\rm d} \underline x _n p( \underline x_1)  J^c (\underline z ,\underline z _1)p( \underline x_2)  
J^c (\underline x _1,\underline x _2)... p( \underline x_n)  J^c (\underline x _{n},\underline {\bar x }).\end {equation}
Then  one can apply the classical Perron Frobenius Theorem to  the kernel $K (\cdot,\cdot)$.
  As a consequence we have 
that     the maximum   eigenvalue of the spectrum of $ \PP$, which we denote $ \nu_0$,  has   multiplicity one and any other point of the spectrum of $\PP$  is strictly
 smaller than  $\nu_0$.   Let $V_0$ be the   eigenfunction  associated to $\nu_0$. Then it  does not change sign and  
 we assume that it is positive.
 Differentiating the eigenvalue equation, taking into account that $p \in C^1(\TT_\l)$ we get that  $V_0  \in C^1(\TT_\l).$
 Next we show \eqref {luz1}. 
  From  \eqref  {S1.5bt} of Lemma  \ref {S12t}, stated  below,  we have that $ \int_T  V_0(s,z)^2 ds $ is exponentially  decreasing for  $|z| \ge z_0$.
Further $$ \int_{\TT_\l} \frac 1 {p(s,z)} V_0(s,z)^2 ds dz =1.$$ 
Therefore  we must have  that there  
exists $z_1>0$  independent on $\l$ so that 
\begin {equation} \label {ab1} \int_{- z_1}^{z_1}  \frac 1 {p(s,z )}  \int_T  V_0(s,z)^2 ds dz \ge \frac 1 2. \end {equation}
This implies that  there exists $\zeta>0$, independent on $\l$ so that 
$$ V_0(s,z)  \ge \zeta, \qquad   |z| \le z_1,  \forall s \in T.$$
Namely, if this is false, there  exists $(\bar s, \bar z) \in T \times [-z_1,z_1] $ so that 
$V_0(\bar s, \bar z)=0$.  Since $ \PP$ is positivity improving and $V_0$ is an eigenfunction, repeating the argument  done after formula  \eqref {St1}, we get $V_0(s, z)=0$ in $T \times [-z_1,z_1] $.
This is impossible since \eqref {ab1} holds.  \end{proof}
 \vskip0.5cm
\begin {lem} \label {S12t} 
For any $\e_0 \in (0,  \frac {(1-\s(m_\b))} {2})$ there exists   $\l_0= \l_0(\e_0)>0$   so that for $\l \le \l_0$ the following holds.   
        Let  $\nu>1-  \e_0  $    be  an  eigenvalue of the operator
$\PP$ on  $\HH$ 
and   $ \Psi $ be  any  of the corresponding  normalized  eigenfunctions.
 There exists  $z_0=z_0(\e_0) \in I_\l $ independent on $\l$   so that 
  \begin{equation} \label{S1.5at}  |\Psi (s,z)| \le \frac {C} {\sqrt \l}  e^{-\a(\e_0)|z|}  \|\Psi\|_{\HH}  \qquad |z| \ge z_0,\quad  \forall s  \in T,
   \end {equation}
   where $\a(\e_0)$ is given in \eqref {S12.4t}.
   Further 
    \begin{equation} \label{S1.5bt}  \int_T |\Psi (s,z)|^2 ds  \le  C  e^{- \a(\e_0)|z|} \| \Psi \|^2_{\HH}, \qquad |z| \ge z_0.
   \end {equation}

\end {lem}

\begin {proof}    For any  $\e_0$, choose   $ \l_0=\l_0(\e_0)>0$,
  $z_0= z_0(\e_0)>0$,   such that for   for $\l \le \l_0$
 \begin{equation} \label{S12.2t}   p(s,z)   <  1 -2\e_0, \qquad |z| \ge z_0,  \quad  (s,z) \in \TT_\l.   \end {equation}
 This is possible  since  $ |p(s,z)- \s(\bar m (z))| \le C \l $,  $ \lim_{|z| \to \infty}  \s(\bar m (z)) = \s(m_\b)$ and $1 -2\e_0 < \s(m_\b)$.  
 Further  $ \nu>  1-   \e_0 $,  hence,  by \eqref {S12.2t},   we have  for  $|z| \ge z_0$,   
\begin{equation} \label{anc1} \frac 1 {\nu}  p(s,z) <  \frac{1- 2 \e_0} {1- \e_0}. \end {equation}
 Take $ \l_0= \l_0 (\e_0)$ small enough so that  $2z_0  \in I_\l$.
 Take    $z = z_0+n$
where $n$ is any integer so that $z_0+2n \in \TT_\l $.   
By  the eigenvalue equation   we have 
 \begin{equation} \label{S12.2at} \Psi (s,z_0+n) = \frac 1 {\nu} (\PP \Psi ) (s,z_0+n).    \end {equation}
 Iterating  $n$ times \eqref {S12.2at} and    by  \eqref {anc1}     we obtain 
  \begin{equation} \label{S12.3ta}     |\Psi (s,z_0+n )| \le    \left ( \frac{1- 2 \e_0} {1- \e_0} \right )^n  |(J^c)^n \Psi (s, z_0+n)| \le  \left ( \frac{1- 2 \e_0} {1- \e_0} \right )^n   \|(J^c)^n \Psi \|_\infty.
  \end {equation}
  By   Proposition \ref {GC3} we have that
$$  \|J^c\|_1=      \int_{T}    ds' dz'    J^c(s,s',z,z')=
1  +  \G^{s} (z),$$
where $|\G^{s} (z)| \le   C \l^2$.
 Hence, by Jensen inequality we get
$$
 \|(J^c)^n \Psi \|_\infty
 \le  \|J^c\|_2 \|(J^c)^{n-1}\Psi \|_2 \le 
   \frac 1 { \sqrt \l }   C \|  J \|_2  (\|J^c\|_1)^{n-1} \|\Psi \|_{\HH} \le  \frac 1 { \sqrt \l }   C  (1+ C \l^2)^{n-1}  \|  J \|_2    \|\Psi \|_{\HH}, $$
    where we estimated  $ \|J^c\|_2  \le  \frac C {\sqrt \l  }\|J\|_2$.
 Take  $ \l$ small enough so that $\frac 12   \ln \left ( \frac{1-  \e_0} {1- 2\e_0}\right) \ge \ln (1+ C \l^2)$.
   Hence, by \eqref {S12.3ta},   we have 
    \begin{equation} \label{S12.3t}   
  |\Psi (s,z_0+n )| \le       \frac 1 {\sqrt \l  } C e^{-n   \a(\e_0)} \|  J \|_2    \|\Psi \|_{\HH}, \end {equation}
  where   
 \begin{equation} \label{S12.4t}\a(\e_0) =  \frac 12 \ln \left ( \frac{1-  \e_0} {1- 2\e_0}\right)>0.   \end {equation} 
 To get  \eqref  {S1.5bt} we proceed as above obtaining
 \begin{equation} \label{S12.bt}  
  \int_T ds  |\Psi (s,z_0+n )|^2   \le \left ( \frac{1- 2 \e_0} {1- \e_0}\right )^{2n}   \int_T ds |(J^c)^n \Psi(s,z_0+n)|^2. 
\end {equation}
By  Proposition \ref {GC3}
we have that 
$$       \int_{T}    ds'     J^c(s,s',z,z')=
   \bar J (z-z') +   \G^{s}_1(z,z') $$
where $| \G^{s}_1(z,z')|  \le C \l^2 \bar J  (z-z').$
 Set $\bar z = z_0+n$,    by  Jensen  inequality  we have 
 \begin{equation} \label{S12.ct}  \begin {split}  &
 \int_T ds |(J^c)^n \Psi(s,z_0+n)|^2  =    \int_T ds  \left ( \int_{\TT_\l} ds' dz'    J^c(s,s', \bar z ,z')  \left | (J^c)^{n-1} \Psi(s',z') \right | \right )^2 \cr & \le  (1+C\l^2)
   \int_{\TT_\l} ds' dz'  \left ( \int_T ds J^c(s,s', \bar z,z') \right)  \left |  (J^c)^{n-1} \Psi(s',z') \right |^2  \cr &\le 
(1+C\l^2)  \int_{\TT_\l} ds' dz'   
\bar J (\bar z-z')   \left |  (J^c)^{n-1} \Psi(s',z') \right |^2  +  C\l^2  \int_{\TT_\l} ds' dz'   
\bar J (\bar z-z')   \left |  (J^c)^{n-1} \Psi(s',z') \right |^2  \cr &  =
(1+2C\l^2)  \int_{\TT_\l} ds' dz'   
\bar J (\bar z-z')   \left |  (J^c)^{n-1} \Psi(s',z') \right |^2  \le 
  (1+2C\l^2) \sup_z \bar J(z)  \int_{\TT_\l} ds' dz'    \left |  (J^c)^{n-1} \Psi(s',z') \right |^2 \cr & \le
 (1+2C\l^2)^n \sup_z \bar J(z) \| \Psi \|^2_{\HH} \le C (1+2C\l^2)^n \| \Psi \|^2_{\HH}.    \end {split}
\end {equation}
Therefore by \eqref {S12.bt}, \eqref {S12.ct} and \eqref {S12.4t} 
\begin{equation} \label{S12.bt1}  
  \int_T ds  |\Psi (s,z_0+n )|^2   \le C \left (  \frac{1- 2 \e_0} {1- \e_0} \right )^{2n}    (1+2C\l^2)^n \| \Psi \|^2_{\HH} \le   C e^{- 4 n \a(\e_0)}  e^{n  \ln (1+2C\l^2)}\| \Psi \|^2_{\HH}. \end {equation}
 Take $ \l$ small enough so that $  3 \a(\e_0) \ge \ln {(1+2C \l^2)}$, we then obtain \eqref {S1.5bt}.

  \end {proof}

      {\bf {Proof of Theorem \ref {GC2}}} 
The  points  (0),  (1)  and (2)   are an immediate consequence of  Theorem \ref {P-Fb}. 
Namely, let $ T: \HH  \to L^2(\TT_\l) $  so that $ TV= \frac {V } {\sqrt { p}}$.
The map $T$ is an   isometry and  the operator  $\1- \PP: \HH \to \HH$  is therefore
conjugate to the operator $\A:  L^2(\TT_\l) \to L^2(\TT_\l)$.  This means that the spectrum of the two operators are equal,
moreover    if $V$ is an eigenfunction of $\PP$, then $T V$ is an eigenfunction of $ \A$. 
Next we show  the upper bound in \eqref {8.2h}. 
  By the variational   form for the eigenvalues
 $$ \mu_0= \inf_{V \in L^2(\TT_\l), \|V\|=1}  \langle  \A V,V \rangle \le \langle \A \bar V, \bar V \rangle, $$
 where $\bar V  = \frac {\bar m' (z)} {   \|\bar m'\|_{L^2 (\TT_\l)}}$.
  Since $\bar V$ is a function of only the $z-$variable, by Proposition \ref {GC3},
$$ \A \bar V = \LL^s \bar V + \G^{s} \bar V.$$
Hence 
 $$   \mu_0 \le \langle \A \bar V, \bar V \rangle =   \int_T ds \langle \bar V (s, \cdot),   \LL^s \bar V (s, \cdot) \rangle_s +\int_T ds \langle \bar V (s, \cdot),  \G^{s} \bar V (s, \cdot) \rangle_s. $$
 By  \eqref {8.25}
  $$  \int_T ds \langle \bar V (s, \cdot),   \LL^s \bar V (s, \cdot) \rangle_s \le  C \l^2, $$
and by \eqref {GC5}
 $$\int_T ds \langle \bar V (s, \cdot),  \G^{s} \bar V (s, \cdot) \rangle_s \le  C \l^2. $$
    Therefore
\begin{equation}    \label {8.25t} \mu_0 \le  C \l^2.   \end {equation}
Let   $\Phi_0$ be the  normalized positive eigenfunction associated to $ \mu_0$.
   Multiply  the eigenvalue equation  by $ \Psi_1$, see \eqref {gc3a} the principal  eigenvalue of $\LL^{s}$, defined in \eqref {8.100}. We have, since $\A $  is self-adjoint
$$  \mu_0 \langle \Phi_0,  \Psi_1 \rangle =  \langle \A \Phi_0, \Psi_1 \rangle =  \langle  \Phi_0, \A \psi_1 \rangle .$$
By  Taylor formula  $ \Psi_1 (s',z)=  \Psi_1 (s,z) + (s-s') \nabla_s  \Psi_1 (\tilde s,z) $, where $\tilde s \in (s,s')$.
By  Proposition  \ref {GC3}   
\begin{equation}    \label {gt}   \begin {split} & \A \Psi_1 = 
\frac {\Psi_1 } {\s(m_A)} - \int_{\TT_\l}     J^c(s,s', z, z')     \Psi_1  (s',z')      ds' dz' 
\cr & =  \frac { \Psi_1  } {\s(m_A)} - \int_{\TT_\l}     J^c(s,s', z, z')     \Psi_1  (s,z')      ds' dz'    -
 \int_{\TT_\l}     J^c(s,s', z, z')   (s-s') \nabla_s  \Psi_1 (\tilde s,z')       ds' dz'  \cr &=
(\LL^{s}\Psi_1) (s,z)   + \G^{s} \Psi_1 - \int_{\TT_\l}     J^c(s,s', z, z')    (s-s') \nabla_s  \Psi_1 (\tilde s,z')       ds' dz'.
\end {split} \end {equation}
Therefore
\begin{equation}    \label {gth1}   \begin {split} &   \mu_0 \langle   \Phi_0,  \Psi_1 \rangle  =  \langle  \Phi_0, \LL^s \Psi_1 \rangle +
 \langle  \Phi_0, \G^{s} \Psi_1 \rangle \cr & + \int_{\TT_\l} ds dz  \Phi_0(s,z)  \int_{\TT_\l}      J^c(s,s', z, z')    (s-s') \nabla_s  \Psi_1 (\tilde s,z')       ds' dz'. \end {split} \end {equation}
By \eqref {E.9}  and \eqref {MG2}  we estimate 
$$ \left | \int_{\TT_\l} ds dz  \Phi_0(s,z)  \int_{\TT_\l}      J^c(s,s', z, z')  (s-s') \nabla_s  \Psi_1 (\tilde s,z')       ds' dz'  \right |  \le C\l  \|\Phi_0 \| \|\nabla_s \Psi_1 \| \le C \l  \|\nabla_s m_A\|  \le C \l^2,$$
$$ \langle  \Phi_0, \G^{s} \Psi_1 \rangle \le C\l^2.$$
  Therefore, 
we have 
 $$     \mu_0 \langle \Phi_0, \Psi_1 \rangle   \ge    \inf_{s}  \mu_1^s      \langle \Phi_0, \Psi_1 \rangle     -
 C\l^2.$$ 
 We need to show that  
 \begin{equation}    \label {luz2}  \langle \Phi_0,  \Psi_1 \rangle \ge C. \end {equation}
By  \eqref  {8.k1}   and \eqref {8.4a}   there exist  $z_1>0$ and $\z_1>0$  independent on $\l$  so that 
$ \Phi_0(s,z) \Psi_1 (s,z) \ge  \z_1 $ for     $|z| \le z_1$  and  for all $s \in T$. 
Therefore 
      \eqref {luz2} holds  and 
we obtain 
$$  \mu_0  \ge - C \l^2. $$
  
\qed

\subsection { Proof  of Theorem \ref {82}.}  
As explained in the introduction, we would like to  take advantage of \eqref {sc1}, splitting the quadratic form associated to the operator $A^\l_{m_A}$ in two integrals.   One     integral  is  over   the region $\Om \setminus   \NN (\frac {d_0} 2)$  and because of  \eqref {sc1} is   positive. The second integral  is over the region $ \NN (\frac {d_0} 2)$, i.e near the surface $\G$ and  we can estimate it  from below by applying     Lemma \ref {F2} and Lemma  \ref {M2}.    But   because of the non locality  of the operator this argument  does not  work.  Namely  when  splitting the  integral of the quadratic form  there is an extra term    which might spoil the estimate. 
 Here,   we show that  it is always possible to find  a way  to split the integral  of  the quadratic form associated to the operator $A^\l_{m_A}$  to obtain the desired estimate. 

Define   for any integer $k \in  \{ 0, \dots, N \}$,  for  $ N= [ \frac 1 \l ]$ where  for $x \in \R$, $ [x]$ is the integer part of $x$,    a sequences of  cut off functions
\begin {equation} \label {sc2}  \eta^{k}_1 (\xi) = \left \{ \begin {split} & 1  \qquad \hbox {when} \quad  \xi \in    \NN (\frac {d_0} 2 ( 1 + \l k ))\cr &
0 \quad \hbox {otherwise}, \end {split} \right .  \end {equation}
 $ \eta^{k}_2 (\xi)= 1 -  \eta^{k}_1 (\xi)$    
 and set 
$$s_k=  2 \int_{ \Om} \dha    \xi  \eta^{k}_1 (\xi) v  (\xi) (J^{\l} \star  \eta^{k}_2 v)(\xi).$$
Let $k=0$.  
We  have, taking into account that  $\eta^{0}_2 (\xi)\eta^{0}_1 (\xi)=0$ for $\xi \in \Om$ and the symmetry of $J^\l (\cdot)$
\begin {equation} \label {gch2}  \begin {split}     \int_{ \Om}  ( A^\l_{m_A} v)  (\xi) v(\xi) \dha    \xi &   = \int_{ \Om}  ( A^\l_{m_A} \eta^{0}_1 v)  (\xi)  \eta^{0}_1 (\xi) v(\xi) \dha    \xi   \cr & +
 \int_{ \Om}  ( A^\l_{m_A} \eta^{0}_2 v)  (\xi)  \eta^{0}_2 (\xi) v(\xi) \dha    \xi     -
 s_0
\end {split}
\end {equation}
Because of \eqref  {sc1} 
$$ \int_{ \Om}  ( A^\l_{m_A} \eta^{0}_2 v)  (\xi)  \eta^{0}_2 (\xi) v(\xi) \dha    \xi   \ge (C^*-1) \| \eta^{0}_2   v\|^2_{L^2 (\Om)} >0.$$ 
 By  Lemma \ref {F2} and Lemma \ref {M2} we have that 
  $$  \int_{\Om }  (A^\l_{m_A}  \eta^{0}_1v ) (\xi)  \eta^{0}_1 v(\xi)\dha    \xi  \ge - C \l^2 \|v\|^2_{L^2(\Om)}. $$
But  the last term of \eqref  {gch2}  might create problems.
Obviously if  
\begin {equation} \label {sc3} s_0<0  \end {equation}
or  if
\begin {equation} \label {sc4}  s_0 \le \d^*    \|  \eta^{0}_2  v \|^2_{L^2(\Om)} \end {equation}
 for $\d^*>0$ so that $  (C^*-1 -\d^*)>0$  or if 
\begin {equation} \label {sc5}  s_0 \le  C \l^2\|v\|^2_{L^2(\Om)} \end {equation}
then  the theorem would be proven.  But this might not be the case.
   We proceed recursively as following.
Assume that  no one of the three  conditions \eqref {sc3}, \eqref {sc4}, \eqref {sc5} hold.
Take $\d>0$ so that  $2 \frac {\d} {1-\d}\le \d^*$  
Notice that since we are assuming that  \eqref  {sc3} does not hold, $s_0>0$.
We must have 
\begin {equation} \label {sc6}  s_0 >  \d \left [s_0 +   2 \|  \eta^{0}_2  v \|^2_{L^2(\Om)}  \right ]. \end {equation}
Namely if the reverse inequality holds in \eqref {sc6} then
\begin {equation} \label {sc7}  s_0 \le  2\frac {\d} {1-\d}       \|  \eta^{0}_2  v \|^2_{L^2(\Om)}  \le \d^*     \|  \eta^{0}_2  v \|^2_{L^2(\Om)}, \end {equation}
which is   \eqref  {sc4}. But we are assuming that  \eqref  {sc4}  does not hold.
Notice that the integral defining $s_0$  has support  in a stripe of width $2 \l$ around
$\G_{ \pm \frac {d_0} 2}= \{ \xi \in \Om: r(\xi, \G) = \pm  \frac {d_0} 2 \}$. So we might try    to split the integral of the quadratic form  
moving    by $\l$ from  $\G_{\pm \frac {d_0} 2}$, i.e  applying  the cut off function  $\eta^{1}_1$.
  If  one of the following conditions holds
 \begin {equation} \label {sc8} 
 \left \{ \begin {split} & s_1 \le 0, \cr & s_1 \le  \d^*   \| \eta^{1}_2  v  \|^2_{L^2(\Om)}, \cr &
 s_1 \le \l^2 \|v\|^2_{L^2(\Om)}, \end {split} \right.  \end {equation}
  we  can conclude the proof of the theorem.
If not then
\begin {equation} \label {sc12}  s_1 >  \d \left [s_1 +     2 \| \eta^{1}_2  v  \|^2_{L^2(\Om)} \right ] . \end {equation}
By  \eqref {sc6}, 
\begin {equation} \label {sc11}  s_1    =
 \sum_{i=0}^1 s_i    - s_0   
    \le   \sum_{i=0}^1 s_i    -    \d \left [s_0 +   2 \|  \eta^{0}_2  v \|^2_{L^2(\Om)}  \right ]. 
   \end {equation}
 Notice that
\begin {equation} \label {sc13} s_1= 2 \int_{ \Om} \dha    \xi  \eta^{1}_1 (\xi) v  (\xi) (J^{\l} \star  \eta^{1}_2 v)(\xi)   \le   2 \int_{ \Om} \dha    \xi  \eta^{0}_2 (\xi) |v  (\xi)| | (J^{\l} \star  \eta^{0}_2 v)(\xi)|   \le   2\|  \eta^{0}_2  v \|^2_{L^2(\Om)}.   \end {equation}
Therefore
\begin {equation} \label {sc14}  s_1  
  \le  (1-\delta) \  \sum_{i=0}^1 s_i. 
    \end {equation}
 Next we show that there exists $ \bar k \in \{0, \dots, N\}$ so that  if   for $j \in \{0, \dots, \bar k-1\}$ 
 no one of the following conditions  is satisfied 
 \begin {equation} \label {sc8} 
 \left \{ \begin {split} & s_j \le 0 \cr & s_j \le  \d^*    \| \eta^{j}_2  v  \|^2_{L^2(\Om)} \cr &
 s_j \le \l^2 \|v\|^2_{L^2(\Om)} \end {split} \right.  \end {equation}
    then 
$$  s_{\bar k} \le \l^2 \|v\|^2_{L^2(\Om)}. $$
 Namely, reiterating the argument done in the case of $s_1$ we have that for $k \in  \{0, \dots, N\}$
 $$ s_{  k}    \le (1-\d)^{k}   \sum_{i=0}^{ k} s_i . $$
Denote $\d_0= -\log (1-\d) >0$.   Take    $\l$   small enough   so that  $ \frac {1 } {\d_0} \log  \frac  1 {\l^2} <  [\frac 1 \l ] = N$    and set $\bar k = [\frac 1 {\d_0} \log \frac  1 {\l^2}] $.    With such a choice $\bar k < N$.
We have
  \begin {equation}  \begin {split} &  s_{\bar k}    \le  e^{-\d_0 \bar k}     \sum_{i=0}^{\bar k} s_i   \cr &   \le  \l^2      \sum_{i=0}^{\bar k} s_i     \le  C \l^2 \|v\|^2_{L^2(\Om)}. 
 \end {split}  \end {equation}
We then split
\begin {equation} \label {gch20}  \begin {split}     \int_{ \Om}  ( A^\l_{m_A} v)  (\xi) v(\xi) \dha    \xi &   = \int_{ \Om}  ( A^\l_{m_A} \eta^{\bar k }_1 v)  (\xi)  \eta^{\bar k}_1 (\xi) v(\xi) \dha    \xi   \cr & +
 \int_{ \Om}  ( A^\l_{m_A} \eta^{\bar k }_2 v)  (\xi)  \eta^{\bar k}_2 (\xi) v(\xi) \dha    \xi     - s_{\bar k}. 
\end {split}
\end {equation}
Because of \eqref  {sc1} 
$$ \int_{ \Om}  ( A^\l_{m_A} \eta^{\bar k }_2 v)  (\xi)  \eta^{\bar k }_2 (\xi) v(\xi) \dha    \xi   \ge (C^*-1) \| \eta^{\bar k }_2   v\|_{L^2 (\Om)} >0.$$ 
By  Lemma \ref {F2} and Lemma \ref {M2} we have that 
  $$  \int_{\Om }  (A^\l_{m_A}  \eta^{\bar k }_1v ) (\xi)  \eta^{\bar k}_1 v(\xi)\dha    \xi  \ge - C \l^2  \| v\|^2_{L^2 (\Om)}, $$
and
$$  s_{\bar k} \le C \l^2 \| v\|^2_{L^2 (\Om)}. $$
Theorem is proved.
   \qed 
             
\section { Two dimensional  convolution operators in  enlarged  bounded domains.}
Essential ingredient to show   the $H^{-1}-$ estimate, stated in Theorem  \ref {86}, is the knowledge of the spectrum of the 
  operator defined below,   see \eqref   {lg1a}. 
        For  $ V \in    L^2 (\TT_\l)$, where $ \TT_\l$  is the the  enlarged cylinder  defined in \eqref {lu1x}, denote
 \begin {equation} \label {lg1a} (\GG^{\l}  V ) (s,z) = \frac {  V   (s,z)} {\s( \bar m(z))} -  \int_{\TT_\l} \frac 1 \l J (\frac {(s-s')} \l, z-z') \  V(s',z') dz' ds', \end {equation}
 where  $J$   is  the  symmetric  probability density   on $ \R^2$ defined in Subsection 2.1.
We study the spectrum of the operator  $ \GG^{\l}$ on  $ L^2 (\TT_\l)$  by Fourier analysis.
For $h \in \R$,  let $J^h (\cdot)$ be  the $h$ component of the  Fourier transform of $J(s, \cdot)$,  $s  \in \R$:
\begin {equation} \label {aa3}J^h   (z) = \int_\R    J(s, z) e^{ih s} ds=   \int_\R    J(s, z)  \cos (h s) ds. \end {equation}
The last identity holds because $J$ is an even function of $s$. This implies that  $J^h = J^{-h}$. 
  Since   $ J  \in C^1(\R^2)$    we have that  
   \begin {equation}    \label {scu3}  |J^{h} (z)|   \le  \frac {C(z)} {(1+ |h|)},  \end {equation}
   where  $C(z)=0$ when $|z| >1$ and 
  \begin {equation} \label {scu5} C(z)= \int  ds |\frac {d} {ds}  J(s,z)|, \qquad \hbox {for} \quad  |z| \le 1.  \end {equation}
 Further 
\begin {equation} \label {ab3}J^{0}  (z) = \int   J( s, z)   d s = \bar J(z),  \end {equation}
see  Subsection 2.1. 
    For $w \in L^2 (I_\l)$ denote   
\begin {equation} \label {aa2} (\LL^{h} w) (z)=  \left [ \frac {w (z)} {\s (\bar m (z))}-   (J^{h} \star_{I_\l} w ) (z)  \right ].  \end {equation}
\begin {prop}  \label {st1}  The operator  $\LL^{h}$, $h \in \R$,  is a bounded,  self-adjoint operator  on $L^2 (I_\l)$.
 The spectrum of  $\LL^{h}$ is discrete.
   \end{prop}
    \begin {proof} The proof is straightforward.   One needs to exploit that   $  \PP^h w= \s (\bar m) J^{h} \star w $  is a bounded  integral operator in  $L^2 (I_\l)$ and that  $\LL^{h}$ is conjugate to  $\1-\PP^h$.
   \end {proof}
By general arguments  one can deduce   informations about the spectrum of  $ \GG^{\l}$  by the knowledge of the spectrum of  $ \LL^{h} $. 
Any  $ V \in L^2(\TT_\l)$   can be    expanded  in Fourier complex series    as the following  
\begin {equation} \label {aa1} V (s,z) = \sum_{k \in \Z_L} e^{i k  s}  u_k(z)  \end {equation}
where   $ \Z_L= \frac {2 \pi} L \Z$ and $u_k(z) =  \frac 1 L \int_T V(s,z) e^{- i   k   s} ds $.  
Let $ \FF: L^2(\TT_\l) \to \bigoplus_{k\in \Z_L}  H_k $, $H_k = L^2 (I_\l) $,  be the isometry  induced by the Fourier expansion.   We denote by $W= (u_k)_{k \in \Z_L}  $ an element of $   \bigoplus_{k\in \Z_L}  H_k$.

\begin {thm} \label {spo1}  Let $ \s (  \GG^{\l}) $  be the spectrum of $  \GG^{\l} $ in  $L^2 (\TT_\l)$ and 
$ \s (\LL^{h}) $ the spectrum of $ \LL^{h}$ in $L^2 (I_\l)$. 
Then
$$  \s (  \GG^{\l}) =   \bigcup_{ \{k\in \Z_L \}}   \s (\LL^{  \l k }).$$
 \end {thm}
\begin {proof}
Let   $ \tilde \GG^{\l}= \bigoplus_{k\in \Z_L}  \LL^{ \l   k}$ be the operator
defined on $\bigoplus_{k\in \Z_L}  H_k$ so that  $ \tilde \GG^{\l} W =   (\LL^{ \l    k} u_k)_{k \in \Z_L}$.
We have that 
$$\GG^{\l}= \FF^{-1} \tilde \GG^\l  \FF.$$  Being conjugate  $\tilde \GG^\l  $ and $\GG^\l$    have the same spectrum.
By   \cite {CI}, $ \s ( \tilde  \GG^{\l}) =   \bigcup_{ \{k\in \Z_L \}}   \s (\LL^{  \l    k })$.  The  thesis of theorem follows. 
\end {proof}
 The aim is then to      study the spectrum of  $\LL^{h}$, $ h \in R$.   It turns out that  when $|h| > h_0$, where $h_0$  is   a positive real number, conveniently chosen,  the spectrum of
 $\LL^{h}$ is strictly positive and can be lower bounded by a  positive constant  depending on $h_0$ but not on $h$.  For  $|h| \le h_0$    the spectrum  of  $\LL^{h}$ is still positive  but  the lower bound does depend on $h$. 
 In this case, we are able to give upper and lower bound of the principal eigenvalue of $\LL^{h}$ which turns out to be  very useful.
 We analyse these type of behaviour in   Proposition \ref {tu2}  and Proposition \ref {tu1a}.
Next, we show that  eigenfunctions associated to small eigenvalues decay exponentially for $|z|$ large enough. This result is valid for all $ \{ \LL^h\}_h$. 
     
   \begin {prop} \label {a1}  For any $ \e_0 \in (0, \frac {(1- \s (m_\b))} 2 )$,   there exists   $z_0>0$,  $ \l_0= \l_0 (\e_0)$ and $\a(\e_0)>0$   so that for   $ \l \le   \l_0$     the following holds. 
   Let   $ \mu \le \frac {\e_0} 2 $  be an eigenvalue  of   $\LL^{h}$    and   $\psi $ be any of the corresponding  eigenfunctions. There is $\a(\e_0)>0$ and $z_0= z_0(\e_0)$  independent on $h$   and $\l$ so that  
  \begin{equation} \label{S1.5a}  |\psi (z)| \le C e^{-\a(\e_0)  |z|} \qquad |z| \ge z_0.
   \end {equation}
     \end {prop}
  The proof   is similar to the one given in \cite {O1}[ Lemma 3.5] and it is therefore omitted.    
  \begin {rem} \label {a2} Notice that $z_0$ and   $\a(\e_0)$   depend only on $ \e_0$ and not on $h$.
   Applying  the argument as in   \cite {O1}[ Lemma 3.5]  one ends up with  $ |\psi (z)| \le \beta \|J^h\|_2 e^{-\a(\e_0)  |z|}$.
   But it is immediate to see that   $ \|J^h\|_2 \le  \|J^0\|_2$.
   \end {rem}
The  spectrum of the operator $\LL^0$, i.e. when $h=0$ has been studied in \cite {O1}.
    When $ |h| \le \frac {\pi} 2$ the integral kernel $J^h$ is positivity improving, i.e  if  $v(z)\ge 0$ and $v(z) \neq 0$  for $z \in I_\l$,   then $  \int _{I_\l} dz'J^h(z-z') v(z') >0$.
Hence we   could  apply  the same type of arguments  used  in   \cite {O1} to study the  spectrum of the operator  $ \LL^{h}$ when $ |h| \le \frac {\pi} 2$.     
In Proposition \ref {tu2} we  summarise  the results for the spectrum of  $ \LL^{h}$, $ |h| \le h_0$, where $  h_0 \le \frac {\pi} 2$ is suitable chosen.  To   prove  a uniform (in   $|h| \le h_0$   and $\l>0$) lower bound for the gap of  $ \LL^{h}$  we apply    perturbation theory.

\begin {prop} \label {cap1}
  Let  $J^{h} $,   $ |h| \le \frac {\pi} 2$,
 be the  integral kernel  defined in   \eqref  {aa3}.  We  have
  that  there exists  $h_0>0$       so that for $|h| \le h_0$, 
\begin {equation} \label {cap2}  \bar J (z) -  \frac 12 h^2 \bar J_{tan} (z)    \le J^{h} (z)    \le \bar J (z) -  \frac 14 h^2 \bar J^{\l}_{tan} (z) ,  \end {equation}
 where $ \bar J_{tan} (\cdot)$ is defined in \eqref {t8}, 
   $$0\le \bar J_{tan} (z) <  \bar J (z).$$
   \end {prop}
\begin {proof} 
      By Taylor expanding   $\cos  h \xi$ we have   
       $$   1 - \frac 12 h^2  \xi^2   + \frac 1 {4!} h^4  \ge    \cos h \xi  \ge  1 - \frac 12h^2  \xi^2. $$
 Denote
  \begin {equation} \label {t8} \bar J _{tan} (z)= \int   J(\xi,z)  \xi^2 d\xi. \end {equation}
 We have  
 $$   \bar J(z) - \frac 12 h^2    \bar J_{tan} (z)  + 
 \frac 1 {4!}  h^4  \int  J (\xi,z)  \xi^4   d\xi   \ge J^{h} (z) \ge  \bar J(z) - \frac 12 h^2   \bar J_{tan} (z).  
 $$
Since 
$$  \int J (\xi,z)  \xi^4  d\xi \le     \bar J_{tan} (z) $$ 
taking  $h_0>0 $ so that when   $  |h| \le h_0 $,    $  \frac 1 {4!}   h^2 \le \frac 1  4 $ we obtain 
 $$ \bar J (z) -  \frac 12 h^2 \bar J_{tan} (z)    \le J^{h} (z)    \le \bar J (z) -  \frac 14 h^2 \bar J^{\l}_{tan} (z).   $$
 Therefore  we get  \eqref {cap2}.

 \end {proof}

\begin {prop}  \label {tu2}   There exists   $ h_0 \in (0, \frac \pi 2)$  so that   for $ |h|   \le  h_0$, the following holds for       $\LL^{h}$  defined in   \eqref  {aa2} on   $L^2 (I_\l)$.
\begin {enumerate}
\item   
 There exists $\mu^{h}_0  \in \R$  and $ \psi^{h}_0$  strictly positive  in $I_\l$,  $ \psi^{h}_0$ even function,  so that
 $$ \LL^{h} \psi^{h}_0= \mu^{h}_0  \psi^{h}_0.$$
 The eigenvalue $\mu^{h}_0 $ has multiplicity one and any other eigenvalue is strictly bigger that   $\mu^{h}_0$.
 \item Let  $\mu^0_0$ be  the  principal eigenvalue  of  the operator $ \LL^0$. We have that 
 $$  \mu^{0}_0 < \mu^{h}_0  \le  \mu^{0}_0 + \frac 12 h^2.$$
   \item  There exists $D>0$ independent on  $\l$ and $h$  so that 
 \begin {equation} \label {tue3} \inf_{ \|\psi\|=1, \langle \psi, \psi^{h}_0 \rangle =0}  \langle  \LL^{h}  \psi,  \psi\rangle \ge D, \qquad \forall h:  |h|   \le  h_0. \end {equation} 
\item  The principal eigenvector $ \psi^{h}_0$  is such that
\begin {equation} \label {scu2}  \| \psi^{h}_0-\psi^{0}_0 \|^2 \le  C  h^2.   \end {equation}
 \item   There exists $z_0>0$ and $ \z_0>0$ independent on $ h$ and $ \l $ so that
    $$ \psi^h_0(z) \ge \z_0, \qquad |z| \le z_0.$$
    \item   There exists $C_0>0$,  independent on $\l$ and $h$,  so that 
 \begin {equation} \label {do8} \mu_0^h  \ge    \mu^0_0 +   C_0 h^2.  \end {equation} 
\end {enumerate}
  \end {prop}
  \begin {proof}
For  $ |h| \le \frac {\pi} 2$, for $|s| \le 1$,   the integral kernel  $J^{h}$, in the definition of   $\LL^{h}$,  is non negative  for $z \in I_\l$.
    Applying the Perron Frobenius  Theorem to   the operator $(\A^h g) (z)= \sigma (\bar m) (J^{h} \star g)(z)$,  $z \in I_\l$  and proceeding as   in  \cite {O1}  [Theorem  2.1] we prove point (1).
 To show (2) and (3)  we apply   standard perturbation  theory for bounded selfadjoint operators,  see \cite {Kato}. 
 Define the following family of operators indexed by $\nu$:
  $$  A_\nu =  \LL^{0} + \nu   B, \qquad \nu \in [0,1]$$
  where 
  $$ B=   \LL^{h}- \LL^{0}.$$
  The family $  A_\nu$ connects in a smooth way the unperturbed operator  $\LL^{0} $ to  $\LL^{h}$.
 We have that 
 $$Bw(z)=   \left [ \LL^{h}- \LL^{0}  \right]  w(z) =     \int J( \xi, z-z') [ 1-\cos ( h \xi)]   d\xi w(z') dz'. $$ 
 Notice that for $|h| \le h_0$,  $B$ leaves invariant the cone of the   positive functions.
  Further, by \eqref {cap2},  
   \begin {equation}  \label {scu10} \|   B  \| = \sup_{\{\|w\|=1\}} \langle w,  \left [ \LL^{h}- \LL^{0}  \right]  w\rangle \le  \frac 12 h^2.   \end {equation} 
 Since $\LL^0$ has an isolated simple eigenvalue and a spectral gap $D$, independent on $\lambda$, see Theorem \ref {81},  the $\LL^{h}$ for all   $$ |h| \le \sqrt { \frac D 3 } \equiv h_0 $$ will have an isolated simple eigenvalue  and a spectral gap bigger or equal of  $D/4$. Moreover the principal eigenvalue  $\mu_0^\nu$ and eigenvector $\psi_0^\nu $ of  $ A_\nu$ are analytic in $\nu$.  
By Perron Frobenius Theorem   $\psi_0^\nu >0$ and we assume   $ \langle \psi_0^\nu,    \psi_0^\nu  \rangle=1$. 
 Next we would like to show that that $\mu^0_h > \mu^0_0$.
   We derive with respect to $\nu$ the eigenvalue equation
$$ A_\nu \psi_0^\nu= \mu_0^\nu  \psi_0^\nu.$$
We have 
$$ B \psi_0^\nu + A_\nu \partial_\nu( \psi_0^\nu) =  \partial_\nu ( \mu_0^\nu)  \psi_0^\nu +  \mu_0^\nu \partial_\nu  ( \psi_0^\nu),$$
$$ \langle \psi_0^\nu,  \left [ B \psi_0^\nu + A_\nu \partial_\nu( \psi_0^\nu)\right ]  \rangle  =   \langle \psi_0^\nu,  \left [ \partial_\nu ( \mu_0^\nu)  \psi_0^\nu +  \mu_0^\nu \partial_\nu  ( \psi_0^\nu))\right ]  \rangle.$$
  Therefore
  $$ \langle \psi_0^\nu,   B \psi_0^\nu \rangle + \mu_0^\nu \langle \psi_0^\nu, \partial_\nu( \psi_0^\nu) \rangle =  \partial_\nu ( \mu_0^\nu) \langle \psi_0^\nu,     \psi_0^\nu \rangle+  \mu_0^\nu \langle \psi_0^\nu,  \partial_\nu  ( \psi_0^\nu) \rangle .$$
   We have
  $$ \langle \psi_0^\nu,   B \psi_0^\nu \rangle = \partial_\nu ( \mu_0^\nu).$$
   Hence 
   \begin{equation} \label{a6}     \mu_0^\nu =  \mu^0_0 + \int_0^{\nu}   \langle \psi_0^{\nu'},  B \psi_0^{\nu'} \rangle  d \nu'.   \end {equation}
   Since $\psi_0^\nu>0$ and  $B$  is a positive operator  $\mu_0^\nu >  \mu^0_0$. 
By   \eqref   {scu10},  
  $$  \mu_0^\nu     \le    \mu^0_0 + \nu \frac 12  h^2. $$
  When $ \nu =1$,  $ \A_{\{\nu=1\}}= \LL^h$, $ \mu_0^1= \mu_0^h$  and  we have
  $$ \mu_0^h  \le    \mu^0_0 + \frac 12 h^2. $$
   Next we show \eqref {scu2}.   
For $h \neq 0$,   split 
\begin{equation} \label {sto1} \psi_0^0 =   a \psi_0^h  + (\psi_0^h )^{\perp}.\end {equation}   Then
\begin{equation} \label {8.31}  a^2  +  \|(\psi_0^h )^{\perp} \|^2   =1  \end {equation} 
\begin{equation} \label {8.30}    \langle \LL^h   \psi_0^0,  \psi_0^0\rangle = a^2 \mu_0^h +  \langle \LL^h   (\psi_0^h )^{\perp},
 (\psi_0^h )^{\perp} \rangle \ge a^2 \mu_0^h +\frac  D 4  \|(\psi_0^h )^{\perp} \|^2. \end {equation} 
Further    
\begin{equation}  \begin {split} &  \langle \LL^h   \psi_0^0,  \psi_0^0\rangle  =  \langle \LL^0  \psi_0^0, \psi_0^0\rangle  +
 \langle (\LL^h- \LL^0) \psi_0^0, \psi_0^0 \rangle  \cr  & \le 
 \mu_0^0 +   \frac 12 h^2. 
\end {split}  \end {equation} 
By \eqref  {8.31},  \eqref  {8.30} and  $ \mu_0^h >0$ we have 
that 
$$ \mu_0^0 +  \frac 12 h^2   \ge   a^2  \mu_0^h + \frac D 4  \|(\psi_0^h )^{\perp} \|^2 \ge   \frac D 4  \|(\psi_0^h )^{\perp} \|^2.  $$
By \cite {O1} [Theorem 2.2, formula (2.8)], $ 0 \le \mu_0^0  \le C e^{-2 \a \frac 1 \l } $.
It follows      that  there exists   $C>0$ independent on $ \l $ and $h$  so that 
\begin{equation} \label {8.31a} \|(\psi_0^h )^{\perp} \|^2   \le  C h^2.  \end {equation}
By  \eqref  {8.31}  $a^2= 1-    \|(\psi_0^h )^{\perp} \|^2$. This, together with decomposition \eqref {sto1} and \eqref {8.31a}
implies  \eqref {scu2}.

The proof of the point (5) can be done as in   \cite {O1} [Lemma 3.6,  formula  (3.22)].  The proof is similar to the one given in Theorem \ref {83} when proving \eqref {8.k1}.  

 To  show (6)    we need to lower bound,  see  \eqref {a6},    $\int_0^{1}   \langle \psi^0_{\nu'},  B \psi^0_{\nu'} \rangle  d \nu'$.
 First of all we note that
  the same type of argument  of point (5) applies  to   $\psi^0_{\nu}$ the principal eigenvalue of $\A_\nu$. Namely  $\psi^0_{\nu} (z)$, $z \in I_\l$,  is positive and exponentially decaying  when $|z|$  large enough. Hence, as in point (5),   there exists $ z_0>0$ and $ \zeta_0>0$ so that
\begin {equation} \label {ma1} \psi^0_{\nu} (z) \ge \zeta_0, \qquad |z| \le z_0, \qquad  \forall \nu \in [0,1],  \end {equation}
where $ \zeta_0>0$ and $z_0>0$ are independent on $ \l$. 
  Hence, by  Proposition \ref {cap1}    we have that
 \begin {equation} \label {a8}   \begin {split}  & \langle \psi^0_{\nu'},  B \psi^0_{\nu'} \rangle   =
    \int_{I_\l} dz   \psi^0_{\nu'}(z)        \int_{I_\l}  [J( z-z') - J^h(z-z')]  \psi^0_{\nu'}(z') dz' \cr & \ge  2\z_0^2  z_0
    \frac 14 h^2 \int  \bar J_{tan} (z') dz'  \equiv C_0 h^2.
 \end {split}  \end {equation}
 The statement follows.
   \end   {proof}

    \begin {prop}  \label {tu1a}Let $\LL^{h}$ be the operator defined in   \eqref  {aa2} on $L^2 (I_\l)$.
    Let $h_0$ be as in Proposition \ref {tu2}.  There exists $ \nu= \nu (h_0)> 0$ independent on $h$ and $ \l$ so that for $|h| >h_0$ 
    \begin {equation}  \label {scu4}  \langle w, \LL^{h} w \rangle  \ge \nu  \|w\|^2.  \end {equation}
\end {prop}
  The proof of  Proposition \ref {tu1a} follows from Proposition \ref {tu1} and Proposition  \ref {tu3}.

  \begin {prop}  \label {tu1}Let $\LL^{h}$ be the operator defined in   \eqref  {aa2} on  $L^2 (I_\l)$.   There exists $h_1= h_1(\beta, J)>0$   independent on $\l$    so that for  $ |h|   > h_1$  
  \begin {equation}  \label {scu4}  \langle w, \LL^{h} w \rangle  \ge \frac 1 2\beta^{-1} \|w\|^2.  \end {equation}
\end {prop}
\begin {proof}
By \eqref {scu3}
$$ |\langle w, J^{h} w \rangle| \le   \frac 1  {(1+ |h|)} \int_{I_\l} dz |w (z)| \int dz'  C (z-z') |w (z')|  \le  \frac {C_2}  {(1+ |h|)} \|w\|^2$$
 where 
   \begin {equation}    \label {scu6}  C_2=  \int_{I_\l}  C (z)dz=  \int_{I_\l}   dz \int_T ds |\frac {d} {ds}  J(s,z)|.   \end {equation}
Hence,  by definition  of $\LL^{h}$  and \eqref {scu5},   we get 
  \begin {equation}    
 \sup_{\{ w: \|w\|=1\}} \langle w, \LL^{h} w \rangle  \ge     \sup_{\{ w: \|w\|=1\}}    ( \frac 1 \b -  \frac {C_2}  {(1+ |h|)} )\|w\|^2 .   \end {equation}
Choosing   $h_1= h_1(\b, J)>0$, so that  $\frac 1 2\beta^{-1}  \ge  \frac {C_2}  {(1+ |h|)} $ we get \eqref {scu4}.
\end {proof}        
 \begin {prop}  \label {tu3} Let $\LL^h$ be the operator defined in   \eqref  {aa2} over  functions $L^2 (I_\l)$.   
 For any  given $h_0>$ and $h_1>0$, there exists    $\nu=\nu(h_0,h_1)>0$    so that for  $h_0 \le  |h|   \le h_1$  
  \begin {equation}  \label {scu4}  \langle w, \LL^{h} w \rangle  \ge \nu \|w\|^2.  \end {equation}
\end {prop}
\begin {proof}
We first show   that there exists $c_1= c_1(h_0,h_1)>0$  so that    
 \begin {equation}  \label {scu6a}|J^{h} (z)|   \le\bar J (z) (1  - c_1), \qquad h_0 \le |h| \le h_1,\quad  z \in (-1,1).
 \end {equation}
  To show this we argue by contradiction.   Assume that  there exists    $\bar z  \in (-1,1)$ and  $\bar h $   so that
$| J^{\bar h} (\bar z)|=\bar J (\bar z)$.   Then $J^{\bar h} (\bar z)= e^{i \theta} \bar J (\bar z)$ for some  $\theta \in  \R$, 
and  therefore   
$$\bar J (\bar z)=   e^{-i \theta} J^{\bar h} (\bar z)$$
which  means  
 \begin {equation} \label {t1a}   \int_\R   J (\xi,\bar z) \left [1-  e^{-i \theta} e^{i   \bar h \xi }  \right ] d\xi =0.  \end {equation}
This implies  that the real part  of \eqref {t1a}, i.e 
$$  \int_\R   J (\xi, \bar z )  \left [1-  \cos (\theta-\bar h \xi ) \right ] d\xi=0.$$
 This  forces  $  J (\xi, \bar z )  =0$ for all $\xi  \in \R $, which  is a contradiction.
Since    for any given $z$ the set  $ \{  J^{ h } (z),  h_0 \le    | h| \le h_1 \} $
is a compact subset of $\R$,  
  and    $ | J^{h} (z)|<  \bar J (z) $, then \eqref {scu6a} follows.    Hence,  by definition  of $\LL^{h}$  and \eqref {scu6a},   we get, for $1>a>0$,   
  \begin {equation}  \begin {split} \langle w, \LL^{h} w \rangle &=   (1-a)    \int  dz \frac {w^2 (z)} {\s (\bar m (z))}+   a \left [  \int  dz \frac {w^2 (z)} {\s (\bar m (z))}-    \langle |w|( \bar J \star |w| \rangle \right ]  \cr & + a   \langle |w|( \bar J \star |w| \rangle    - \langle w,J^{h} \star w \rangle. 
  \end {split}   \end {equation}
  By  \eqref {scu6a}, 
  $$\left |  \langle w,J^{h} \star w \rangle \right | \le   \langle |w|,|J^{h}| \star| w| \rangle \le (1  - c_1)  \langle |w|, \bar J \star| w| \rangle.$$
Then,   taking  $a= 1- c_1$ we have that 
 $$ a   \langle |w|, \bar J \star |w| \rangle    - \langle w,J^{h} \star w \rangle  \ge 0.$$
 Further, since   
  \begin {equation}    
     \inf_{\{ w: \|w\|=1\}}  \langle |w|, \LL^0 |w|\rangle \ge  \inf_{\{ w: \|w\|=1\}} \langle w, \LL^0 w \rangle  \ge 0,  \end {equation}
     and $\s (\bar m (z)) \le \beta$ 
\begin {equation} \inf_{\{ w: \|w\|=1\}}  \langle w, \LL^{h} w \rangle  \ge   \frac {c_1} \beta \equiv \nu.    \end {equation}
\end {proof}

\section { Representation formula for   functions with small energy.}
 In this section, we show   a  representation theorem for functions having small
energy,  see   \eqref  {gv9}  below.
  The   representation   stated in Theorem \ref {g10}  is    reminiscent of      the  one    obtained by  X.Chen,  see  \cite [Lemma  2.4] {Chen}  in the C-H case. The   proof, as explained in the introduction,    is  different.
   
 \begin {thm}   \label {g10}  Take 
$ f
\in H^1 (\NN(d_0))$ such that $\|f\|_{L^2 (\NN(d_0))}=1$,   
  and   \begin {equation} \label {gv9}    
 \int_{\NN (d_0)} \left ( A^\l_{m_A}  f  (\xi) \right ) f(\xi) d\xi  \le  C \l^2.  \end {equation}
  Then, there exists $ \l_0>0$, so that for any $\l \in (0, \l_0)$   we can construct   $ Z (\cdot)  \in H^1 (T)$,   $f^R (\cdot,\cdot)  \in L^2 (\NN(d_0))$  such that
\begin {equation} \label {E.10a} f (s,r)=    Z(s)\frac 1 {\sqrt {\a (s,r)} }    \frac 1 {\sqrt \l} \psi^0_0 (\frac r \l)+   f^R(s,r),
 \end {equation}
 where $\psi^0_0 (\cdot)$ is the first eigenvalue of $ \LL^0$, see Theorem \ref {81}, 
 \begin {equation} \label {fusco.2a}   1- C \l^2 \le  \| Z \|^2_{L^2 (T)}  \le1, \qquad  \|\nabla Z\|_{L^2 (T)}    \le C
 \end {equation}
\begin {equation} \label {fusco.1a}    \| f^R \|^2_{L^2 (\NN(d_0))}  \le C \l^2.
 \end {equation} 
  \end {thm}
  Set $\hat f  (s,r)= \sqrt {\a(s,r)}  f(s,r)$. By    Lemma \ref {F2}, \eqref {gv9} implies
\begin {equation} \label {g11}  C \l^2 \ge \int_{\NN (d_0)}  A^\l_{m_A} f  (\xi) f(\xi) d\xi   \ge \langle \hat f,  L^\l  \hat f \rangle - \l^2  C \| f \|^2_{L^2 (\NN(d_0))} .
\end {equation} 
Since  by assumption  $\| f \|^2_{L^2 (\NN(d_0))}  =1$     we have 
\begin {equation} \label {g11a}   \l^2 C   \ge \langle \hat f, L^\l  \hat f \rangle.   \end {equation} 
Further setting
 $$   \hat f (s,r) = \frac 1 {\sqrt \l}  V(s, \frac r \l ) $$
 we have that  \eqref {g11a} is equivalent to
 \begin {equation} \label {g11z}    \l^2 C   \ge \langle V,  \A V\rangle,   \end {equation} 
where  $\A$  is the operator defined in \eqref {op2t}
 Therefore  Theorem \ref {g10} follows once  we show the following theorem. 
 \begin {thm} \label {gio1}   Take   $ V \in H^1 (\TT_\l)$, $ \|V\|=1$,
  \begin  {equation}  \label {gi8} \langle V,   \A  V \rangle  \le C \l^2,   \end {equation} 
 where $ \A $ is the operator  defined in \eqref {op2t}.
Let   $\psi^0_0 (\cdot)$ be  the first eigenvalue of $ \LL^0$, see Theorem \ref {81}, we have 
 \begin  {equation}  \label {gi8z} V(s,z) = Z(s) \psi^0_0 (z)  + V^R(s,z), \end {equation} 
\begin  {equation}  \label {gi8b}  1- \l^2 \le \|Z \|^2_{L^2(T)}  \le 1, \qquad    \|\nabla Z\| \le C, \qquad \|V^R\|^2 \le C \l^2. \end {equation} 
\end {thm}  
The proof of Theorem \ref {gio1} is based on a  deeper knowledge of the spectrum of the operator $\A$ in $ \L^2(\TT_\l)$.
 We explain in Subsection  6.1 how to prove  \eqref  {gi8z}  when the operator  $\A$ is replaced by the operator $ \GG$ 
 defined in \eqref {lg1a}. 
 Then in  Subsection  6.2 we   write the operator   $\A$ in term of $\GG$ plus extra terms.
In Subsection  6.3 we show   Theorem \ref  {gio1}.

  \subsection { Toy model} 
 We explain the method  for proving  Theorem \ref {gio1}  in a simpler context.
Replace the operator $ \A$ with  the   operator  $\GG^{\l}$  defined in \eqref {lg1a}.
 Notice that  $ \GG^{\l}$ defined in \eqref {lg1a}  is different from the  $ \A$ defined in   \eqref {op2t}.
Namely  $\bar m $ replaces $ m_A$   and the convolution  term   replaces  $J^c$. 
Assume   \begin  {equation}  \label {gi8a} \langle V,  \GG^{\l}  V \rangle  \le C \l^2.   \end {equation} 
We would like to prove that  \eqref {gi8z} and   \eqref {gi8b} follow.
Write   $V (s,z) = \sum_{k \in \Z_L} e^{i 2 \pi  \frac {k} L s}  u_k(z)$, see \eqref {aa1}.
  By \eqref {gi8a}   and simple computations
 \begin  {equation}  \label {gi8d}   C \l^2   \ge \langle V,    \GG^{\l} V \rangle =  \sum_{k \in \Z_L}    \langle  \LL^{k\l} u_k, u_k\rangle. \end {equation} 
 We then apply the  spectral results for  $\{ \LL^h\}_h$  obtained  above.
 Let $h_0>0$ be  as in  Proposition  \ref {tu2}. When $|k| \le \frac {h_0} \l $   split 
$$ u_k = \a_k \psi_0^{k\l} + u_k^\perp$$
where  $\psi_0^{k\l} $ is the principal eigenvalue of  $\LL^{k\l}$ and 
$$ \int dz \psi_0^{k\l}(z) u_k^\perp (z) =0.$$
 By Proposition  \ref {tu2} and  Proposition \ref {tu1a} we have
 \begin {equation} \begin {split}  \label {b1}   & \sum_{k \in \Z_L}  \langle  \LL^{k\l} u_k, u_k\rangle  \ge    \sum_{|k| \le \frac {h_0} \l }  \langle  \LL^{k\l} u_k, u_k\rangle  +  \nu  \sum_{|k| > \frac {h_0} \l } \|u_k\|^2 \cr & =
   \sum_{|k| \le \frac {h_0} \l }   \left [ \mu_0^{k\l} \a_k^2 +  \langle u_k^\perp,     \LL^{k\l}u_k^\perp \rangle  \right ]   
   +   \nu  \sum_{|k| > \frac {h_0} \l } \|u_k\|^2 
    \cr &\ge     \sum_{|k| \le \frac {h_0} \l }   \left [   \mu_0^{k\l} \a_k^2 + D \|u_k^\perp\|^2 \right ]  +     \nu  \sum_{|k| > \frac {h_0} \l } \|u_k\|^2.\end {split} \end {equation}   
 Therefore, see  \eqref {gi8d},
\begin  {equation}  \label {gi8f}   C \l^2   \ge \sum_{|k| \le \frac {h_0} \l }   \left [   \mu_0^{k\l}  \a_k^2 + D\|u_k^\perp\|^2 \right ]  +   \nu    \sum_{|k| > \frac {h_0} \l } \|u_k\|^2.   \end {equation}   
Since  $  \mu_0^{k\l} \ge  0$,  see \eqref {do8},    the  three  terms on the right hand side of \eqref  {gi8f} are positive,
hence 
\begin  {equation}  \label {v1a} \sum_{  |k| \le \frac {h_0} \l }      \mu_0^{k\l}  \a_k^2  \le C \l^2,  \end {equation} 
 \begin  {equation}  \label {v6} \sum_{|k| \le \frac {h_0} \l }     \|u_k^\perp\|^2  \le  \frac {C} D \l^2,  \end {equation} 
  \begin  {equation}  \label {v7}\sum_{|k| > \frac {h_0} \l } \|u_k\|^2    \le \frac {C} \nu \l^2.  \end {equation} 
By Proposition \ref {tu2}, since   $\mu_0^{0} \ge 0$, 
 \begin {equation} \begin {split}  \label {b1a} & \sum_{  |k| \le \frac {h_0} \l }       \mu_0^{k\l}   \a_k^2  =
  \sum_{  |k| \le \frac {h_0} \l }      [\mu_0^{k\l}-\mu_0^{0}+ \mu_0^{0}]  \a_k^2 
 \ge  \sum_{  |k| \le \frac {h_0} \l }   [\mu_0^{k\l}-\mu_0^{0}] \a_k^2  \cr & >   C_0  \sum_{  |k| \le \frac {h_0} \l }  [(k \l)^2 ]  \a_k^2.  
 \end {split} \end {equation}  
 Hence, by \eqref {v1a}
 \begin  {equation}  \label {v14}  \sum_{  |k| \le \frac {h_0} \l }  k^2    \a_k^2   \le \frac C {C_0}.  \end {equation}  
 Define
 \begin {equation}  \label {zeta1}  Z(s)=   \sum_{|k| \le  \frac {h_0} \l } e^{i k  s}  \a_k,   \end {equation} 
\begin {equation}  \label {zeta2}   V^R(s,z) =   V_2 (s,z) + V_3(s,z),  \end {equation} 
where 
\begin {equation}  \label {zeta3}  V_2(s,z) =   \sum_{|k| \le  \frac {h_0} \l } e^{i  k  s}  u_k^\perp +  \sum_{|k| > \frac {h_0} \l } e^{i  k s}  u_k(z),  \end {equation} 
 \begin {equation}  \label {zeta4}   V_3(s,z)  =  \sum_{|k| \le  \frac {h_0} \l } e^{i   k  s}  \a_k      [\psi_0^{k\l}-  \psi_0^{0}].  \end {equation} 
 Then
 \begin {equation}  \label {zeta5}   V (s,z) =  Z(s)   \psi_0^{0}(z) + V^R(s,z).  \end {equation}  
 By \eqref {v6} and \eqref {v7}
$$ \|V_2\|^2  \le C \l^2. $$ 
 By \eqref {scu2} and \eqref {v14}
   $$  \| V_3  \|^2 =  \sum_{|k| \le  \frac {h_0} \l }  \a_k^2  \|\psi_0^{k\l} - \psi_0^{0} \|^2 \le   \l^2 \sum_{|k| \le  \frac {h_0} \l }  \a_k^2 k^2       \le \l^2 C.  $$ 
Hence 
 $\|V^R \|^2 \le C \l^2$,  $1 -  C \l^2 \le    \|Z\|^2 \le 1$, 
   \begin {equation}  \label {zeta6}    \| \nabla_s  Z\|^2 =   \sum_{|k| \le  \frac {h_0} \l } k^2 \a_k^2  \le C.  \end {equation} 
    In this way we  get the decomposition \eqref  {gi8z}, \eqref   {gi8b} in the   toy model.
    The proof of  Theorem \ref {gio1}  is more complicated because  terms outside diagonal are present.
    
\subsection{ Expansion of $ \A$ in term of $\GG$.}   
In this subsection we  decompose  $\A =  \GG + \RR$,  but  the $L^2$ norm   of  $ \RR$ is not  of order $\l^2$.
Nevertheless  we show that 
when $V$ satisfies \eqref  {gi8}  then      $ \langle \A V, V \rangle  \simeq  \langle \GG V, V \rangle   +C\l^2 |V\|^2$.

   We  start  writing   the operator $ \B$ defined in \eqref {me7d}  in term of   a convolution operator 
 plus  an operator involving the first derivative  with respect to the $s-$ variable of $V$,   plus a remainder.
The remainder   for general $ L^2(\TT_\l)$ functions  is of order 1 in $L^2(\TT_\l)$, so it is not small.  To get the remainder small  we need to require  decay properties of $V$.   They hold when $V$ satisfies \eqref  {gi8}.

   \begin {lem} \label {car1} Let $V \in H^1(\TT_\l)$ and  $ \B$ be the operator defined in \eqref {me7d}.
       We have
    \begin {equation} \label {v6a}  \begin {split}   & 
(\B V)(s,z) =   \int_{\TT_\l}    \frac 1 \l J \left  (  \frac {  (s-s')} {\l},  (z-z')\right )  V(s',z') ds' dz'    \cr & 
 -
  \l^2  \int_{\TT_\l}     \frac 1 \l   J  \left  ( \frac {  (s-s')} {\l} ,    (z-z') \right )    \left [  \frac {(s-s')} \l   k (s^*)z^*\a(s^*,z^*)  D_{s'} V(s',z') \right ]ds' dz'   
  \cr &  +
 (R^\l V) (s,z)   \end {split}   \end {equation}
   where  $R^\l$ is defined in \eqref {rm5}.
  \end {lem} 
 \begin {proof}     
Set $x_0= \frac {  (s-s')} {\l} $ and $x=\frac {  (s-s')} {\l} \a(s^*,z^*)$.
 We   expand $   J  $ in Taylor formula, up to second order,  at the point $x_0  $ and for any given $z \in I_\l$.
  We obtain 
 \begin {equation} \label {v2z}  \begin {split}      J \left  (x,  z\right )= &   J \left  ( x_0, z\right ) + D_1   J \left  ( x_0,  z\right ) (x-x_0)     \cr & +
 \frac 12  \int_{x_0}^x D_{11}    J \left  ( x',  z\right )  (x_0-x')  dx'.
  \end {split}   \end {equation}
   We denoted by $D_1   J$  and by $D_{11}  J$ the  first derivative   and respectively the second  derivative of $J$
  with respect to the first argument.  Recalling that the  support of $J$ is the ball of radius 1 we have that  $|s-s'|  \le \l$, $ |z-z'| \le 1$   and  $  (x-x_0)= -   \frac 1 \l   (s-s') \l k(s^*)z^* \simeq \l k(s^*)z^*$.    We set 
\begin {equation} \label {mt}   M(s,s', z,z')= \frac 12  \int_{x_0}^x D_{11}    J \left  ( x',  z\right )  (x_0-x')  dx'.  \end {equation}
By the properties of $J$ and the boundness of the curvature    
\begin {equation} \label {ms1}  \left |  M(s,s', z,z')\right | \le \l^2 C (z^*)^2 \1_{|s-s'|\le \l}  \1_{|z-z'|\le 1}.   \end {equation}
 By \eqref {v2z}, taking into account that  $\a(s^*,z^*) = 1 - \l k(s^*) z^*$, we have 
 \begin {equation} \label {v4}  \begin {split}   & ( \B V) (s,z) = 
   \int_{\TT_\l}   \frac 1 \l J \left  (  \frac {  (s-s')} {\l} \a(s^*,z^*),  (z-z')\right )  \a(s^*,z^*)V(s',z') ds' dz' \cr &  =
   \int_{\TT_\l}   \frac 1 \l  { J} \left  ( \frac {  (s-s')} {\l} ,    (z-z') \right )   V(s',z') ds' dz'  \cr &
   -    \l \int_{\TT_\l}   \frac 1 \l  { J} \left  ( \frac {  (s-s')} {\l} ,    (z-z') \right )  k(s^*) z^* V(s',z') ds' dz' \cr &  -
   \l  \int_{\TT_\l}   \frac 1 \l    (D_1 { J})\left  ( \frac {  (s-s')} {\l} ,    (z-z') \right )   \frac 1 \l  (s-s') k(s^*)z^*  \a(s^*,z^*) V(s',z') ds' dz'  \cr & +
   \frac 1 \l   \int_{\TT_\l}   M(s,s', z,z') \a(s^*,z^*)V(s',z') ds' dz'.
  \end {split}   \end {equation}
Since $$(D_1 { J})\left  ( \frac {  (s-s')} {\l} ,    (z-z') \right )= \l (D_s { J})\left  ( \frac {  (s-s')} {\l} ,    (z-z') \right )= -\l (D_{s'} { J})\left  ( \frac {  (s-s')} {\l} ,    (z-z') \right )$$  
  we obtain 
   \begin {equation} \label {v5}  \begin {split}   &  \int_{\TT_\l}     (D_1 { J})\left  ( \frac {  (s-s')} {\l} ,    (z-z') \right )   \frac 1 \l   (s-s') k(s^*)z^*  \a(s^*,z^*) V(s',z') ds' dz'  \cr &  =
  -    \int_{\TT_\l}     (D_{s'} { J})\left  ( \frac {  (s-s')} {\l} ,    (z-z') \right )     (s-s') k(s^*)z^* \a(s^*,z^*)  V(s',z') ds' dz' \cr &=
 \int_{\TT_\l}      J  \left  ( \frac {  (s-s')} {\l} ,    (z-z') \right )       D_{s'} \left [  (s-s') k(s^*)z^*   \a(s^*,z^*)V(s',z') \right ]ds' dz' \cr &=
  -  \int_{\TT_\l}      J  \left  ( \frac {  (s-s')} {\l} ,    (z-z') \right )        k(s^*)z^* [1-  \l k(s^*)z^*] V(s',z')  ds' dz'  \cr & +
   \int_{\TT_\l}      J  \left  ( \frac {  (s-s')} {\l} ,    (z-z') \right )        (s-s')   k_{s'}(s^*)z^*  \a(s^*,z^*) V(s',z') ds' dz'
   \cr & +
   \int_{\TT_\l}      J  \left  ( \frac {  (s-s')} {\l} ,    (z-z') \right )       (s-s')   k (s^*)z^* \a(s^*,z^*) D_{s'} V(s',z') ds' dz' \cr & +
    \int_{\TT_\l}      J  \left  ( \frac {  (s-s')} {\l} ,    (z-z') \right )    (s-s')   k (s^*)z^*  D_{s'}[\a(s^*,z^*)]  V(s',z')  ds' dz'. 
   \end {split}   \end {equation}
   We  replaced    $\a(s^*,z^*)= 1-  \l k(s^*)z^*$ in the first term of the last equality  of \eqref {v5}.
   We insert \eqref {v5} into the third  term of the last equality in \eqref{v4}. Notice that we must multiply \eqref {v5} by $- \l \frac 1 \l$. 
      We obtain 
   \begin {equation} \label {v6z}  \begin {split}   & 
\int_{\TT_\l}   J^c ( s,s',z,z') V(s',z') ds' dz' =   \int_{\TT_\l}    \frac 1 \l J \left  (  \frac {  (s-s')} {\l},  (z-z)\right )  V(s',z') ds' dz'    \cr & 
 -
  \l  \int_{\TT_\l}     \frac 1 \l   J  \left  ( \frac {  (s-s')} {\l} ,    (z-z') \right )    \left [  (s-s')   k (s^*)z^*  \a(s^*,z^*) D_{s'} V(s',z') \right ]ds' dz'   
  \cr & +  (R^\l V) (s,z),  
  \end {split}   \end {equation}
where 
 \begin {equation} \label {rm5}  \begin {split}   (R^\l V) (s,z) =&    - \l^2 \int_{\TT_\l}      \frac 1 \l    J  \left  ( \frac {  (s-s')} {\l} ,    (z-z') \right )    \left [    (k(s^*)z^*)^2   V(s',z') \right ]ds' dz'  \cr & 
- \l^2  \int_{\TT_\l}       \frac 1 \l J  \left  ( \frac {  (s-s')} {\l} ,    (z-z') \right )    \left [  \frac {(s-s')} \l \frac 12  k_{s'}(s^*)z^* \a(s^*,z^*)  V(s',z') \right ]ds' dz' \cr &  -
   \l^3  \int_{\TT_\l}     \frac 1 \l  J  \left  ( \frac {  (s-s')} {\l} ,    (z-z') \right )     \left [ \frac {(s-s')} \l  k (s^*)     k_{s'}(s^*)(z^*)^2  V(s',z') \right ]ds' dz' \cr &+    \int_{\TT_\l}     \frac 1 \l M(s,s', z,z') \a(s^*,z^*)V(s',z') ds' dz',
  \end {split}   \end {equation}
with  $M$    defined  in \eqref {mt}.
The  term   
$$ -   \l  \int_{\TT_\l}  \frac 1 \l J \left  (  \frac {  (s-s')} {\l},  (z-z)\right )  k(s^*)z^* V(s',z') ds' dz'   $$
appearing in the third  lines of \eqref {v4} 
  cancels with   the   first addend of the last equality of    \eqref {v5}     when multiplied by $- \l \frac 1 \l$.
     \end {proof}
     We have the following
  
  \begin {lem} \label {car1b} Let $V \in L^2 (\TT_\l)$  with the property that there exists $z_0$ and $a>0$ so that 
\begin {equation} \label {tue1}  \int_T ds V(s,z)^2   \le  e^{-a |z|}  \|V \|^2 \qquad for  |z|\ge z_0.  \end {equation}
Let $  R^\l  $ be the quantity defined in \eqref {rm5},    there exists $C= C(z_0)>0$ so that 
$$  \| R^\l   V \|  \le  \l^2  C \| V \|.$$
 \end {lem}
\begin {proof}
Denote
$$(R^\l_1 V) (s,z)=  - \l^2 \int_{\TT_\l}      \frac 1 \l    J  \left  ( \frac {  (s-s')} {\l} ,    (z-z') \right )         [k(s^*)z^*)]^2   V(s',z')  ds' dz' ,$$
$$(R^\l_2 V) (s,z) = - \l^2  \int_{\TT_\l}       \frac 1 \l J  \left  ( \frac {  (s-s')} {\l} ,    (z-z') \right )    \left [  \frac {(s-s')} \l \frac 12  k_{s'}(s^*)z^* \a(s^*,z^*)  V(s',z') \right ]ds' dz',$$
  $$ (R^\l_3 V) (s,z)=- \l^3  \int_{\TT_\l}     \frac 1 \l  J  \left  ( \frac {  (s-s')} {\l} ,    (z-z') \right )     \left [ \frac {(s-s')} \l  k (s^*)     k_{s'}(s^*)(z^*)^2  V(s',z') \right ]ds' dz.$$
 By Jensen inequality, \eqref {tue1}  and taking into account that $|z-z'| \le 1$, we have  
 \begin {equation} \label {rm8}  \begin {split}  \|R^\l_1 V\|^2 =&  \l^4  \int_{\TT_\l}   ds dz \left \{ \int_{\TT_\l}      \frac 1 \l    J  \left  ( \frac {  (s-s')} {\l} ,    (z-z') \right )       [k(s^*) z^*]^2   V(s',z')  ds' dz' \right \}^2  \cr & \le 
\l^4 C  \int_{\TT_\l}   ds dz  \int_{\TT_\l}      \frac 1 \l    J  \left  ( \frac {  (s-s')} {\l} ,    (z-z') \right )    \left [  ( z'+1)^2   V(s',z') \right ]^2  ds' dz'  \cr &    \le  \l^4  C(z_0) \| V \|^2.
 \end {split}   \end {equation}
  We   estimate similarly  the term  $R^\l_2 V$ and $R^\l_3 V$.
  By \eqref {ms1} and  \eqref {tue1}   we  have  
 \begin {equation} \label {rm2}  \begin {split}  \|R^\l_4 V\|^2 &\equiv  \int_{\TT_\l} ds dz  \left \{  \frac 1 \l   \int_{\TT_\l}   M(s,s', z,z') \a(s^*,z^*)V(s',z') ds' dz' \right\}^2 \le     \cr &
   \le 
   C  \l^4   \int_{\TT_\l} ds dz  \left \{    \int_{\TT_\l}       (z^*)^2  \frac 1 \l \1_{|s-s'|\le \l} \1_{|z-z'|\le 1}  |V(s',z')| ds' dz' \right\}^2 
 \cr & \le    C  \l^4   \int_{\TT_\l} ds dz      \int_{\TT_\l}     \frac 1 \l \1_{|s-s'|\le \l} \1_{|z-z'|\le 1}         (z^*)^4   V(s',z')^2  ds' dz'  \cr & \le  C(z_0)  \l^4  |V\|^2.
 \end {split}   \end {equation}
   Since $R^\l   V = \sum_{i=1}^4 R^\l_i   V$, the thesis follows.
\end{proof}

  \begin {lem} \label {car2}   Set   $V(s,z)= \sum_{h\in \Z_L} e^{ihs} u_h(z)$, see \eqref {aa1}.   We have 
   \begin {equation} \label {v11}   
   \begin {split}   &   
 \int_{\TT_\l} ds dz  V(s,z) ( \B V) (s,z)\cr & =  
  \sum_{h \in \Z_L,k \in \Z_L} \left  \{   \d_{h,k}    \int_{I_\l}   dz    u_h(z)     \int_{I_\l} J^{\l k} (z-z')  u_k(z') dz'
+  F^\l_1 (u_h,u_k) \right \} +   \int_{\TT_\l} ds dz  V(s,z) (R^\l V) (s,z)
    \end {split}   \end {equation}
    where $J^{\l k}$ is defined in \eqref {aa3},  $F^\l_1 (u_h,u_k) $ in  \eqref {v14z}. 
   \end{lem}

\begin {proof}
  From   \eqref {v6a} we have
 \begin {equation} \label {v12}   
   \begin {split}   &    \int_{\TT_\l} ds dz  V(s,z) ( \B V) (s,z)  
 \cr & =   \sum_{h,k}
  \int_{\TT_\l} ds dz  e^{ihs } u_h(z)    \int_{\TT_\l}    \frac 1 \l J \left  (  \frac {  (s-s')} {\l},  (z-z')\right )  e^{iks'} u_k(z')   ds' dz'  
   \cr &
    -\l^2   \sum_{h,k}  \int_{\TT_\l} ds dz  e^{-ihs} u_h(z)   \int_{\TT_\l}     \frac 1 \l   J  \left  ( \frac {  (s-s')} {\l} ,    (z-z') \right )    \left [  \frac {(s-s')} \l   k (s^*)z^*  ik  e^{iks'} u_k(z') \right ]ds' dz'  \cr &
+    \int_{\TT_\l} ds dz  V(s,z) (R^\l V) (s,z).   \end {split}   \end {equation}
We have that 
 \begin {equation} \label {v13}   
   \begin {split}   &  \int_{\TT_\l} ds dz  e^{-ihs } u_h(z)    \int_{\TT_\l}    \frac 1 \l J \left  (  \frac {  (s-s')} {\l},  (z-z')\right )  e^{iks'} u_k(z')   ds' dz'   \cr & =  \int_{\TT_\l} ds dz  e^{-ihs } u_h(z)   e^{iks} J^{\l k} (z-z')  u_k(z')  =   \d_{h,k}    \int_{I_\l}   dz    u_h(z)     \int_{I_\l} J^{\l k} (z-z')  u_k(z') dz'
 \end {split}   \end {equation}
 where, see \eqref {aa3},  
  \begin {equation} \label {cg2a}  J^{\l k} (z-z')  =    \int_{T}    \frac 1 \l J \left  (  \frac {  (s-s')} {\l},  (z-z')\right )  e^{ik \l \frac { (s'-s)} \l} ds'  =
  \int     J \left  (w,  (z-z')\right )  e^{-ik \l w} dw.     \end {equation}
  Set 
    \begin {equation} \label {v14z}   
   \begin {split}  F^\l_1 (u_h,u_k) = &   
  -  ik  \l^2 \int_{\TT_\l} ds dz  e^{-ihs} u_h(z)     e^{i ks} \int_{\TT_\l}   \frac 1 \l   J  \left  ( \frac { (s-s')} {\l},    (z-z') \right )     \frac {(s-s')} {\l}  k (s^*)  e^{i k \l  \frac  {(s-s')} {\l}}  z^*   u_k(z')  ds' dz' \cr & =
    - ik  \l^2  \int_{I_\l}    dz    u_h(z)     \int_{I_\l}   z^*  u_k(z')   dz'     \int    J  \left  (w,    (z-z') \right )  w   e^{-ik \l w}   \int_{T} ds  e^{i (k-h)s}   k (s+\frac 12 \l w)       dw   \cr &  =
   - ik  \l^2  \int_{I_\l}    dz    u_h(z)     \int_{I_\l}   z^*  u_k(z')   dz'     \int    J  \left  (w,    (z-z') \right )  w  e^{-i \frac 12 (3k-h) \l w}    \int_{T} ds'  e^{i (k-h)s'}   k (s') \cr & =
     - k  \l^2  \int_{I_\l}    dz    u_h(z)     \int_{I_\l}   z^*  u_k(z')   dz'     \int    J  \left  (w,    (z-z') \right )  w    \sin  ( \frac 12 (3k-h) \l w)    \int_{T} ds'  e^{i (k-h)s'}   k (s').   
   \end {split}   \end {equation} 
   We use that $ s^*= \frac   {s+s'} 2 = s +    \frac   {s'-s} 2$ and therefore $  k (s^*(w)) = k(s + \frac 12 \l w)$.      
\end {proof}   

\begin {rem} \label {dennis1}  Notice that  when $v$ and $w$ are even function 
$$F^\l_1 (v,w) =0.$$
\end {rem}
  
 \begin {lem} \label {car9}   Take  $ \|V\|=1$,  $V(s,z)= \sum_{k \in \Z_L} e^{iks} u_k(z)$, see \eqref {aa1},
  \begin  {equation}  \label {gi8v} \langle V,   \A V \rangle  \le C \l^2,   \end {equation} 
  where the operator $\A$ is defined in \eqref {op2t}. 
  Then
  \begin {equation} \label {ele1}         \langle V, \A V \rangle  -  \langle V,R^\l_0 V \rangle    =
  \sum_{k \in \Z_L,  h \in \Z_L}    \left  \{    \d_ {h,k} \langle u_h,     \LL^{k\l} u_k \rangle  
+         F^\l_1 (u_h,u_k)  +  F^\l_2 (u_h,u_k) + F^\l_3 (u_h,u_k)\right \},   \end {equation} 
\begin {equation} \label {ele11}     \langle V,R^\l_0 V \rangle  \le   C \l^2,   \end {equation} 
 see \eqref {d9}, where  $F^\l_1$ is  given in \eqref {v14z},    $F^\l_2$ in \eqref {d5} and   $F^\l_3$ in \eqref {d6}. 
 Further  we have for $u$ and $v$ even function 
 \begin {equation} \label {ele10}    F^\l_1 (u,v) = F^\l_2 (u,v) =F^\l_3 (u,v)=0, \end {equation} 
\begin {equation} \label {d502}   
    |F^\l_1 (u_h,u_k)|    \le   \l^2   |k|  C  
   \frac {1} {(1+ |k-h|)^3}   \frac {1} {(1+ |\frac 12 (3k-h)| \l ) } 
       \|u_h\|  \|u_k\|,     \end {equation} 
        \begin {equation} \label {d501}   
 |F^\l_2 (u_h,u_k) + F^\l_3 (u_h,u_k)|  \le C  \l     \|u_h\|  \|u_k\| C 
   \frac {1} {(1+ |k-h|)^3}.        \end {equation} 
\end {lem}  
\begin {proof}     From  \eqref {gi8v} it follows that  $V$   satisfies  \eqref  {tue1}, i.e  there exists  
$z_0$ and $a>0$ so that 
\begin {equation} \label {tue1sa}  \int_T ds V(s,z)^2   \le  e^{-a |z|}  \|V \|^2 \qquad \hbox {for} \quad   |z|\ge z_0.  \end {equation}
Namely,  by Lemma \ref {S12t}, one deduces that the eigenfunctions  $\{\Psi_i \}_i$ associated to  eigenvalues of $\A$ smaller than some $\e_0$
decay as \eqref {tue1sa}.  Therefore  we can write  for some  positive integer $N$, $V= \sum_i ^N   a_i \Psi_i (s,z)$  where      $a_i= \int_{\TT_\l} ds dz  \Psi_i (s,z) V(s,z)$.
We get
\begin {equation} \label {tue1da}  \int_T ds V(s,z)^2  =  \sum_{i}^N   a_i^2  \int_T ds \Psi^2_i (s,z)   \le  e^{-a |z|}   C \sum_{i}^N   a_i^2 \| \Psi_i \|^2 =   Ce^{-a |z|}  \|V \|^2 \qquad \hbox { for} \quad  |z|\ge z_0.  \end {equation}
By  \eqref {mob1} and definitions \eqref {tre.2}, \eqref {tre.3}  and \eqref {tre.4}   we obtain 
  \begin {equation} \label {d1}   
   \begin {split}       
 \int_{\TT_\l} ds dz \frac {  V(s,z)^2} {\s (m_A (s,\l z))}   & =   \int_{\TT_\l} ds dz \frac { V(s,z)^2}   {\b (1-\bar m^2(z))} \cr & +
 \int_T ds \left [ I_{2,s}(V) + I_{3,s}(V) + I_{4,s}(V) \right ].  
   \end {split} \end {equation} 
We write  
\begin {equation} \label {d2}   
   \begin {split}       I_{3,s}(V)  &=  \l \int_{\TT_\l} ds dz \frac { (V(s,z))^2}   {\b (1-\bar m^2(z))^2}    \bar m (z)     
\phi (s, 0)    \cr & +   \l^2  \int_{\TT_\l} ds dz \frac { (V(s,z))^2}   {\b (1-\bar m^2(z))^2}    \bar m (z) [ \phi(s,z)- \phi(s,0]. 
 \end {split} \end {equation} 
  By      \eqref {2.4}        and  \eqref {tue1sa}     we get  
   \begin {equation} \label {d8a}  \left | \int_{\TT_\l} ds dz \frac { (V(s,z))^2}   {\b (1-\bar m^2(z))^2}    \bar m (z)   [ \phi(s,z)- \phi(s,0]  \right | \le C \|V\|^2.  \end {equation} 
Hence by \eqref  {d1}  and  \eqref  {d2} 
\begin {equation} \label {d3}   
   \begin {split}   &   
 \int_{\TT_\l} ds dz \frac { (V(s,z))^2} {\s (m_A (s,\l z))}   =    \int_{\TT_\l} ds dz \frac { (V(s,z))^2}   {\b (1-\bar m^2(z))} \cr & +
 2 \l \int_{\TT_\l} ds dz \frac { (V(s,z))^2}   {\b (1-\bar m^2(z))^2}    \bar m (z)   [ h_1(z) g(s) +   
\phi (s, 0)]    + \l^2   \int_{\TT_\l} ds dz  R_1(s,z) (V(s,z))^2, 
  \end {split} \end {equation} 
 where we denoted by 
 $$R_1 (s,z) =   \frac 1  {\b (1-\bar m^2(z))^2}    \bar m (z)    [ \phi(s,z)- \phi(s,0]+ q^\l (s,\l z).$$
By \eqref {d8a}  and \eqref {2.5}    we have that 
 \begin {equation} \label {d8}      \langle V,R_1 V \rangle   \le C \|V\|^2.   \end {equation} 
Set  $V(s,z)= \sum_{k} e^{iks} u_k(z)$.    Taking into account \eqref {d3} we have 
 \begin {equation} \label {d4}   
   \begin {split}   &   
 \int_{\TT_\l} ds dz \frac { (V(s,z))^2} {\s (m_A (s,\l z))}    = \sum_k   \int_{I_\l}  dz \frac { (u_k(z))^2} {\s ( \bar m(z))} 
+    \sum_k   \sum_h  \left [ F^\l_2 (u_h,u_k)  + F^\l_3 (u_h,u_k) \right ]  \cr & + \l^2   \int_{\TT_\l} ds dz  R_1 (s,z) (V(s,z))^2,  
 \end {split} \end {equation} 
where  \begin {equation} \label {d5}   
 F^\l_2 (u_h,u_k)= 2 \l     \int_{I_\l}  dz  \frac {  u_k(z) u_h(z)}    {\b (1-\bar m^2(z))^2}    \bar m (z)    h_1(z)  \int_T ds e^{i(k-h)s} g(s)  \end {equation} 
 \begin {equation} \label {d6}  
  F^\l_3 (u_h,u_k)= 2 \l      \int_{I_\l}  dz  \frac {  u_k(z) u_h(z)}    {\b (1-\bar m^2(z))^2}    \bar m (z)     \int_T ds e^{i(k-h)s}  \phi (s, 0).  
\end {equation} 
 Then by Lemma \ref  {car2} 
\begin {equation} \label {mac2bb}    \begin {split}    & \langle V,   L^{\l} V \rangle  -  \langle V,R^\l V \rangle  +  \l^2  \langle V,R_1 V \rangle  \cr &  =
  \sum_{k, h}    \left  \{    \d_ {h,k} \langle u_h,     \LL^{k\l} u_k \rangle  
  +     F^\l_1 (u_h,u_k)  +  F^\l_2 (u_h,u_k) + F^\l_3 (u_h,u_k)\right \}.
   \end {split}     \end {equation}
 Since $V$  satisfies    the decay property \eqref {S1.5bt}  we can apply Lemma \ref {car1b}
  obtaining 
 \begin {equation} \label {d9a}   \langle V,R^\l  V \rangle \le  C \l^2 \|V\|^2 . \end {equation} 
  Denote 
  \begin {equation} \label {d9}    \langle V,R^\l_0  V \rangle = \l^2  \langle V,R_1 V \rangle +  \langle V,R^\l  V \rangle.  \end {equation} 
  By \eqref {d8}  and \eqref {d9a} we get  \eqref {ele11}.
      By inspection one  realises that when $u$ and $v$ are even  \eqref {ele10} holds. 
  Further, since 
$$  u_h (z) = \int V(s,z) e^{ihs} ds$$
then, see  \eqref {S1.5bt}, 
 \begin {equation} \label {mart2}  |u_h (z)|  \le  \int |V(s,z)|   ds \le  C \left ( \int |V(s,z)|^2   ds \right )^{\frac 12}  \le Ce^{-  a |z|} \|V \|,  \quad |z| \ge z_0, \end {equation}
where   $a>0$ and $z_0>0$  do not depend  on $\l$. 
To estimate $F_i$, $i=1,2,3$ we  use the smoothness of $J$,   $\G$ and $ m_A$. We therefore use estimate \eqref {scu3},
 \begin {equation} \label {cur1}  \left |\int_{T} ds  e^{i (k-h)s}   k (s) \right | \le    \frac {C} {(1+ |k-h|)^3 },   \end {equation}
  $$   \left | \int_T ds e^{i(k-h)s} g(s) \right |   \le  C  
   \frac {1} {(1+ |k-h|)^3}, $$
   $$  \left |  \int_T ds e^{i(k-h)s} \Phi(s,0)  \right |  \le  C   
   \frac {1} {(1+ |k-h|)^3}. $$
   To estimate $F^\l_1$ we bound   $|z^*| \le |z+1|$ since   $J$ has  support in the ball of radius 1.  The exponential decay of  the $u_h$, see  \eqref   {mart2},  is essential   to control the  growing of   $ |z+1|$.
We get  
 \begin {equation} \label {cg4a}   \begin {split}              \left |    F^\l_1 (u_h,u_k)  \right |  & \le\l^2 |k|  C   
   \frac {1} {(1+ |k-h|)^3}   \frac {1} {(1+ |\frac 12 (3k-h)| \l )} 
       \int_{I_\l}  dz   |u_h(z)|  |z+1|   \int_{I_\l}   C (z-z')   |u_k(z')|  dz' \cr &   \le
       \l^2   |k| C    
   \frac {1} {(1+ |k-h|)^3}   \frac {1} {(1+ |\frac 12 (3k-h)| \l )} 
       \|u_h\|  \|u_k\|,     \end {split}     \end {equation} 
   \begin {equation} \label {d50}   
 |F^\l_2 (u_h,u_k) + F^\l_3 (u_h,u_k)|  \le C  \l     \|u_h\|  \|u_k\|     
   \frac {1} {(1+ |k-h|)^3}.        \end {equation} 
      \end {proof}
 \vskip1.cm 
 \noindent
    { \bf {Proof of Theorem \ref {gio1} } }  By assumption, $  \langle V, \A V \rangle \le C \l^2$,  see \eqref {gi8}, hence  Lemma \eqref {car9} holds.  Take $V(s,z)$  as in \eqref {aa1}  and consider the representation  of $ \langle V,   \A V \rangle $ given in \eqref {ele1}. 
We start considering the diagonal term, i.e when $h=k$.
We  need to lower bound the following quantity. 
 \begin {equation} \label {d10}     \sum_{k}    \left  \{   \langle u_k,     \LL^{k\l} u_k \rangle  
+      F^\l_1 (u_k,u_k)+  F^\l_2 (u_k,u_k) +  F^\l_3 (u_k,u_k) \right \}. \end {equation} 
Split  $\sum_k= \sum_{| \l k| \le h_0} + \sum_{| \l k| > h_0} $ where $h_0$ is as in Proposition \ref {tu2}. 
When  $ | \l k| \le h_0 $ split 
 \begin {equation} \label {d11}   u_k = \a_k \psi_0^{k\l} + u_k^\perp \end {equation} 
where  $\psi_0^{k\l} $ is  the principal eigenvalue of  $\LL^{k\l}$   and 
$$ \int dz \psi_0^{k\l}(z) u_k^\perp (z) =0 .$$
By \eqref {tue3}  
$$ \langle u_k,     \LL^{k\l} u_k \rangle  = \mu_0^{k\l} \a_k^2 +  \langle u_k^\perp,     \LL^{k\l}u_k^\perp \rangle \ge   \mu_0^{k\l} \a_k^2 + D \|u_k^\perp\|^2. $$
By Proposition \ref {tu2}  the principal eigenvalue of  the operator $ \LL^{k\l}$ is even, hence   for $ i= \{1,2,3\}$
 \begin {equation} \label {s1}    F^\l_i ( \psi_0^{h\l}, \psi_0^{k\l}) =0, \qquad  |\l h| \le h_0, \quad  |\l k| \le h_0,  \end {equation} 
 and
  \begin {equation} \label {s2}   F^\l_i (u_k, u_h) =  F^\l_i (\a_k \psi_0^{k\l}, u_h^\perp) +  F^\l_i (u_k^\perp, \a_h \psi_0^{h\l}) +   F^\l_i (u_k^\perp,   u_h^\perp).  \end {equation} 
  Property \eqref {s1} is essential to obtain the  final estimate. 
Taking into account  \eqref {d502}, \eqref {d501}    and \eqref {s1} we have
 \begin {equation} \label {s2a}  \begin {split} & | F^\l_1 (u_k, u_k)| \le  \l^2 |k|    C  
    \frac {1} {(1+ |k| \l )} 
       [ 2 |\a_k| \|u_k ^\perp\| + \|u_k^\perp\|^2] \le   \l    C  
   [   |\a_k| \|u_k^\perp\| + \|u_k^\perp\|^2] ,    \end {split}   \end {equation} 
\begin {equation} \label {s3}  | F^\l_2 (u_k, u_k)+  F^\l_3 (u_k, u_k)|  \le  \l    C  
     [ 2 |\a_k| \|u_k^\perp\| + \|u_k^\perp\|^2].   \end {equation} 
     Therefore
     \begin {equation} \label {d100}   \begin {split} &  \sum_{|k| \le \frac {h_0} \l }    \left  \{   \langle u_k,     \LL^{k\l} u_k \rangle  
+      F^\l_1 (u_k,u_k)+  F^\l_2 (u_k,u_k) +  F^\l_3 (u_k,u_k) \right \} \cr & \ge
 \sum_{|k| \le \frac {h_0} \l }    \left  \{    \mu_0^{k\l} \a_k^2 + D \|u_k^\perp\|^2
-     \l    C  
     [   \|u_k^\perp\| + \|u_k^\perp\|^2]\right \} \cr & \ge
      \sum_{|k| \le \frac {h_0} \l }    \left  \{    \mu_0^{k\l} \a_k^2 + [D-C\l] \|u_k^\perp\|^2
-     \l    C   |\a_k|  \|u_k^\perp\|  \right \}. 
 \end {split}
 \end {equation}  
 When $ |k| > \frac {h_0} \l$, taking advantage that $ \l^2 \frac {|k|}  {(1+ |k| \l )}  \le \l $, 
 and by   \eqref {d502}, \eqref {d501}    we   get  
 $$  | \sum_i F^\l_i (u_k, u_k)| \le \l C \|u_k \|^2.$$ 
 Hence 
  \begin {equation} \label {d101}   \begin {split} &  \sum_{|k| > \frac {h_0} \l }    \left  \{   \langle u_k,     \LL^{k\l} u_k \rangle  
+      F^\l_1 (u_k,u_k)+  F^\l_2 (u_k,u_k) +  F^\l_3 (u_k,u_k) \right \} \cr & \ge
 \sum_{|k| > \frac {h_0} \l }       [\nu-\l C]   \|u_k\|^2.
  \end {split}
 \end {equation}  
Next we estimate the therms outside diagonal,  i.e $ h\neq k$:
 \begin {equation} \label {mac2b}    \begin {split}  
 &\sum_{k, h; k \neq h }  [ F^\l_1 (u_h,u_k)+  F^\l_2 (u_h,u_k)+  F^\l_3 (u_h,u_k)] \cr & =
\left [ \sum_{|k \l| \le h_0}    \sum_{|h \l|  \le  h_0, h \neq k} + \sum_{|k \l| \le h_0}    \sum_{|h \l| > h_0}   +  \sum_{|k \l| > h_0}    \sum_{|h \l| \le h_0}   + \sum_{|k \l| > h_0}    \sum_{|h \l| > h_0, h \neq k}  \right  ]   \sum_i F^\l_i (u_h,u_k).   
   \end {split}
 \end {equation}  
   We  estimate  each   term of \eqref  {mac2b}.
      We control the double sums   using   that 
$$ \forall k \qquad \sum_{h }  \frac {1} {(1+ |k-h|)^3}    \le C, \qquad   \sum_{h }  \frac {1} {(1+ |k-h|)^2}    \le C.$$ 
  When $|k \l| \le h_0$ we decompose     $u_k$     as in  \eqref {d11}. 
 When $|k \l| \le h_0$ and   $|h \l| \le h_0$   we take advantage of  \eqref {s1} and  \eqref {s2}.
  By 
  \eqref {d502}  we  obtain \begin {equation} \label {mac2bb}    \begin {split}  &  
    \sum_{|k \l| \le h_0}    \sum_{|h \l|  \le  h_0, h \neq k}    |F^\l_1 (u_h,u_k)|  \cr &\le 
  \l^2   \ \sum_{|k \l| \le h_0}    \sum_{|h \l|  \le  h_0, h \neq k}   \left \{   |k|     
   \frac {1} {(1+ |k-h|)^3}   \frac {1} {(1+ |\frac 12 (3k-h)| \l )}   \left [    |\a_k| \|u_h^\perp\| + \|u_h^\perp\|  \|u_k^\perp\|  \right ]   \right \}
    \cr &   \le 
   \l h_0   \sum_{|k \l| \le h_0}    \sum_{|h \l|  \le  h_0, h \neq k}        
   \frac {1} {(1+ |k-h|)^3}   
    \left [  |\a_k|  \|u_h^\perp\| +\frac 12  \|u_h^\perp\|^2 + \frac 12  \|u_k^\perp\|^2 \right  ] 
    \cr &  \le 
      \l h_0  C  \left \{ 
    \sum_{|h \l|  \le  h_0}        \|u_h^\perp\|^2  + \sum_{|k \l| \le h_0}    \sum_{|h \l|  \le  h_0, h \neq k}        
   \frac {1} {(1+ |k-h|)^3}   
    \  |\a_k|  \|u_h^\perp\| \right \}.
    \end {split}
 \end {equation}
 Taking into account    \eqref {s1} and  \eqref {s2},  applying estimate \eqref {d501}
 and proceeding as above we get 
  \begin {equation} \label {mac2z}    \begin {split}  &  
   \sum_{|k \l| \le h_0}    \sum_{|h \l|  \le  h_0, h \neq k}    \left | F^\l_2 (u_h,u_k)+ F^\l_3 (u_h,u_k) \right |    \cr & \le 
   \l C\left   \{   
    \sum_{|h \l|  \le  h_0}        \|u_h^\perp\|^2  + \sum_{|k \l| \le h_0}    \sum_{|h \l|  \le  h_0, h \neq k}        
   \frac {1} {(1+ |k-h|)^3}   
    \  |\a_k|  \|u_h^\perp\|  \right\}. 
   \end {split}
 \end {equation}
 Next we  consider the case when  $ |k \l| \le h_0$ and     $ |h \l|  >  h_0$.
By \eqref {d502},  we get  
  \begin {equation} \label {s20a}    \begin {split}  &  |  F^\l_1 (u_h,u_k) | \le \l^2  |k|     
   \frac {1} {(1+ |k-h|)^3}   \frac {1} {(1+ |\frac 12 (3k-h)| \l )}  \|u_h \| [ \|u_k^\perp\|  +  |\a_k|]  \cr & \le
      \l h_0     
   \frac {1} {(1+ |k-h|)^3}      \left [  \frac 12 \|u_h \|^2 +  \frac 12 \|u_k^\perp\|^2  +|\a_k| \|u_h \|  \right ]. 
 \end {split}
 \end {equation}     
 By \eqref {d501}, proceeding as  above, 
    \begin {equation} \label {s20b}   |F^\l_2 (u_h,u_k)+  F^\l_3 (u_h,u_k)|  \le   \l   
   \frac {1} {(1+ |k-h|)^3}      \left [  \frac 12 \|u_h \|^2 +  \frac 12 \|u_k^\perp\|^2  + |\a_k| \|u_h \|  \right ].  
 \end {equation} 
 Therefore 
  \begin {equation} \label {s20}    \begin {split}  &  
 \sum_{|k \l| \le h_0}    \sum_{|h \l|  >  h_0}   \sum_{i} |F^\l_i (u_h,u_k)|  \cr &\le 
  \l  C   \sum_{|k \l| \le h_0}    \sum_{|h \l|  >  h_0}      
   \frac {1} {(1+ |k-h|)^3}    \left [  \frac 12 \|u_h \|^2 +  \frac 12 \|u_k^\perp\|^2  + |\a_k|  \|u_h \|  \right ] 
    \cr & \le
    \l C  \sum_{|k \l| \le h_0}   \|u_k^\perp\|^2 +   \l C   \sum_{|h \l| > h_0}   \|u_h\|^2 +    
    \l C       \sum_{|k \l| \le h_0}    \sum_{|h \l|  >  h_0}      
   \frac {1} {(1+ |k-h|)^3}  |\a_k|        \|u_h \|.   
   \end {split}
 \end {equation} 
  When    $ |k \l| >h_0$ and     $ |h\l|  <  h_0$   the estimate of  $ |F^\l_2 (u_h,u_k)+  F^\l_3 (u_h,u_k)|$  gives
 similar terms as in \eqref {s20b}, with $h$ replacing $k$.
 To estimate  $F^\l_1$ we note that $|k| \le |k-h|+ |h|$.
 Therefore
 by \eqref {d502}  \begin {equation} \label {or1}    \begin {split}  &  
    \sum_{|k \l| > h_0}    \sum_{|h \l|  \le  h_0}    |F^\l_1 (u_h,u_k)|  \cr & \le  \l^2    \sum_{|k \l| > h_0}    \sum_{|k \l|  \le  h_0}     |k|     
   \frac {1} {(1+ |k-h|)^3}   \frac {1} {(1+ |\frac 12 (3k-h)| \l )}  \|u_h \|   \|u_k\|   \cr &
    \le 
  \l^2    \sum_{|k \l| > h_0}    \sum_{|h \l|  \le  h_0}     |h|     
   \frac {1} {(1+ |k-h|)^3}    \|u_h \|   \|u_k\|   \cr & +
    \l^2    \sum_{|k \l| > h_0}    \sum_{|h \l|  \le  h_0}      |k- h|     
   \frac {1} {(1+ |k-h|)^3}    \|u_h \|   \|u_k\|    
    \cr &   \le 
   \l h_0   \sum_{|k \l| > h_0}    \sum_{|h \l|  \le  h_0}        
   \frac {1} {(1+ |k-h|)^3}   
      \|u_h \|   \|u_k\|   + \l^2    \sum_{|k \l| > h_0}    \sum_{|h \l|  \le  h_0}        
   \frac {1} {(1+ |k-h|)^2}   \|u_h \|   \|u_k\|.
    \end {split}
 \end {equation}
Since  $|h \l|  \le  h_0$, we split  $u_h$  as in \eqref {d11}. 
 Insert this decomposition  into \eqref {or1} we get 
 \begin {equation} \label {s20e}    \begin {split}  &   \sum_{|k \l| > h_0}    \sum_{|h \l|  \le   h_0}     |F^\l_1 (u_h,u_k)|  \cr & \le
    \l    \sum_{|k \l| > h_0}   \|u_k\|^2 +   \l   \sum_{|h \l| < h_0}   \|u_h^\perp \|^2 +    
    \l          \sum_{|k \l| > h_0}    \sum_{|h \l|  \le  h_0}        
   \frac {1} {(1+ |k-h|)^3}    |\a_h|  \|u_k\|.   
     \end {split}
 \end {equation}
 Next  we consider the case when   $ |k \l| >h_0$ and     $ |h \l|  >  h_0, h \neq k$.  The main point is to   control the $|k|$ in the  estimate of $F^\l_1$.    
   When $k$ and $ h$  have opposite sign, then
 \begin {equation} \label {ol1}    
   \frac { |k|  } {  1+ |k-h|}    \le 1. \end {equation}
  When   $k$ and $h$  have the same sign
  it might be that  $  \frac { |k|  } { 1+ |k-h| }>1$.    
  But if this is case, then  for these values of $h$ and $k$ \begin {equation} \label {ol2}      \frac {k \l} {(1+  \frac 12 (3k-h)  \l )} \le 2. \end {equation}
    This  statement can be easily verified. 
Take $k>0$ the case with $k<0$ can be treated in the same way.
   Assume that \eqref {ol1} does not hold. This means that, 
    for given $k$, $ |k-h| <   k-1$. This inequality is satisfied  for  $ h \in \{1, \dots, 2k-1\}$. For such value of $h$ we have
  that   $ 3k-h \ge k$ and therefore 
 \eqref  {ol2}  holds. 
  Hence 
    \begin {equation} \label {s21}   
   \sum_{|k \l| > h_0}    \sum_{|h \l|  >  h_0, h \neq k}   F^\l_1 (u_h,u_k) =   \sum'_{|k \l| > h_0}    \sum'_{|h \l|  >  h_0, h \neq k}    F^\l_1 (u_h,u_k)  +   \sum''_{|k \l| > h_0}    \sum''_{|h \l|  >  h_0, h \neq k}   F^\l_1 (u_h,u_k),   \end {equation}
where $ \sum'$ is a sum restricted   to the $h$ and $k$ so that \eqref {ol1} holds,
and $\sum''$ is a sum restricted   to the $h$ and $k$ so that \eqref {ol1}  does not hold.
  By \eqref {d502}  and \eqref {ol1} we get 
   \begin {equation} \label {s21a}    \begin {split}  &   \sum'_{|k \l| > h_0}    \sum'_{|h \l|  >  h_0, h \neq k}    F^\l_1 (u_h,u_k)    \cr &\le 
  \l^2  C \sum'_{|k \l|> h_0}    \sum'_{|h \l|  >  h_0, h \neq k}   |k|     
   \frac {1} {(1+ |k-h|)^3}   \frac {k \l} {(1+  \frac 12 (3k-h)  \l )}\|u_h \| \|u_k\|    \le  \l^2  C\sum_{|k \l|> h_0}    \sum_{|h \l|  >  h_0, h \neq k}        
   \frac {1} {(1+ |k-h|)^2}   \|u_h \| \|u_k\|   \cr & 
   \le  \l^2 C    \sum_{|k \l|> h_0}    \sum_{|h \l|  >  h_0, h \neq k}        
   \frac {1} {(1+ |k-h|)^2}   [ \frac 12 \|u_h \|^2 + \frac 12  \|u_k\|^2]  \le \l^2 C   \sum_{|k \l|  >  h_0}\|u_k \|^2.
      \end {split}
 \end {equation}  
 By \eqref {d502}  and \eqref {ol2} we get 
   \begin {equation} \label {s21a}    \begin {split}  &   \sum''_{|k \l| > h_0}    \sum''_{|h \l|  >  h_0, h \neq k}    F^\l_1 (u_h,u_k)    \cr &\le 
  \l^2C  \sum''_{|k \l|> h_0}    \sum''_{|h \l|  >  h_0, h \neq k}   |k|     
   \frac {1} {(1+ |k-h|)^3}   \frac {k \l} {(1+  \frac 12 (3k-h)  \l )}\|u_h \| \|u_k\|    \le  \l  C\sum_{|k \l|> h_0}    \sum_{|h \l|  >  h_0, h \neq k}        
   \frac {1} {(1+ |k-h|)^3}   \|u_h \| \|u_k\|   \cr & 
   \le   \l  C     \sum_{|k \l|> h_0}    \sum_{|h \l|  >  h_0, h \neq k}        
   \frac {1} {(1+ |k-h|)^3}   [ \frac 12 \|u_h \|^2 + \frac 12  \|u_k\|^2]  \le C \l    \sum_{|k \l|  >  h_0}\|u_k \|^2.
      \end {split}
 \end {equation}  
         
    By  the previous estimates we have
    \begin {equation} \label {dom1}    \begin {split}  
 &\sum_{k, h; k \neq h }  [ F^\l_1 (u_h,u_k)+  F^\l_2 (u_h,u_k)+  F^\l_3 (u_h,u_k)]  \cr & \le
   \l C  \left (   \sum_{|k \l| \le h_0}   \|u_k^\perp\|^2 +    \sum_{|k \l| > h_0}   \|u_k\|^2  \right )\cr &
   + \l C \sum_{|k \l| \le h_0}    \sum_{|h \l|  \le  h_0, h \neq k}        
   \frac {1} {(1+ |k-h|)^3}   
       |\a_k|  \|u_h^\perp\| \cr & + 
     \l C       \sum_{|k \l| \le h_0}    \sum_{|h \l|  >  h_0}      
   \frac {1} {(1+ |k-h|)^3}  |\a_k|        \|u_h \| \cr & +
   \l        C  \sum_{|k \l| > h_0}    \sum_{|h \l|  \le  h_0}        
   \frac {1} {(1+ |k-h|)^3}    |\a_h|  \|u_k\|. 
 \end {split}
 \end {equation}     
  Define
 $$ b_h = \left \{ \begin {split}  &   \|u_h^\perp \|, \quad  |h\l| \le h_0\cr &
  \|u_h \|, \quad  |h\l| > h_0. \end {split} \right .$$
 Adding to the  \eqref {dom1} the terms  on the diagonal, i.e $F^\l_i (u_k,u_k)$, for $i=1,2,3$,    we get 
  \begin {equation} \label {dom2}    \begin {split}  
 & \left |\sum_{k, h }  [ F^\l_1 (u_h,u_k)+  F^\l_2 (u_h,u_k)+  F^\l_3 (u_h,u_k)] \right |  \cr & \le
   \l C  \left (   \sum_{|k \l| \le h_0}   \|u_k^\perp\|^2 +    \sum_{|k \l| > h_0}   \|u_k\|^2  \right )\cr &
   + \l C\sum_{|k \l| \le h_0}    \sum_{h}        
   \frac {1} {(1+ |k-h|)^3}   
       |\a_k|  b_h.
        \end {split}
 \end {equation}     
By Schwartz
\begin {equation} \label {giov1}    \begin {split}   & \sum_{|k \l| \le h_0}    \sum_{h}        
   \frac {1} {(1+ |k-h|)^3}   
       |\a_k|  b_h \le \sqrt {   \sum_{|k \l| \le h_0}    \sum_{h}        
   \frac {1} {(1+ |k-h|)^3}   
       |\a_k|^2 } \sqrt {  \sum_{|k \l| \le h_0}    \sum_{h}        
   \frac {1} {(1+ |k-h|)^3}   
       b_h^2   }  \cr &  \le 
       C    \sqrt {   \sum_{|k \l| \le h_0}     
       |\a_k|^2 }  \sqrt {  \sum_{h}       b_h^2}  \le  C \sqrt {  \sum_{h}       b_h^2}.
       \end {split}
 \end {equation}     
        Taking into account \eqref {ele1} and \eqref {ele11}   and estimates
 \eqref {d100}, \eqref {d101}  and \eqref {dom1}  we have   
 \begin {equation} \label {d010}    \begin {split}   & \langle V, \A V \rangle  -  \langle V,R^\l_0 V \rangle    =   \sum_{k, h}    \left  \{    \d_ {h,k} \langle u_h,     \LL^{k\l} u_k \rangle  
+         F^\l_1 (u_h,u_k)  +  F^\l_2 (u_h,u_k) + F^\l_3 (u_h,u_k)\right \} \cr & \ge
  \sum_{|k| \le \frac {h_0} \l }      \mu_0^{k\l} \a_k^2   +   [D-C\l]  \sum_{|k| \le \frac {h_0} \l }       \|u_k^\perp\|^2   +   [\nu-\l C]  \sum_{|k| > \frac {h_0} \l }         \|u_k\|^2  -
   \l    \sqrt {  \sum_{h}       b_h^2}  
    \end {split}
 \end {equation}
   Since  $  \langle V, \A V \rangle \le C \l^2$,  see \eqref {gi8},   and    $|\langle V,R^\l_0 V \rangle| \le C \l^2$, see    \eqref {ele11},  
  taking $C \le   \min\{ [D-C\l],  [\nu-\l C]\}$ we get 
 \begin {equation} \label {d011}    \begin {split}   &   C \l^2    
    \ge
     \sum_{|k| \le \frac {h_0} \l }      \mu_0^{k\l} \a_k^2   + C \sum_{k}  b_k^2   -
   \l    \sqrt {  \sum_k     b_k^2}  \cr &\ge C   \sum_{k}  b_k^2   -
   \l    \sqrt {  \sum_k     b_k^2}. 
    \end {split}
 \end {equation}
 In the last inequality we use that $ \mu_0^{k\l} > \mu_0^0>0$, see  Proposition \ref {tu2} and  Theorem \ref {81}. 
 Therefore
  \begin {equation} \label {d012}    \begin {split}   & 
  C \l^2 +   \l    \sqrt {  \sum_k     b_k^2} \ge   C  \sum_{k}  b_k^2. 
      \end {split}
 \end {equation}
 This inequality immediately implies
 \begin {equation} \label {d013}     \sqrt {  \sum_k     b_k^2}  \le C \l.   \end {equation}
   Lower bounding \eqref {d010}  by \eqref  {d013} 
 we get
  \begin {equation} \label {d020}    \begin {split}   &   C \l^2    
    \ge
    \sum_{|k| \le \frac {h_0} \l }      \mu_0^{k\l} \a_k^2   +    [D-C\l]  \sum_{|k| \le \frac {h_0} \l }      \|u_k^\perp\|^2   +   [\nu-\l C]   \sum_{|k| > \frac {h_0} \l }        \|u_k\|^2. 
    \end {split}
 \end {equation}
We  get  similar estimates as  in Subsection 6.1 (Toy  model). This implies 
   \begin {equation} \label {f0} \sum_{|k| \le \frac {h_0} \l }      \mu_0^{k\l} \a_k^2  \le  C \l^2,  \end {equation}
     \begin {equation} \label {f1}   \sum_{|k \l|  \le  h_0, } \|u_k^\perp\|^2 \le C \l^2,  \end {equation}
     \begin {equation} \label {f2}   \sum_{|k \l|  >  h_0, } \|u_k\|^2 \le C \l^2.  \end {equation}
 
We  therefore define 
\begin {equation} \label {f3}  Z(s) =  \sum_{|k| \le \frac {h_0} \l }   e^{iks} \a_k  \end {equation}
\begin {equation} \label {f4}  V^R(s,z)= \sum_{|k| \le  \frac {h_0} \l } e^{i  k  s}  u_k^\perp (z)+  \sum_{|k| > \frac {h_0} \l } e^{i  k s}  u_k(z)  +    \sum_{|k| \le  \frac {h_0} \l } e^{i   k  s}  \a_k      [\psi_0^{k\l}(z)-  \psi_0^0(z)].   \end {equation}
Then
\begin {equation} \label {f5}  V(s,z)=   Z(s)  \psi_0^0(z) +   V^R(s,z)   \end {equation}
and, proceeding as in Subsection 6.1,  the requirements \eqref {gi8b} hold. 
        \qed

\section {Proof of Theorem \ref {86}}

\noindent  \proof  If $v=0$ then the assertion of the theorem  holds. 
So  let $v$ be  non identically equal to zero and assume  that $\int_\Omega v^2 (\xi) d\xi =1$. 
Proceeding as in the proof of Theorem \ref {82} there exists $\bar k \in \{0, \dots N \}$, with $N= [ \frac 1 \l ]$, and   cut off functions $ \eta_1= \eta^{ \bar k}_1$ and   $\eta_2 =  1- \eta^{\bar k}_1$ 
so that 
\begin {equation} \label {gch2a}  \begin {split}     \int_{ \Om}  ( A^\l_{m_A} v)  (\xi) v(\xi) \dha    \xi &   = \int_{ \Om}  ( A^\l_{m_A} \eta_1 v)  (\xi)  \eta_1 (\xi) v(\xi) \dha    \xi   \cr & +
 \int_{ \Om}  ( A^\l_{m_A} \eta_2 v)  (\xi)  \eta_2 (\xi) v(\xi) \dha    \xi   \cr &  -
2 \int_{ \Om} \dha    \xi  \eta_1 (\xi) v  (\xi) (J^{\l} \star  \eta_2 v)(\xi),
\end {split}
\end {equation}  
\begin {equation} \label {GS2} \int_{ \Om}  ( A^\l_{m_A} \eta_2 v)  (\xi)  \eta_2 (\xi) v(\xi) \dha    \xi   \ge (C^*-1) \| \eta_2   v\|_{L^2 (\Om)} >0, \end {equation}
and
\begin {equation} \label {GS1} \left |2 \int_{ \Om} \dha    \xi  \eta_1 (\xi) v  (\xi) (J^{\l} \star  \eta_2 v)(\xi) \right  | \le \l^2  C \| v\|^2_{L^2 (\Om)}. \end {equation}
 Without loss of generality we  keep on denoting by   $\NN(d_0)$  the set   where    $ \eta_1(\xi)=1$.  The 
 \eqref {v1},  \eqref {GS2} and  \eqref {GS1}  imply  that 
\begin {equation} \label {S.201} \int_{\NN(d_0)} \left ( A^\l_{m_A} \eta_1 (\xi) v  (\xi) \right ) \eta_1 (\xi) v  (\xi)   d\xi  \le C\l ^2  -   \int_{\Om \setminus \NN(d_0) }  (A^\l_{m_A}  \eta_2v ) (\xi)  \eta_2 v(\xi)\dha    \xi  \le  C \l^2.
  \end {equation}
   By  \eqref {v1}  and \eqref {GS1}
  \begin {equation} \label {D.13}  \int_{\Omega \setminus \NN(d_0)}  (A^\l_{m_A}\eta_2 v)  (\xi) v(\xi) d\xi   \le C\l^2 -  \int_{\NN(d_0)}  
\left (A^\l_{m_A}   \eta^{N}_1 (\xi) v  (\xi) \right )  \eta_1 (\xi)v(\xi) d\xi.     \end {equation} 
 By  Lemma \ref {F2} and Lemma \ref {M2} we have that 
  \begin {equation} \label {D.12} \int_{\NN(d_0) }  (A^\l_{m_A}  \eta_1v ) (\xi)  \eta_1 v(\xi)\dha    \xi  \ge - C \l^2 \int_\Om v^2  (\xi) d\xi = -C \l^2.   \end {equation}
Then, from \eqref {D.13}, taking into account \eqref {D.12} we obtain
 \begin {equation} \label {D.14} \int_{\Omega \setminus \NN(d_0)}   (A^\l_{m_A}  \eta_1v ) (\xi)  v(\xi) d\xi   \le  C\l^2. \end {equation} 
This together with \eqref {GS2} implies
  \begin {equation} \label {S.4}\int_{\Omega
\setminus \NN(d_0)} v^2(\xi)\dha     \xi   \le C \l^2, \end {equation} 
hence
  \begin {equation} \label {S.5}\int_{  \NN(d_0)} v^2(\xi)\dha     \xi   \ge 1-  C \l^2.   \end {equation} 
Therefore, see \eqref {S.201} and \eqref {S.5},  we can apply Theorem  \ref {g10} 
  decomposing $v$ as 
 \begin {equation} \label {S.5a} v(r,s) =    Z (s) \frac 1 {\sqrt {\a (s,r)} }    \frac 1 {\sqrt \l} \psi^0_0 (\frac r \l)+   v^R(s,r),
 \end {equation} 
 where $\psi^0_0 (\cdot)$ is the first eigenvalue of $ \LL^0$, see Theorem \ref {81}, 
with 
  \begin {equation} \label {D.600} \| v^{R} \|^2_{L^2(\NN(d_0)} \le \l^2 C, \end {equation} 
  and  \begin {equation} \label {fusco.2aa}   1- C \l^2 \le  \| Z \|^2_{L^2 (T)}  \le1, \qquad  \|\nabla Z \|_{L^2 (T)}   \le C.
 \end {equation}
Set  
\begin {equation} \label {D.5}   \Delta  w  = v.   \end {equation}
 Denote 
$$ \hat w = \frac 1 {\sqrt \l} w.  $$
Then,  from \eqref  {S.5a} 
 $$   \Delta   \hat w  = \frac 1 {\sqrt \l} \left [   Z  (s) \frac 1 {\sqrt {\a (s,r)} }    \frac 1 {\sqrt \l} \psi^0_0(\frac r \l)+   v^R(s,r) \right ].
  $$ 
To show \eqref {S.200}  it is enough to prove that 
\begin {equation} \label {D.51}  \|\nabla \hat  w \|_{L^2(\Omega)} \ge C  \end {equation}
for some positive $C$ independent on $\l$. 
Let $ \d \in (0, \frac 1 2 ] $ a small constant to be determined and let 
$\chi \in C^\infty_0 (\R) $ be  a
cut-off function,  such that
\begin {equation} \label {D.9}   \chi (x) =1 \qquad \hbox {if} \qquad |x|  \le \frac 12, \qquad
\qquad   \chi (x) =0 \qquad
\hbox {if} \qquad |x|>1,\qquad \qquad  x \chi'(x) \le 0 \qquad \hbox {in}
\qquad  \R.
\end {equation}
Set $\chi^\d (x)= \chi (\frac x \d)$, with $\chi $ as in \eqref {D.9}. 
We have
\begin {equation} \label {D.10}\begin {split} & \int_{\NN(d_0)} Z (s(\xi)) \chi^\d (r(\xi,\G)) \Delta 
 \hat w (\xi)  d \xi \cr & =  \int_{ T \times [-\d, \d] } Z (s) \chi^\d (r) \left \{ 
   Z (s) \frac 1 {\sqrt {\a (s,r)} }    \frac 1 {  \l} \psi^0_0 (\frac r \l)+   \frac 1 {\sqrt \l} v^R(s,r)  \right \} \a(s,r)\dha    r\dha    s.   \end {split} \end {equation}
By  \eqref {8.4},  for $ \d >\l$,  $\int_{  [-\d, \d] }
   \frac 1 {  \l} \psi^0_0 (\frac r \l) \dha     r  \ge   2 \bar m (\frac \d \l) + Ce^{-\a \frac \d \l} \ge C $. This, together      with      \eqref {fusco.2aa},  allows to lower bound  
 \begin {equation} \label {D.11}\begin {split}& \int_{ T \times [-\d, \d] }   \sqrt {\a(r,s)} Z^2 (s) \chi^\d (r)      \frac 1 {  \l} \psi^0_0 (\frac r \l) \dha     r\dha     s  \cr & \ge C \inf_{\{(r,s) \in  T \times [-\d, \d]  \}}  \sqrt {\a(s,r)} \|
Z^2\|^2_{L^2(T)}  \int_{  [-\d, \d] }
   \frac 1 {  \l} \psi^0_0(\frac r \l) \dha     r   \ge C \|
Z^2\|^2_{L^2(T)}  \ge C (1-\l^2). 
\end {split} \end {equation}
 By Schwartz inequality
 \begin {equation} \label {D.120}\begin {split}&   \left | \frac 1 {\sqrt \l} \int_{ T \times [-\d, \d] } Z (s) \chi^\d(r)  
   v^{R} (r,s) \a (r,s)\dha     r\dha     s \right |  \cr &  \le 
\frac 1 {\sqrt \l} \| v^{R}\|_{L^2(\NN(\d))} \left ( \int_{ T \times [-\d, \d] }   \a (r,s) Z^2 (s) [ 
\chi^\d (r) ]^2\dha     r\dha     s  \right )^{\frac 12 }  \cr & \le  \frac 1 {\sqrt \l} \| v^{R}\|_{L^2(\NN(\d))} 
\left ( \sup_{\{(r,s) \in  T \times  [-\d, \d]   \}} |\a (r,s)|\right )^{\frac 12 }  \|Z\|_{L^2(T)} \d^{\frac 12 } \le
C \d^{\frac 12} \l^{\frac 12}.  \end {split} \end {equation}
 The last inequality  is obtained applying \eqref {D.600}. Then from \eqref {D.10}, \eqref {D.11} and \eqref {D.120} we obtain  
\begin {equation} \label {D.130} \int_{  \NN(d_0)} Z (s(\xi)) \chi^\d (r(\xi,\G))  \Delta \hat w   (\xi)    \dha     \xi    \ge C(1- \l^2) - C\d^{\frac 12} \l^{\frac 12}.  
  \end {equation}
 On the other hand  we can calculate
\begin {equation} \label {D.18}\begin {split}& \int_{\NN(d_0)} Z(s(\xi)) \chi^\d (r(\xi,\G))  \Delta  \hat w  (\xi)   \dha     \xi \cr & = -
\int_{\NN(d_0)} \nabla \left \{  Z(s(\xi)) \chi^\d (r(\xi,\G)) \right \}  \cdot \left  \{
\nabla \hat w  (\xi) \right \}  \dha     \xi \cr & =
 -
\int_{\NN(d_0)}\left \{  \nabla  Z(s(\xi)) \chi^\d (r(\xi,\G)) + Z (s(\xi)) \nabla \chi^\d ( 
r(\xi,\G))  \right \}  \cdot  \nabla \hat w  (\xi)  \dha     \xi \cr &
  \le   \|\nabla \hat w \|_{L^2(\NN(\d))} \left \{ \|\nabla  Z\|_{L^2(T)} \d^{\frac 12}+ \sup_{r \in [-\d, \d]}  |\nabla
\chi^\d (r)|  \d^{\frac 12}\|   Z \|_{L^2(T)}  \right \}  \cr & \le \|\nabla \hat w \|_{L^2(\NN(\d))} \left [ \d^{\frac 12}+
\d^{-\frac 12}\right ]C,   \end {split} \end {equation}
where we   estimated    $ \|   Z \|_{L^2(T)}$ and $\|\nabla  Z \|_{L^2(T)} $,  as in    \eqref {fusco.2aa} and 
$$\left ( \int_{\NN(d_0)} (\nabla  Z (s(\xi))^2 \chi^\d (r(\xi,\G)) \dha     \xi  \right )^{\frac 12} \le C  \|  \nabla  Z\|_{L^2(T)} \d^{\frac 12}. $$
Combining  \eqref {D.18} with the estimates \eqref {D.130} we obtain  for $\delta$ small enough
$$\|\nabla \hat w \|_{L^2(\NN(\d))} \ge \frac { C(1- \l^2) - \d^{\frac 12}  \l^{\frac 12}} { \left [ \d^{\frac 12}+
\d^{-\frac 12}\right ] }  \ge C >0, $$ 
hence \eqref {D.51}. The  theorem is proved.
\qed

\section {Appendix}
 
 \vskip0.5cm
\begin {lem} \label  {exp}   Let $ \mu (s)$ be  any eigenvalue of $\LL^s$ such that 
 \begin{equation}   \label {exp.1}    
\e_0=  \frac 1 {\s(m_\b)} -1-  \sup_{s \in T} \mu(s) >0. 
\end {equation}
Let $ \psi (s, \cdot)$ be   one of the  normalized   eigenfunctions  corresponding to $ \mu(s)$. 
There exists   $z_0= z_0(\e_0)>0$ and  $ \l_0 \equiv \l_0 (\e_0) >0$ such that  for $ \l \in (0, \l_0]$,  we have
that for  $|z| \ge z_0$, for all $s \in T$
\begin{equation}   \label {exp.2} |\psi (s,z) | \le  e^{-\alpha (\e_0)(|z| - z_0 )} \|\bar J\|_2,
\end {equation}
where $\alpha (\e_0)$ is given in \eqref  {exp.10}.
\end {lem}
\begin {proof} 
Let  $ \mu (s)$ and $\psi (s, \cdot)$ be as in the hypothesis, then 
\begin{equation}   \label {ag.1}  \frac {\psi (s,z)} { \s(m_A(s, \l z))}- (\bar J \star_z \psi) (s,z) = \mu (s) \psi (s,z) \qquad z \in  I_\l.   \end {equation}
By the definition of $m_A$, see \eqref {8.16} and \eqref {2.5}, there exists $C>0$ so that
\begin{equation}   \label {rm.1} \sup_{s \in T}  | \s(m_A(s, \l z))- \s(\bar m(z))| \le \l C.\end {equation}
 Since  $ \lim_{|z| \to \infty} \s(\bar m(z)) =\s(m_\b) $, there exists    $z_0= z_0(\e_0)>0$ so that  for $ |z| \ge z_0$ 
 \begin{equation}   \label {rm.2}   
 \frac 1 {\s(\bar m (z))} -1 >\sup_{s \in T} \mu(s)  + \frac {\e_0} 2.   
\end {equation} 
 Set  for $ |z| \ge z_0 $ 
$$ A^0(s,z)= \frac 1 { \frac 1{ \s(\bar m (z))}-\mu(s) }=    \frac { \s(\bar m (z))} {1-\mu(s)  \s(\bar m (z)) }.  $$
By \eqref {rm.2}  there exists $ \e_1  = \e_1 (\e_0)$ so that   for all $s \in T$ 
 \begin{equation}   \label {rm.4}   0<A^0(s,z)<1- \e_1 , \qquad   |z| \ge z_0. \end {equation}
Set  for $ |z| \ge z_0 $ 
$$ A^\l (s,z)= \frac 1 { \frac 1{ \s(m_A(s, \l z))}-\mu(s) }=    \frac { \s(m_A(s, \l z))} {1-\mu(s)  \s(m_A(s, \l z)) } .$$
   By \eqref {rm.1} we have that
\begin{equation}   \label {rm.5}  | A^\l (s,z)- A^0(s,z)| \le C \l,  \quad \forall s \in T.\end {equation}
Choose then $ \l_0= \l_0(\e_0) $  small enough so that
that for $
\l
\le
\l_0$,  
$z_0
< \frac 1
{2 \l} $  and 
\begin{equation}   \label {rm.3}  | A^\l (s,z)- A^0(s,z)| \le \frac{ \e_1} 2  \quad \forall s \in T. \end {equation}
 By \eqref {rm.4}  and \eqref  {rm.3} we have that 
 \begin{equation}   \label {rm.6}  A^\l (s,z) 
 \le  A^0(s,z)+ \frac{ \e_1} 2 <  1-  \frac{ \e_1} 2, \qquad \forall s \in T, \qquad \forall |z| \ge z_0.  \end {equation}
From \eqref {ag.1}
 \begin{equation}   \label {L.1}  \psi (s,z)= A^\l(s,z) (\bar J \star_z \psi) (s,z).  \end {equation}
    Suppose    $z=z_0+n$ where 
$n$ is any integer  such that
$ z_0 +2n
\le
\frac 1
\l$.  Same can be done
when
$z<0$ and by simple interpolation argument for any $z \in [z_0+n, z_0+n+1)$.   We have, see
\eqref {L.1},
 \begin{equation}   \label {exp.4}  |\psi (s,z_0  +n)|  \le A^\l (s,z_0+n)  |(\bar J \star_z \psi) (s,z_0+ n)|.  
  \end {equation}
We iterate  $n$ times \eqref {exp.4}. The support of $n$ fold convolution is  the interval $[ z_0,
z_0+2n] $. By \eqref {rm.6}   we obtain the following estimate

\begin{equation}     \begin {split} &      |\psi ( s,z_0+n)|  \le [1- \frac {\e_1} 2 ]^n  |((\bar J)^n \star_z \psi )(s,z_0+n)|  \cr & \le   [1- \frac {\e_1} 2 ]^n \|(\bar J)^n\|_2 \|\psi (s, \cdot)
\|_2=  [1- \frac {\e_1} 2 ]^n \|\bar J \|_2  = e^{-n \a(\e_0)}  \|\bar J\|_2  
\end {split} \end {equation}
where, since $\e_1= \e_1(\e_0)$,
  \begin{equation}   \label {exp.10} \a(\e_0) = \log \frac 1 { [1- \frac {\e_1} 2]  }  \end {equation}
 The thesis follows.  \end {proof}
 
 \vskip0.5cm
  {\bf Proof of Lemma \ref {M0}} Take   $ \xi$ and $\xi'$ in  $\NN (d_0)$.  Write 
  in     local  variables   $  \xi =  \g(s) + \nu (s) r$ and  $    \xi' =  \g(s') + \nu (s') r'$.
  It is convenient to  express the  difference
     $$  \xi - \xi' =  \g(s) + \nu (s) r -[ \g(s') + \nu (s') r'] $$ in term of $s^*= \frac {s+s'} 2$ and $r^*= \frac {r+r'} 2$, the middle points  between $s$ and $s'$ and $r$ and $r'$ respectively.
     This allows to  get  some cancellations. 
  We have 
       $$\g(s)= \g(s^*) + \g'(s^*)(s-s^*) + \frac 12 \g''(s^*)(s-s^*)^2 + \frac 16 \g'''( \tilde s)(s-s^*)^3,   $$
 $$\nu (s)= \nu(s^*) + \nu'(s^*)(s-s^*) + \frac 12 \nu''(s^*)(s-s^*)^2 + \frac 16 \nu'''( \tilde s )(s-s^*)^3   
         $$
 where $\tilde s \in (s,s')$ and it might change from one occurrence to the other. Similarly  expressions hold for $\g(s')$ and $\nu (s')$.
  Since $s-s^*= \frac {s-s'}2$  and $s'-s^*= \frac {s'-s}2$
 we have  $$ \g(s)- \g(s')=  \g'(s^*)(s-s')  +    \frac 1{24} \g'''(\tilde s)(s-s')^3 .$$
Note that the term 
 $$\frac 12 \g''(s^*)(s-s^*)^2- \frac 12 \g''(s^*)(s'-s^*)^2 =0 .$$
 Further
  \begin {equation}     \begin {split}    \nu (s)r - \nu (s')r'     & =
   \nu(s^*) [r-r'] +  \nu'(s^*) (s-s') r^* \cr & + \frac 18 \nu''(s^*)(s-s')^2 (r-r')   +  \frac 1 {24} \nu'''( \tilde s)(s-s')^3  r^*.
    \end {split}     \end {equation}
 Taking into account that      $\nu' (s)= -k(s) \g'(s)$, we have 
      \begin {equation}    \label {m2}  \begin {split}   \xi - \xi' = &
       \g'(s^*) (s-s')[1-  k(s^*)r^*] + \nu (s^*) (r-r')   \cr &+
      \frac 18    \nu'' (s^*) (s-s')^2 (r-r')   
   \cr & +  \frac 1 {24} \nu'''(  \tilde s )(s-s')^3 r^*  +\frac 1{24} \g'''(\tilde s)(s-s')^3.   
      \end {split}     \end {equation}
Denote by $ a$ and $b$ the following vectors
$$ a=      \g'(s^*) (s-s')[1-  k(s^*)r^*] + \nu (s^*) (r-r'),  $$
$$  b=      \frac 18   \nu'' (s^*)(s-s')^2 (r-r')   
    +  \frac 1 {24} \left [ \nu'''( \tilde s )(s-s')^3 r^* + \g'''(\tilde s )(s-s')^3 \right ].   
         $$
         It is important to notice that for $|s-s'| \le \l$ and $|r-r'|\le \l$ we have
         $|b| \le C \l^3$.
          By Taylor expansion up to the second order  of $ J^\l (\cdot)$ we get that there exists $c^* \in \R^2$
      so that     
         \begin {equation}    \label {m3}     
      J^\l (\xi  -  \xi')  =      J^\l (a)  +  (\nabla  J^\l) (a) \cdot b + \frac 12   b \cdot  ( D^2J^\l )(c^*)\cdot b,  \end {equation}
             where we denote  
         by  $ (\nabla  J^\l) (a)$ the gradient of $J^\l$  computed in  $a$ and by  $D^2J^\l (c^*)$ the matrix of the second derivatives of $J^\l (\cdot )$ computed  at $c^*$.
        Notice that   $   \|\nabla  J^\l \|_\infty \le  C\l^{-d-1} $,  since $\l^{-d}$ comes from the normalization  and $\l^{-1}$ by differentiating  
       one time,  $   \|D^2J^\l \|_\infty \le  C\l^{-2-d} $,  since $\l^{-d}$ comes from the normalization  and $\l^{-2}$ by differentiating   twice.     When $|s-s'|\le \l$ and   $|r-r'|\le \l$
         we then obtain 
         $$ | (\nabla  J^\l) (a) \cdot b| \le    C\l^{2-d}$$
        $$  | b \cdot  ( D^2J^\l )(c^*)\cdot b | \le C\l^{4-d}. $$
         Define
$$ J^\l (s,s', r, r')= J^\l (a) =  J^\l ( (s-s')[1-  k(s^*)r^*],  (r-r') )  $$
\begin {equation} \label  {sm4} R_1^\l (s,s', r, r') =  \ (\nabla  J^\l) (a) \cdot b   \end {equation}
\begin {equation} \label  {sm5} R_2^\l (s,s', r, r')   =  \frac 12   b \cdot  ( D^2J^\l )(c^*)\cdot b.  \end {equation}
Therefore, see \eqref {se1}, we obtain 
 \begin {equation} \label  {VM1}  \begin {split}   
 \int_{\NN (d_0)} J^\l (\xi-\xi') u(\xi') d \xi'  & =   \int_{ \TT}    J^\l (s,s', r, r')    u(s',r') \a(s',r')  ds' dr'     \cr &  +
 \int_{ \TT} R_1^\l (s,s', r, r')     u(s',r') \a(s',r')  ds' dr'  
 \cr &  +   \int_{ \TT}  R_2^\l (s,s', r, r')    u(s',r') \a(s',r')  ds' dr', 
  \end {split}\end {equation}
  where
  $$ \left |  \int_{ \TT} R_1^\l (s,s', r, r')       \a(s',r')  ds' dr' \right |   \le  C\l^2, $$
  $$  \left |  \int_{ \TT}  R_2^\l (s,s', r, r')       \a(s',r') ds' dr' \right |   \le  C\l^4 .$$
\qed

 \end {document}